\documentclass[10pt]{article}

\usepackage[hidelinks]{hyperref}
\usepackage[utf8]{inputenc}

\usepackage{geometry}
\usepackage[english]{babel}
\usepackage{caption}

\usepackage[backend=biber,style=alphabetic]{biblatex}
\usepackage{mathtools}
\usepackage{amsmath}
\usepackage{amssymb}
\usepackage{mathrsfs}
\usepackage{dsfont}
\usepackage{graphicx}
\usepackage{booktabs}
\usepackage{array}
\usepackage{paralist}
\usepackage{verbatim}
\usepackage{subfig}
\usepackage{tikz}
\usepackage{tikz-cd}
\usetikzlibrary{arrows}
\tikzset{node distance=1.8cm, auto}
\usepackage{csquotes}

\usepackage{lmodern}
\usepackage{sectsty}
\usepackage{amsthm}
\theoremstyle{definition}
\allsectionsfont{\sffamily\mdseries\upshape} 
\newtheorem{theorem}{Theorem}[section]
\newtheorem{proposition}[theorem]{Proposition}
\newtheorem{corollary}[theorem]{Corollary}
\newtheorem{definition}[theorem]{Definition}

\newtheorem{example}[theorem]{Example}

\newtheorem{lemma}[theorem]{Lemma}
\newtheorem{remark}[theorem]{Remark}
\newtheorem{conjecture}[theorem]{Conjecture}
\newtheorem{algorithm}[theorem]{Algorithm}

\usepackage{rotating}
\usepackage{graphicx}

\let\phi\varphi
\let\theta\vartheta
\DeclareMathOperator\Sp{Sp}

\DeclareMathOperator\End{End}
\DeclareMathOperator\Aut{Aut}

\DeclareMathOperator\rk{rk}

\DeclareMathOperator\Tr{Tr}

\DeclareMathOperator\Frac{Frac}
\renewcommand{\:}{\colon}
\newcommand{\T}{\mathbf{T}}

\renewcommand{\S}{\operatorname{S}}
\newcommand{\M}{\operatorname{M}}
\newcommand{\Tmod}{\widetilde{\T}}
\newcommand{\Smod}{\widetilde{\S}}

\newcommand{\bF}{\mathbb{F}}

\title{\textsc{Traces of Hecke operators on Drinfeld modular forms for \texorpdfstring{$\operatorname{GL}_2(\bF_q[T])$}{GL2(Fq[T])}}} 
\author{Sjoerd de Vries \\ \\ \emph{With an appendix joint with Jonas Bergstr\"om}}
\date{} 

\begin{document}
\maketitle

\begin{abstract}
    \noindent In this paper, we study traces of Hecke operators on Drinfeld modular forms of level~1 in the case $A = \mathbb{F}_q[T]$. We deduce closed-form expressions for traces of Hecke operators corresponding to primes of degree at most~2 and provide algorithms for primes of higher degree. We improve the Ramanujan bound and deduce the decomposition of cusp forms of level~$\Gamma_0(\mathfrak{p})$ into oldforms and newforms, as conjectured by Bandini--Valentino, under the hypothesis that each Hecke eigenvalue has multiplicity less than~$p$.
\end{abstract}

\tableofcontents

\section{Introduction}

Drinfeld modular forms are function field analogues of elliptic modular forms. At the heart of this analogy lies the one between Drinfeld modules of rank~2 and elliptic curves, both of which can be realised as rank~2 lattices in the algebraic closure of the completion of a global field at infinity. To fix ideas, let $\bF_q$ be a finite field of order~$q$ and characteristic~$p$, let $K$ be the function field of a smooth projective curve over~$\mathbb F_q$, let $\infty$ be a place of~$K$, and let $A$ be the ring of integers in~$K$ (the elements of~$K$ with poles only at~$\infty$). Drinfeld modular forms are then defined as certain functions on the moduli space of Drinfeld $A$-modules. 

The theory of Drinfeld modular forms bears, as expected, many similarities to the classical theory of modular forms. For instance, every modular form has a weight, and the space of modular forms of fixed weight is finite-dimensional. However, there are also notable differences. Let us mention two examples. First, Drinfeld modular forms 
come equipped with a type, which can be seen as an element of $\mathbb Z/(q-1)\mathbb Z$ and which is absent in the classical setting. Second, since everything happens in positive characteristic, there is no obvious analogue of a Petersson inner product, and as such there is also no straightforward analogue of notions derived from it. Chief among these is the decomposition of modular forms into oldforms and newforms. There is recent work initiated by Bandini--Valentino which aims to achieve a decomposition into oldforms and newforms via non-classical methods \cite{band-val_oldnew,val_atkin-lehner-theory}.

In both settings, Hecke operators play a central role in the study of modular forms. Again, one finds striking similarities and differences. In both settings, there is a Hecke operator for each prime, and each Hecke operator respects the modular forms of a given weight. On the other hand, eigenvalues of Hecke operators classically arise as coefficients of $q$-expansions of modular forms; in the Drinfeld setting, there is no such relationship, except for some special modular forms (those with $A$-expansions). Even more jarringly, classical Hecke operators are always diagonalisable (they are orthogonal with respect to the Petersson inner product), but for Drinfeld modular forms this is not always the case \cite{li-meemark}.

This mixture between familiarity and strangeness makes Drinfeld modular forms interesting objects of study, and much is yet to be understood, even in the simplest case $A = \mathbb F_q[T]$. In this paper we study the Hecke operators~$\T_{\mathfrak p}$, when $A = \mathbb F_q[T]$, using a trace formula proved in~\cite{devries_rambound}. The trace formula is geometric, in the sense that it expresses the trace of a Hecke operator as a sum over points on the moduli space of Drinfeld modules. The number of isomorphism classes of Drinfeld modules over a finite field in a given isogeny class has been studied by Yu and Gekeler \cite{yu_isogs,gekeler_frob_dist_dm}, allowing for an explicit evaluation of the trace formula in some cases.\\

We now state our main results. Denote by $\S_{k,l}$ the space of Drinfeld cusp forms of weight~$k$ and type~$l$, so that $\S_{k,l} \neq 0$ only if $k \equiv 2l \pmod{q-1}$. Write $\wp$ for the monic generator of~$\mathfrak p \trianglelefteq A$. Firstly, we obtain a closed-form expression for traces of Hecke operators for primes of degree~1.

\begin{theorem}[Thm.~\ref{thm:deg1}]\label{thm:intro1}
    Fix $k \geq 0$ and $l \in \mathbb Z$ such that $k+2 \equiv 2l \pmod{q-1}$. Then we have
    \[
    \Tr(\T_{T} \hspace{0.2em} | \hspace{0.1em} \S_{k+2,l}) = \sum_{\substack{0 \leq j < k/2 \\ j \equiv l-1 \pmod{q-1}}} (-1)^j {\binom{k-j}{j}}T^j.
    \]
\end{theorem}

Theorem~\ref{thm:intro1} allows for the computation of the $\T_T$-eigenvalues on all 1-dimensional spaces of cusp forms (Thm.~\ref{thm:deg1dim1}).

In the appendix, 
we completely describe the number of isomorphism classes of Drinfeld modules over finite fields in a fixed isogeny class in terms of Hurwitz class numbers, extending the aforementioned results by Yu and Gekeler. The main application of this is the following: if $q$ is even, then Theorem~\ref{thm:intro1} models the traces of every Hecke operator.

\begin{theorem}[Thm.~\ref{thm:char2traces}]\label{thm:intro_char2traces}
Suppose $2 \mid q$. Fix $k \geq 0$ and $l \in \mathbb Z$ such that $k+2 \equiv 2l \pmod{q-1}$. Then for any Hecke operator $\T_{\mathfrak p}$ and any $n \geq 1$, we have
\[
    \Tr(\T_{\mathfrak p}^n \hspace{0.2em} | \hspace{0.1em} \S_{k+2,l})  = \sum_{\substack{0 \leq j < k/2 \\ j \equiv l-1 \pmod{q-1}}} {\binom{k-j}{j}}\wp^{nj}.
\]
\end{theorem}

Since Theorem~\ref{thm:intro_char2traces} is valid for any $n \geq 1$, we can draw conclusions about the Hecke eigenvalues.

\begin{theorem}[Thm.~\ref{thm:char2eigs} and Thm.~\ref{thm:char2_repetition}]
    Suppose $2 \mid q$. Then $\lambda$ is an eigenvalue of $\T_{\mathfrak p}$ on~$\S_{k+2,l}$ with odd algebraic multiplicity if and only if $\lambda = \wp^j$ for some $j \in \mathbb Z$ satisfying $0 \leq j < k/2$,  $j \equiv l-1 \pmod{q-1}$, and $\binom{k-j}{j} \equiv 1 \pmod{2}$.
    Moreover, there are only finitely many weights in which the action of~$\T_{\mathfrak p}$ has no repeated eigenvalues.
\end{theorem}

If $q$ is odd, the traces of Hecke operators are more complicated for primes of higher degree; see for instance Theorem~\ref{thm:deg2} for primes of degree~2. Nonetheless, the traces can be effectively computed, as described in Algorithm~\ref{alg}.

Generally speaking, it is a difficult task to determine whether a given Hecke operator is injective. Bandini and Valentino conjectured \cite{band-val_slopes} that the Hecke operator~$\T_T$ is always injective on~$\S_{k,l}$. We note that a result by Petrov and Joshi \cite{joshi-petrov} in fact implies this result for any Hecke operator on~$\S_{k,l}$.

\begin{theorem}[Thm.~\ref{thm:injectivity}]\label{thm:intro3}
    The Hecke operator $\T_{\mathfrak p}$ is injective on $\bigoplus_{k,l} \S_{k,l}$ for any prime~$\mathfrak p$.
\end{theorem}

Further conjectures of Bandini and Valentino aim to establish a decomposition of Drinfeld cusp forms of level~$\Gamma_0(\mathfrak p)$ into oldforms and newforms. This is related to bounds on slopes of Hecke operators on Drinfeld cusp forms of level~1. We study analogous bounds on traces of Hecke operators in Section~\ref{sec:ram}. We have a precise guess for a sharp bound in this case (Conj.~\ref{conj:strong_ram_traces}), which we can prove in some special cases (Thm.~\ref{thm:strong_ram_known}). 

Although we are unable to prove Conjecture~\ref{conj:strong_ram_traces} in general, let alone the stronger version for slopes (Conj.~\ref{conj:strong_ram}), we are able to improve the Ramanujan bound from~\cite{devries_rambound} (see Prop.~\ref{prop:ram_bd_strict}). This already has implications for the decomposition of cusp forms into oldforms and newforms.

\begin{theorem}[Cor.~\ref{cor:oldnew}]\label{thm:intro2}
If $\dim \S_{k,l} < p$, then $\S_{k,l}(\Gamma_0(\mathfrak p)) = \S^\text{new}_{k,l}(\Gamma_0(\mathfrak p)) \oplus \S^\text{old}_{k,l}(\Gamma_0(\mathfrak p))$.
\end{theorem}

Using the trace formula, we also prove the following theorem, which addresses an open question about Drinfeld modular forms with $A$-expansions. Namely, given an eigenform $f \in \S_{k,l}$ and an integer $n \geq 1$ such that $\T_{\mathfrak p}f = \wp^{n-1}f$ for all~$\mathfrak p$, it does not necessarily follow that $f$ has an $A$-expansion with $A$-exponent~$n$.

\begin{theorem}[Thm.~\ref{thm:trivial_eigs}]\label{thm:intro_Aexp} 
For any $\mathfrak p \trianglelefteq \bF_3[T]$, the $\T_{\mathfrak p}$-eigenvalue of $g^2h^2 \in \S_{12,0}$ equals~$\wp^3$.
\end{theorem}

Many of our results have been implemented in \textsc{Magma}. The code can be found at~\url{https://github.com/Sjoerd-deVries/DMF_Trace_Formula.git}.

This paper grew out of the author's licentiate thesis \cite{diva}.  

\subsection*{Outline of the paper}
In Section~\ref{sec:intro}, we recall the basic theory of Drinfeld modules and Drinfeld modular forms, including the necessary background on $A$-expansions.

In Section~\ref{sec:traces}, we begin our study of the trace formula. One of the first observations is a certain symmetry in weights of the form~$k = p^m+1$ for $m \in \mathbb Z_{\geq 1}$ which turns out to be very useful. We prove Theorems~\ref{thm:intro1} and~\ref{thm:intro_char2traces} and deduce several consequences. We obtain some results on traces of Hecke operators for primes of higher degree and describe an algorithm to compute these. We also obtain results on the traces of $\T_{\mathfrak p}$ modulo~$\mathfrak p$.

In Section~\ref{sec:ram}, we study the Ramanujan bound from~\cite{devries_rambound}. We prove that it is not sharp at level~1 and propose a stronger version. We prove the strong Ramanujan bound in some cases and give a sufficient condition for it to hold in general.

In Section~\ref{sec:slopes}, we explore to what extent our results give information about the Hecke eigenvalues. In characteristic~2, we reach a satisfactory answer: we can explicitly determine all eigenvalues which occur with odd multiplicity, and we get no information whatsoever about the ones which occur with even multiplicity. We also prove Theorem~\ref{thm:intro3}, and discuss the relationship between traces and slopes. In particular, we deduce the decomposition into oldforms and newforms at level $\Gamma_0(\mathfrak p)$ under the condition that no Hecke eigenvalue is repeated $p$~times. We also formulate several conjectures based on our computations.

Section~\ref{sec:comp} is of a more computational nature. We mostly focus on 1-dimensional spaces of cusp forms. A computation in weight~$12$ for $q=3$ yields the aforementioned result on $A$-expansions. We explicitly compute the $\T_T$-eigenvalues occurring for 1-dimensional spaces, and show that there exist quadratic $\T_T$-eigenvalues on some 2-dimensional space for any odd $q$. We also briefly consider the action of $\T_{\mathfrak p}$ when $\mathfrak p$ is a prime of degree~2. 

In Appendix~\ref{sec:appendix}, we make explicit the connection between Hurwitz class numbers, points on Jacobians of hyperelliptic curves, and isogeny classes of Drinfeld modules over finite fields.

\subsection*{Notation and conventions}
Throughout, $p$ denotes a prime number and $q$ denotes a power of~$p$. We write $A = \bF_q[T]$ and $K = \text{Frac}(A) = \bF_q(T)$. The symbol~$\mathfrak p$ always denotes a non-zero prime ideal of~$A$. Its unique monic generator is denoted by~$\wp$. We write $\bF_{\mathfrak p}$ for the finite field $A/\mathfrak p$, and $\bF_{\mathfrak p^n}$ for the unique degree~$n$ extension of~$\bF_{\mathfrak p}$. Every polynomial $a \in A$ has a degree; we work under the convention that $\deg(0) = -\infty$. We use the shorthand $x \equiv_n y$ to denote the congruence $x \equiv y \pmod{n}$.

\subsection*{Acknowledgements}
I would like to thank my supervisor Jonas Bergstr\"om for his indispensable encouragement and countless conversations on the topic of pursuit. I would also like to thank my co-supervisor Olof Bergvall for his constructive comments, and Shin Hattori for pointing out an error in an earlier version.

\section{Background}\label{sec:intro}

\subsection{Drinfeld modular forms for \texorpdfstring{$\mathbb F_q[T]$}{Fq[T]}}
Drinfeld modular forms are function field analogues of modular forms, originally defined by Goss in his PhD thesis \cite{goss_pi-adic} and further studied by himself and Gekeler \cite{goss_modf,gekeler_modcurves,gekeler_coeffs}; all the material of this subsection can be found in more detail there. Let $A = \mathbb F_q[T]$ and $K = \Frac(A)$. Let $K_\infty = \mathbb F_q(\!(T^{-1})\!)$ be the completion of~$K$ at the place~$\infty$, and let $\mathbb C_\infty$ be the completion of an algebraic closure of~$K_\infty$. Denote by $\Omega := \mathbb C_\infty \setminus K_\infty$ the Drinfeld upper half-plane, seen as a rigid-analytic space. It has an action of $\operatorname{GL}_2(K_\infty)$ by M\"obius transformations.

\begin{definition}
    Fix $k \in \mathbb Z$ and $l \in \mathbb Z$. A \emph{Drinfeld modular form of weight~$k$ and type~$l$} is a function $f\: \Omega \to \mathbb C_\infty$ satisfying the following properties:
    \begin{enumerate}
        \item For each $\gamma = \begin{pmatrix} a & b \\ c & d \end{pmatrix} \in \operatorname{GL}_2(A)$, we have
        \begin{equation}\label{eq:transf}
        f(\gamma z) = \det(\gamma)^{-l} (cz + d)^k f(z).
        \end{equation}
        \item $f$ is holomorphic on $\Omega$ and at infinity.
    \end{enumerate}
    If additionally $f$ vanishes at infinity, it is called a \emph{cusp form}. A cusp form which vanishes to order at least 2 at infinity is called a \emph{double cusp form}. Denote the $\mathbb C_\infty$-vector space of Drinfeld modular forms of weight $k$ and type $l$ by $\M_{k,l}$, the subspace of cusp forms by $\S_{k,l}$, and the subspace of double cusp forms by $\S^2_{k,l}$.
\end{definition}

Note that $\M_{k,l} = \M_{k,l'}$ whenever $l \equiv l' \pmod{q-1}$, since $\det(\operatorname{GL}_2(\bF_q[T])) = \bF_q^\times$. Moreover, $\M_{k,l} = 0$ unless $k \equiv 2l \pmod{q-1}$.


\begin{remark}Classically, holomorphicity of a modular form at infinity implies that its $q$-series is a power series. For Drinfeld modular forms, an analogous property holds, but instead of a $q$-series one has a so-called $t$-expansion; holomorphicity at infinity is then equivalent to the $t$-expansion being a power series, and being a cusp form is equivalent to the constant term being~0. For details, see \cite[Section~5]{gekeler_coeffs}.
\end{remark}

Multiplication induces maps
\[
\M_{k,l}\times \M_{k',l'} \longrightarrow \M_{k+k',l+l'}.
\]
This turns the space of all Drinfeld modular forms into a doubly graded algebra, which we denote by $\M := \bigoplus_{k,l} \M_{k,l}$. The doubly graded ideal of cusp forms is denoted $\S := \bigoplus_{k,l} \S_{k,l}$.

When $k > 0$ is a multiple of $q-1$, there exists a non-zero Drinfeld modular form of weight~$k$, non-vanishing at the cusp, called an Eisenstein series. We denote this modular form by $g_k \in \M_{k,0}$, and write $g := g_{q-1}$. The first cusp form of type zero is denoted by $\Delta \in \S_{q^2-1,0}$. Moreover, there exists a modular form $h \in \S_{q+1,1}$ such that $h^{q-1} = -\Delta$. The following theorem shows that the algebra of Drinfeld modular forms has a very simple structure \cite[Thm.~5.13]{gekeler_coeffs}.

\begin{theorem}\label{thm:modforms}
    We have an isomorphism of doubly graded algebras
    \[
    \M = \mathbb C_{\infty}[g,h].
    \]
\end{theorem}

In particular, the ideal $\S$ of cusp forms is the principal ideal generated by~$h$, and similarly $\S^2$ is generated by $h^2$. Theorem~\ref{thm:modforms} implies that each $\M_{k,l}$ is finite-dimensional and yields the following dimension formulae for the spaces of (double) cusp forms.

\begin{lemma}\label{lem:cuspdim}
    Let $k \in \mathbb{Z}_{\geq 0}$ and $1 \leq l \leq q-1$. Then we have
    \[
    \dim \S_{k,l} = \begin{cases}
        \left\lfloor \frac{k+(q-1-l)(q+1)}{q^2 - 1} \right\rfloor & \text{if }  k \equiv 2l \pmod{q-1}; \\ 0 & \text{otherwise.}
    \end{cases}
    \]
Moreover, the dimension of the space of double cusp forms is
    \[
    \dim \S^2_{k,l} = \begin{cases} 
    \dim \S_{k,l} - 1 & \text{if } l=1 \text{ and } \S_{k,l} \neq 0; \\
    \dim \S_{k,l} & \text{otherwise.}
    \end{cases}
    \]
\end{lemma}

One can also define Drinfeld modular forms of higher level. For a non-zero ideal $\mathfrak n \trianglelefteq A$, define
\[
\Gamma(\mathfrak n) = \left \{ 
    M \in \operatorname{GL}_2(A) \ \bigg{|} \ M \equiv \begin{pmatrix} 1 & 0 \\ 0 & 1 
\end{pmatrix} \pmod {\mathfrak{n}} \right \}.
\]

\begin{definition}
    A subgroup $\Gamma \subseteq \operatorname{GL}_2(A)$ is called a \emph{congruence subgroup} if $\Gamma(\mathfrak n) \subseteq \Gamma$ for some ideal $\mathfrak n$.
\end{definition}

In particular, 
\[
\Gamma_0(\mathfrak n) := \left\{ M \in \operatorname{GL}_2(A) \ \bigg{|} \ M \equiv \begin{pmatrix} * & * \\ 0 & * \end{pmatrix} \pmod{\mathfrak n} \right\}
\]
is a congruence subgroup for any $\mathfrak n$.

Let $\Gamma \subseteq \operatorname{GL}_2(A)$ be a congruence subgroup. A \emph{Drinfeld modular form of level~$\Gamma$} is a rigid-analytic function on $\Omega$, holomorphic at the cusps of~$\Gamma$ \cite[Sec.~V.2]{gekeler_modcurves}, which satisfies the transformation property~\eqref{eq:transf} for all $\gamma \in \Gamma$. By a modular form of level 1, we mean a modular form for $\operatorname{GL}_2(A)$.

The space of modular forms of weight~$k$, type~$l$, and level~$\Gamma$ is denoted by~$\M_{k,l}(\Gamma)$, and the subspace of cusp forms by~$\S_{k,l}(\Gamma)$. The set of all Drinfeld modular forms of level~$\Gamma$ is again a doubly graded algebra.

Finally, we fix the notation 
\[
\S_{k} := \bigoplus_{l=1}^{q-1} \S_{k,l}.
\]
It then follows that $\S_k = \S_k(\operatorname{SL}_2(A))$ \cite[Thm.~17.6]{BBP}.

\subsection{Hecke operators}
As in the classical case, Drinfeld modular forms admit actions of Hecke operators. 

\begin{definition}\label{def:hecke_operators}
    Let $\wp \in A$ be monic and irreducible, and write $\mathfrak p = (\wp)$ for the maximal ideal generated by~$\wp$. Define the \emph{Hecke operator associated to $\mathfrak p$} to be the linear map $\T_{\mathfrak p}\: \M_{k,l} \to \M_{k,l}$ given by
    \[
    (\T_{\mathfrak p}f)(z) := \wp^{k-1} f(\wp z) + \wp^{-1} \sum_{\deg(b) < \deg(\wp)} f\left( \frac{z+b}{\wp} \right).
    \]
    The subalgebra of $\End_{\mathbb C_\infty}(\M)$ generated by all Hecke operators is called the \emph{Hecke algebra}. A non-zero modular form $f \in \M$ is called an \emph{eigenform} if $\T_{\mathfrak p}$ acts as a scalar on~$f$ for every $\mathfrak p \trianglelefteq A$.
\end{definition}

\begin{remark}\hfill
\begin{enumerate}
    \item If $\wp$ is a generator of~$\mathfrak p$, we sometimes also write $\T_{\wp}$ for the Hecke operator associated to~$\mathfrak p$.
    \item Our definition of $\T_{\mathfrak p}$ is a rescaled version of the Hecke operator $\T_{\mathfrak p}^{\mathbb F_q[T]}$ defined in \cite{goss_pi-adic,gekeler_coeffs}. More precisely, 
    we have
    \[
    \T_{\mathfrak p} := \wp^{-1}\T^{\mathbb F_q[T]}_{\mathfrak p}.
    \]
    This normalisation is made purely to simplify computations and should not cause confusion. Note that the same normalisation is used in e.g.~\cite{hattori_gouv,nicole-rosso}.
\item The Hecke operators are linear endomorphisms of $\M$ but do not preserve the ring structure in general. In particular, the Hecke eigenvalues of the generators $h$ and $g$ a~priori say nothing about the Hecke eigenvalues of other modular forms.
\item
    One can extend the definition of Hecke operators from prime ideals to arbitrary ideals. In doing so, one obtains the relation $\T_{\mathfrak n }\T_{\mathfrak n'} = \T_{\mathfrak n \mathfrak n'}$ for all maximal ideals $\mathfrak n, \mathfrak n' \trianglelefteq A$. In the classical setting, this relation only holds when $\mathfrak n, \mathfrak n' \trianglelefteq \mathbb Z$ are coprime. This difference can be explained by the fact that the Galois representation associated to a cuspidal eigenform is one-dimensional, as opposed to two-dimensional in the classical setting \cite[Sec.~14]{bockle}.
\end{enumerate}
\end{remark}

The Hecke operators preserve the spaces $\M_{k,l}$ and $\S_{k,l}$ for any $k$ and $l$. Moreover, they preserve the space $\S^2_{k,l}$ of double cusp forms,
a feature unique to the function field setting.

\subsection{\texorpdfstring{$A$}{A}-expansions}
Every Drinfeld modular form has a $t$-expansion, analogous to the $q$-series of elliptic modular forms. The $t$-expansion of $f \in \M$ is an expansion of the form
\[
f = \sum_{n = 0}^\infty a_nt^n, \qquad a_n \in \mathbb C_\infty,
\]
where $t = t(z)$ is a parameter at infinity. In the spirit of function field arithmetic, one would like to replace the sum over $\mathbb Z$ by a sum over $A$. In \cite{petrov_a-exp}, Petrov showed that this can be done for certain Drinfeld modular forms~$f$, but that such modular forms are very special. We recall some key points of the theory of $A$-expansions here.

Denote by $A_+$ the set of monic polynomials in $A$. For $a \in A_+$, write $t_a := t(az)$. For $n \geq 1$, denote by $G_n(X)$ the $n$-th Goss polynomial (suitably normalised; see~\cite{petrov_a-exp}).

\begin{definition}
    Let $f \in \S$ be a Drinfeld cusp form. Then $f$ has an \emph{$A$-expansion} if there exists an integer $n \geq 1$ and constants $c_a \in \mathbb C_\infty$ such that
    \[
    f = \sum_{a \in A_+} c_a G_n(t_a).
    \]
    In this case, $n$ is called an \emph{A-exponent} of~$f$.
\end{definition}

If $f \in \S_{k,l}$ has an $A$-expansion with $A$-exponent~$n$, then necessarily $n \equiv l \pmod{q-1}$. If $f$ is an eigenform with $A$-expansion, then its $A$-exponent is uniquely determined, but in general it is not known whether a non-zero cusp form can have $A$-expansions with different $A$-exponents.

The main theorem of \cite{petrov_a-exp} gives infinitely many examples of eigenforms with $A$-expansions; moreover, all known examples of modular forms with $A$-expansions can be obtained by applying Petrov's theorem. We highlight the following cases.

\begin{proposition}\label{prop:a-exp_examples}
    Let $q = p^r$. The following cusp forms are eigenforms with $A$-expansions.
    \begin{enumerate}
        \item The forms $g^nh^q$ for $0 \leq n \leq q-1$, with $A$-exponent~$(n+1)(q-1)+1$.
        \item The forms $g^{ap^s}h^l$ for $1 \leq s \leq r$, $0 \leq a \leq p^{r-s}$, and $1 \leq l \leq p^s$, with $A$-exponent~$l$.
    \end{enumerate}
    Moreover, for all $k$ such that $\S_{k,1} \neq 0$, the space $S_{k,1} / S^2_{k,1}$ is spanned by an eigenform with $A$-expansion and $A$-exponent~1.
\end{proposition}

\begin{proof}
    The statement about the quotient of cusp forms by double cusp forms is \cite[Thm.~3.2]{petrov_a-exp}, which in particular proves the claim about the forms $g^{ap^s}h$. In the other cases, note that each cusp form is the unique (up to scalar multiplication) double cusp form of the given weight and type. Hence it suffices to show that there is a non-zero doubly cuspidal eigenform with $A$-expansion and the claimed $A$-exponent in that weight and type. In each case, this follows from \cite[Thm.~1.3]{petrov_a-exp}.
\end{proof}

The main reason why eigenforms with $A$-expansions are relevant for our purposes is that their Hecke eigensystems are particularly simple.

\begin{theorem}\label{thm:a-exp_eigs}
    Let $f$ be an eigenform with $A$-expansion and $A$-exponent~$n$. Then
    \[
    \T_{\mathfrak p} f = \wp^{n-1}f \quad \text{for any } \mathfrak p \trianglelefteq A.
    \]
\end{theorem}

The Hecke eigensystems arising from eigenforms with $A$-expansions are called \emph{power eigensystems}. In Theorem~\ref{thm:trivial_eigs}, we show that an eigenform can have a power eigensystem even though it does not have an $A$-expansion.


\subsection{Drinfeld modules over finite fields}\label{sec:dm/ff}
Recall that $A = \mathbb F_q[T]$. We give a quick introduction to Drinfeld $A$-modules over finite fields in order to understand the terms involved in the trace formula from \cite{devries_rambound}. 
We also recall some results from \cite{yu_isogs} about isogeny classes of Drinfeld modules of rank~2. 
For proofs of the results in this section, see \cite[Chapter 4]{papikian}.

\begin{definition}
    Let $F$ be a field containing~$\mathbb F_q$. The \emph{ring of additive polynomials over~$F$}, denoted $F\{\tau\}$, is defined as the non-commutative polynomial ring over~$F$ in the variable~$\tau$ satisfying $\tau x = x^q \tau$ for all $x \in F$.
\end{definition}

\begin{definition}
    Let $r \geq 1$ be an integer and let $\bF_q \subseteq F \subseteq \overline{\mathbb F}_q$ be a field. A \emph{Drinfeld $A$-module of rank $r$ over $F$} is an $\bF_q$-algebra homomorphism $\phi \colon A \to F\{\tau\}$ such that 
    \[
    \phi_T := \phi(T) =  \alpha_0 + \alpha_1\tau + \ldots + \alpha_r \tau^r, \quad \alpha_0,\ldots,\alpha_r \in F, \quad \alpha_r \neq 0.
    \]
    Given a Drinfeld module $\phi$ with $\phi_T$ as above, the \emph{characteristic} of~$\phi$ is the prime ideal of~$A$ generated by the minimal polynomial of~$\alpha_0$ over~$\bF_q$. A \emph{morphism of Drinfeld modules $\phi \to \psi$} is an element $f \in F\{\tau\}$ such that $f \phi_T = \psi_Tf$. We say that $\phi$ is \emph{isogenous} to $\psi$ if there exists a non-zero morphism $\phi \to \psi$. The ring of endomorphisms of~$\phi$ is denoted by $\End(\phi) \subseteq F\{\tau\}$.
\end{definition}

\begin{remark}\hfill
\begin{enumerate}
\item Clearly, $\phi$ is determined by $\phi_T$.
\item It follows from the definition that if $\phi$ is a Drinfeld module with characteristic~$\mathfrak p$, then $F$ is an extension of $\mathbb F_{\mathfrak p} = A/\mathfrak p$. We tacitly assume that a Drinfeld module over~$\mathbb F_{{\mathfrak p}^n}$ has characteristic~$\mathfrak p$. 
\item Isogeny is an equivalence relation.
\end{enumerate}
\end{remark}

\begin{definition}
    Suppose $F = \mathbb F_{q^m}$ is finite, and let $\phi$ be a Drinfeld module over~$F$. The \emph{Frobenius endomorphism} of $\phi$ is defined to be $\pi_{\phi} := \tau^m \in F\{\tau\}$.
\end{definition}

Since $\tau^m$ lies in the center of $F\{\tau\}$, the Frobenius endomorphism is indeed an endomorphism of~$\phi$. To understand its importance, we consider $\End(\phi)$ as a finitely generated free $A$-module via the map~$\phi$. Tensoring with $K = \Frac(A)$ yields the $K$-algebra $\End^0(\phi) := K \otimes_A \End(\phi)$.

\begin{proposition}\label{prop:endalg}
Let $\phi$ be a Drinfeld $A$-module of rank $r$ over a finite field.
    \begin{enumerate}
        \item $\End^0(\phi)$ is a division algebra, which depends up to isomorphism only on the isogeny class of~$\phi$.
        \item The center $Z(\End^0(\phi))$ of $\End^0(\phi)$ equals $K(\pi_{\phi})$.
        \item The reduced degree of $\End^0(\phi)$ equals
        \[
        [\End^0(\phi):K]^\text{red} := [\End^0(\phi):K(\pi_\phi)]^{1/2}[K(\pi_\phi):K] = r.
        \]
    \end{enumerate}
\end{proposition}

In particular, we see that $r \leq \text{rk}_A\End(\phi) \leq r^2$. Moreover, the minimal polynomial of $\pi_\phi$ has degree $[K(\pi_\phi):K]$, which divides $r$. Hence the following definition makes sense.

\begin{definition}
The \emph{characteristic polynomial of Frobenius} is the unique monic polynomial $c_\phi(X) \in A[X]$ of degree $r$ which is a power of the minimal polynomial of $\pi_{\phi} \in \End(\phi)$. A polynomial $f(X) \in A[X]$ is called a \emph{Weil polynomial of rank~$r$ over~$F$} if there exists a Drinfeld module~$\phi$ of rank~$r$ over~$F$ such that $f(X) = c_\phi(X)$.
\end{definition}

Weil polynomials have remarkable properties \cite[Thm.~4.2.7]{papikian}.

\begin{proposition}\label{prop:charpol_properties}
    Let $\phi$ and $\psi$ be Drinfeld modules of rank~$r$ over~$\mathbb F_{q^m}$ with characteristic~$\mathfrak p$. Consider the characteristic polynomial of Frobenius
    \[
    c_\phi(X) = X^r + a_{r-1}X^{r-1} + \ldots + a_1X +  a_0 \in A[X].
    \]
    Then $c_{\phi}(X) = c_{\psi}(X)$ if and only if $\phi$ is isogenous to~$\psi$. Moreover, the following properties hold:
    \begin{enumerate}
        \item (Riemann Hypothesis) Any root $\pi \in \bar{K}$ of $c_{\phi}(X)$ satisfies $|\pi|_{\infty} = q^{m/r}$.
        \item For each $0 \leq i < r$, we have
        \[
        \deg(a_i) \leq \frac{(r-i)m}{r}.
        \]
        \item There exists some $\lambda \in \mathbb F_q^\times$ such that $a_0 = \lambda \wp^{m/\deg(\wp)}$.
    \end{enumerate}
\end{proposition}

Let us now consider the case of rank $r=2$. Then for any Drinfeld module $\phi$ over a finite field, we have
\[
c_{\phi}(X) = X^2 + a_1X + a_0 = (X - \pi_{\phi})(X-\bar{\pi}_\phi),
\]
where $\bar{\pi}_{\phi}$ from now on denotes the Galois  conjugate of~$\pi_\phi$.
For $k \in \mathbb Z_{\geq 0}$, denote by $h_k \in \mathbb Z[X_1,X_2]$ the $k$-th homogeneous symmetric polynomial in two variables. Concretely,
\[
h_k(X_1,X_2) = \sum_{i=0}^k X_1^i X_2^{k-i}.
\]
We can now state the trace formula \cite[Thm.~4.10]{devries_rambound}.

\begin{theorem}\label{thm:heckesum}
    For any $n \geq 1$, $k \geq 0$, and $l \in \mathbb Z$, we have
    \begin{equation}\label{eq:trace_formula}
    \Tr(\T^n_{\mathfrak p} \hspace{0.1em} | \hspace{0.1em} \S_{k+2,l}) = \sum_{[\phi] / \mathbb F_{{\mathfrak p}^n}} h_{k}(\pi_\phi,\bar{\pi}_\phi) \cdot \left(\frac{\pi_{\phi} \bar{\pi}_\phi}{\wp^n}\right)^{l-k-1},
    \end{equation}
    where the sum is over isomorphism classes of Drinfeld modules of rank~2 over $\mathbb F_{{\mathfrak p}^n}$.
\end{theorem}

A direct consequence of the trace formula is the Ramanujan bound \cite[Cor.~4.16]{devries_rambound}, which holds for arbitrary~$\mathfrak p$ and~$n$. We state it here as a definition, taking into account our chosen normalisation of Hecke operators.

\begin{definition}\label{def:ram_bd}
    The \emph{Ramanujan bound for~$\T_{\mathfrak p}^n$} states that for any integers~$k$ and~$l$,
    \[
    \deg \Tr(\T_{\mathfrak p}^n \hspace{0.2em} | \hspace{0.1em} \S_{k,l} ) \leq \frac{n\deg(\wp)(k-2)}{2}.
    \]
\end{definition}

In the remainder of this section, we recall some results on isogeny classes. By Prop.~\ref{prop:endalg}, the endomorphism algebra $\End^0(\phi)$ of a Drinfeld module only depends on its isogeny class. Moreover, given a division algebra $D_\pi$ which is the endomorphism algebra of some Drinfeld module $\phi$, one can determine the number of isomorphism classes of Drinfeld modules isogenous to $\phi$ from $D_\pi$. Indeed, any Drinfeld module $\phi'$ isogenous to $\phi$ will give rise to an $A$-order $A[\pi] \subseteq \End(E,\phi') \subset D_\pi$, which depends only on the isomorphism class of $\phi'$. In this way, counting isomorphism classes in a given isogeny class becomes a problem of counting certain $A$-ideals in division algebras.

In general, this is a difficult problem; we refer the reader to~\cite{KKP} for details. For Drinfeld modules of rank~2, however, one can completely describe the isomorphism classes in a given isogeny class (see Appendix~\ref{sec:appendix} for details). In what follows, we describe the possible Weil polynomials of Drinfeld modules over finite fields.

Starting from the characterization of Weil polynomials in~\cite[Thm.~3]{yu_isogs}, it is straightforward to generalize the classification of these in~\cite[Prop.~4]{yu_isogs} to also include characteristic~2. Recall that a quadratic extension $L/K$ of function fields is called \emph{imaginary} if there is only one place of~$L$ lying over~$\infty$.

\begin{proposition} \label{prop:yu_pols}
 The Weil polynomials of rank~2 over~$\bF_{\mathfrak{p}^n}$ are precisely the following: 
    \begin{enumerate}
        \item $c(X) = X^2 - aX + b\wp^n$ such that $a \in A$, $b \in \bF_q^\times$, $(a,\wp) = 1$, $\deg(a)\leq n \deg (\wp)/2 $, and the splitting field of~$c(X)$ over~$K$ is imaginary;
        \item if $n$ is odd: $c(X) = X^2 + b \wp^n$ such that $b \in \bF_q^\times$ and the splitting field of~$c(X)$ over~$K$ is imaginary;
        \item if $n$ is even and $\deg(\wp)$ is odd: $c(X) = X^2 - \lambda \wp^{n/2}X + b\wp^n$ where $X^2 - \lambda X + b$ is irreducible in $\mathbb F_q[X]$;
        \item if $n$ is even: $c(X) = (X - \mu \wp^{n/2})^2$ where $\mu \in \bF_q^\times$.
    \end{enumerate}
\end{proposition}

\begin{definition}\label{def:iso_sets}
Let $a \in A$ and $b \in \bF_q^\times$. For any $n \geq 1$, we define
\[
\text{Iso}_{\mathfrak p^n}(a,b) := \{ \text{Drinfeld modules } \phi \text{ of rank 2 over } \bF_{\mathfrak p^n} \ | \ c_\phi(X) = X^2 - aX + b\wp^n \} / \cong,
\]
the set of isomorphism classes of Drinfeld modules with characteristic polynomial $X^2 - aX + b\wp^n$.
\end{definition}

\begin{remark}\label{rem:iso_c}
    Let $a \in \mathbb F_q[T]$ and $b,c \in \mathbb F_q^\times$. Then
    \[
    \# \text{Iso}_{\mathfrak p^n}(a,b) = \# \text{Iso}_{\mathfrak p^n}(ca,c^2b).
    \]
    To see this, let $F := \bF_{\mathfrak p^n}$ and fix $\lambda \in F^\times$ such that $\operatorname{Nm}_{F/\mathbb F_q}(\lambda)=c$. 
    Then the ring isomorphism $m_\lambda\: F\{\tau\} \to  F\{\tau\}$ defined by $\tau \mapsto \lambda\tau$ induces a bijection $[\phi] \mapsto [m^{-1}_\lambda \circ \phi]$ between $\text{Iso}_{\mathfrak p^n}(a,b)$ and $\text{Iso}_{\mathfrak p^n}(ca,c^2b)$.
\end{remark}

\begin{remark}
    The cardinalities $\# \text{Iso}_{\mathfrak p^n}(a,b)$ are intimately related to Hurwitz class numbers. We refer the reader to Appendix~\ref{sec:appendix} for details.
\end{remark}
\section{Traces}\label{sec:traces}

\subsection{Rewriting the trace formula}

Suppose that $\pi$ is the Frobenius endomorphism of some Drinfeld module $\phi$ over $\bF_{\mathfrak p^n}$, where $\mathfrak p = (\wp)$ for some monic irreducible polynomial~$\wp$ of degree~$d$. Then we have seen that the characteristic polynomial is of the form
\[
c_\phi(X) = (X-\pi)(X-\bar{\pi}) = X^2-aX+b\wp^n,
\]
where $b \in \mathbb F_q^\times$ and $a \in \mathbb F_q[T]$ has degree at most $nd/2$. It will be convenient to rewrite the trace formula~\eqref{eq:trace_formula} in terms of $a$, $b$, and~$\wp$.

For an integer $m \geq 0$, we let $e_m,p_m,h_m \in \mathbb Z[X_1,X_2]$ denote the elementary symmetric, resp.\ power sum, resp.\ homogeneous symmetric polynomials of degree~$m$ in~2 variables. In particular, $e_0 = h_0 = 1$ and $p_0 = 2$.

\begin{lemma}\label{lem:pmexpression}
For any $m \geq 1$, we have (under the convention that $0^0 = 1$):
\[
p_m(\pi,\bar{\pi}) = \pi^m + \bar{\pi}^m = m \sum_{r_1 + 2r_2 = m} (-1)^{r_2} \frac{(r_1+ r_2 - 1)!}{r_1!r_2!}a^{r_1}(b\wp^n)^{r_2}.
\]
\end{lemma}

\begin{proof}
One can express the power polynomials in terms of the elementary symmetric polynomials over $\mathbb Z$ using Newton's identities:
\[
p_m = (-1)^m m \sum_{r_1 + 2r_2 + \ldots + mr_m = m} \frac{(r_1+r_2 + \ldots + r_m - 1)!}{r_1!r_2!\ldots r_m!} \prod_{i=1}^m (-e_i)^{r_i}.
\]
The result follows because $e_1(\pi,\bar{\pi}) = a$, $e_2(\pi,\bar{\pi}) = b\wp^n$, and $e_m(\pi,\bar{\pi}) = 0$ for $m > 2$. The convention $0^0 = 1$ has to be adopted since $e_i^0$ is the constant polynomial $1$, which evaluates to~1 even if the input is~0.
\end{proof}

To get the trace formula into its desired form, we will need a binomial identity. We believe this identity to be known, but include a proof for lack of reference.

\begin{lemma}\label{binomlemma}
    Let $k,j \in \mathbb Z$ with $k > 2j \geq 0$. Then we have
    \[
    \sum_{i = 0}^{j} (-1)^{i}\frac{k-2i}{k-j-i}\binom{k-j-i}{j-i} = \binom{k-j}{j}.
    \]
\end{lemma}

\begin{proof}
    We will prove more generally that for integers $j \leq n$, we have
    \begin{equation}\label{binomidentity}
    \binom{n}{j} = \sum_{i=0}^j (-1)^i \frac{n+j-2i}{n-i} \binom{n-i}{j-i}.
    \end{equation}
    For this to make sense, we take on the convention that the term for $i = j = n$ equals $(-1)^j$.

    We use induction. Our base case consists of the formula for $\binom{n}{0}$ and $\binom{n}{n}$ for all $n \in \mathbb N$; one can check directly that these indeed equal 1. Now suppose we know that the formula holds for all $m < n$. Let $j < n$. Note that the term for $i=j$ in (\ref{binomidentity}) contributes $(-1)^j$. By induction, we obtain
    \begin{align*}
\binom{n}{j} - (-1)^j &= \binom{n-1}{j-1} + \left( \binom{n-1}{j} - (-1)^j \right) \\
        &= \sum_{i=0}^{j-1}  (-1)^i \left( \frac{n+j-2(i+1)}{n-1-i} \binom{n-1-i}{j-1-i} + \frac{n-1+j-2i}{n-1-i}{\binom{n-1-i}{j-i}} \right)   \\
        &= \sum_{i=0}^{j-1} (-1)^i  \frac{n+j-2i}{n-i} {\binom{n-i}{j-i}}.
    \end{align*}
    The last equality above follows from the fact that $\binom{x}{y} + \binom{x}{y+1} = \binom{x+1}{y+1}$ and the equality
    \[
        (n+j-2(i+1))(j-i) + (n-1+j-2i)(n-j) = (n-1-i)(n+j-2i),
    \]
    which is true for all $i,j,n \in \mathbb N$.
\end{proof}

\begin{lemma}\label{lem:hmexpression}
Let $k \in \mathbb N$. Set $\epsilon_k = 0$ if $k$ is odd, and $\epsilon_k = (-b\wp^n)^{k/2}$ if $k$ is even. Then we have (under the convention that $0^0 = 1$):
\begin{equation}\label{eq:hmexpression}
h_k(\pi,\bar{\pi}) = \epsilon_k + \sum_{j=0}^{\lceil k/2 \rceil - 1} (-1)^{j}\binom{k-j}{j} a^{k-2j} (b\wp^n)^{j}.
\end{equation}
\end{lemma}

\begin{proof}
Clearly $h_0(\pi,\bar{\pi}) = 1$ and $h_1(\pi,\bar{\pi}) = \pi + \bar{\pi} = a$. For $k \geq 2$, we have $h_k(\pi,\bar{\pi}) = p_k(\pi,\bar{\pi}) + b\wp^nh_{k-2}(\pi,\bar{\pi})$,
so we see inductively that
\[
h_k(\pi,\bar{\pi}) = (-1)^{k/2}\epsilon_k + \sum_{i=0}^{\lceil k/2 \rceil - 1} (b\wp^n)^i p_{k-2i}(\pi,\bar{\pi}).
\]
Combining this with Lemma~\ref{lem:pmexpression} gives
\[
h_k(\pi,\bar{\pi}) = (-1)^{k/2}\epsilon_k  + \sum_{i=0}^{\lceil k/2 \rceil - 1}(b\wp^n)^i (k-2i) \sum_{r_1 + 2r_2 = k-2i} (-1)^{r_2} \frac{(r_1+ r_2 - 1)!}{r_1!r_2!}a^{r_1}(b\wp^n)^{r_2}.
\]
Noting that $r_1$ and $k$ have the same parity, we substitute $r_1 = k-2j$ for $j = i,\ldots,\lfloor k/2 \rfloor$, which gives $r_2 = j-i$. Thus,
\[
h_k(\pi,\bar{\pi}) = (-1)^{k/2}\epsilon_k  + \sum_{i=0}^{\lceil k/2 \rceil - 1} (k-2i) \sum_{j=i}^{\lfloor k/2 \rfloor} (-1)^{j-i}\frac{(k-j-i - 1)!}{(k-2j)!(j-i)!} a^{k-2j} (b\wp^n)^{j}.
\]
Changing the order of the sums, we get
\begin{equation}\label{eq:inter_expr}
h_k(\pi,\bar{\pi}) = (-1)^{k/2}\epsilon_k  + \sum_{j=0}^{\lfloor k/2 \rfloor} c_{k,j} a^{k-2j} (b\wp^n)^{j},
\end{equation}
where
\[
    c_{k,j} = \sum_{i=0}^{\min(j,\lceil k/2 \rceil - 1)} (-1)^{j-i} (k-2i) \frac{(k-j-i - 1)!}{(k-2j)!(j-i)!}.
\]
If $j \leq \lceil k/2 \rceil - 1$, then $c_{k,j} = (-1)^j \binom{k-j}{j}$ by Lemma~\ref{binomlemma}. If $j = \lfloor k/2 \rfloor > \lceil k/2 \rceil - 1$, then $k$ is even and we see directly that
\[
c_{k,k/2} = 2 \sum_{i=0}^{k/2 - 1} (-1)^{k/2-i} = \begin{cases} 0 & \text{if } k/2 \text{ is even;} \\ -2 & \text{if } k/2 \text{ is odd.} \end{cases}
\]
Hence the term corresponding to $j = k/2$ in equation~\eqref{eq:inter_expr} becomes $\epsilon_k$. This yields the desired expression.
\end{proof}

By Prop.~\ref{prop:charpol_properties}, the
pair $(a,b)$ of coefficients of~$c_\phi(X)$ determines the Drinfeld module~$\phi$ up to isogeny. Because the trace formula is a sum over isomorphism classes of Drinfeld modules, the numbers $\# \text{Iso}_{\mathfrak p^n}(a,b)$ (see Definition~\ref{def:iso_sets}) naturally show up when rewriting the trace formula in terms of~$a$ and~$b$. The following lemma, which is a partial generalization of \cite[Prop.~3]{yu_isogs}, will help to simplify the resulting trace formula.

\begin{lemma}\label{lem:isog_sum}
    Let $\mathfrak p \trianglelefteq A$ be a maximal ideal. Then for any $n \geq 1$ and any $t \in \mathbb Z$, we have
    \[
    \sum_{a \in A} \sum_{b \in \mathbb F_q^\times} \# \text{Iso}_{\mathfrak p^n}(a,b) b^t = 0
    \]
    as elements of $\bF_q$.
\end{lemma}

\begin{proof}
    Since there are no cusp forms of weight $2$, this follows from the trace formula \eqref{eq:trace_formula} by setting $k = 0$ and $l = t+1$.
\end{proof}

\begin{proposition}[Trace formula]\label{prop:trace_formula}
    For any $k \geq 0$, $n \geq 1$, and $l \in \mathbb Z$, we have
    \begin{equation}\label{eq:trace_formula_2}
    \Tr(\T^n_{\mathfrak p} \hspace{0.2em} | \hspace{0.1em} \S_{k+2,l}) = \sum_{a,b} \# \text{Iso}_{\mathfrak p^n}(a,b)  \sum_{j=0}^{\lceil k/2 \rceil - 1} c_{k,j} a^{k-2j} b^{j+l-k-1}\wp^{nj},
    \end{equation}
    where $c_{k,j} = (-1)^j \binom{k-j}{j}$.
\end{proposition}

\begin{proof}
    Combining the trace formula~\eqref{eq:trace_formula} with Lemma~\ref{lem:hmexpression} gives the required formula~\eqref{eq:trace_formula_2}, up to an additional term involving~$\epsilon_k$. To see why the $\epsilon_k$-term disappears, suppose $k$ is even. Then
    \[
    \sum_{a,b} \# \text{Iso}_{\mathfrak p^n}(a,b) \epsilon_k = (-\wp^n)^{k/2} \sum_{a,b} \# \text{Iso}_{\mathfrak p^n}(a,b)b^{k/2},
    \]
    which is zero by Lemma~\ref{lem:isog_sum}.
\end{proof}

\begin{remark}
Combining Prop.~\ref{prop:trace_formula} with Remark~\ref{rem:iso_c}, one sees that $\Tr(\T^n_{\mathfrak p} \hspace{0.2em} | \hspace{0.1em} \S_{k+2,l}) = 0$ if $k+2 \not\equiv 2l \pmod{q-1}$, in accordance with Lemma~\ref{lem:cuspdim}. 
As a result, one obtains by induction on~$k$ that for any $\mathfrak p$ and $n$,
\[
\sum_{a,b} \# \text{Iso}_{\mathfrak p^n}(a,b)a^kb^t \neq 0 \implies k + 2t \equiv_{q-1} 0.
\]
\end{remark}

\begin{remark}
    One can prove analogues of Lemma~\ref{lem:isog_sum} by using the trace formula and the fact that there are no cusp forms weight less than~$q+1$. For instance, one obtains by induction on~$k$ that for any $\mathfrak p$, $n$, and~$t$,
    \[
    \sum_{a,b}\#\text{Iso}_{\mathfrak p^n}(a,b)a^kb^t = 0 \quad \text{if } k<q-1.
    \]
    In particular, the sum over $j$ in the trace formula \eqref{eq:trace_formula_2} only goes up to $\lfloor (k-q+1)/2\rfloor$.
    
    One can prove further identities by exploiting the eigenforms of low weight with $A$-expansions, the simplest of which corresponds to $h \in \S_{q+1,1}$:
    \[
    \sum_{a,b} \# \text{Iso}_{\mathfrak p^n}(a,b)a^{q-1}b^t = \begin{cases}
        1 & \text{if } t \equiv 0 \pmod{q-1}; \\ 0 & \text{otherwise.}
    \end{cases}
    \]
\end{remark}

We also have a version of the trace formula for the spaces $\S_{k+2}$.

\begin{proposition}[Trace formula for $\operatorname{SL}_2(A)$]\label{prop:trace_formula_SL2}
    For any $k \geq 0$, $n \geq 1$, and $l \in \mathbb Z$, we have
    \begin{equation}\label{eq:trace_formula_SL2}
    \Tr(\T^n_{\mathfrak p} \hspace{0.2em} | \hspace{0.1em} \S_{k+2}) = - \sum_{a \in A} \# \text{Iso}_{\mathfrak p^n}(a,1)  \sum_{j=0}^{\lceil k/2 \rceil - 1} c_{k,j} a^{k-2j}\wp^{nj}.
    \end{equation}
\end{proposition}

\begin{proof}
    By definition, $\S_{k+2}$ is the sum of the spaces $\S_{k+2,l}$ as $l$ ranges over all types. For $b \in \mathbb F_q^\times$, we have
    \[
    \sum_{l=1}^{q-1}b^l = \begin{cases}
        -1 & \text{if } b = 1; \\ 0 & \text{otherwise}.
    \end{cases}
    \]
    Hence the result follows from Prop.~\ref{prop:trace_formula}.
\end{proof}

\subsection{Symmetry}

The trace formula contains the binomial coefficients $c_{k,j} = (-1)^j \binom{k-j}{j}$. These coefficients exhibit a certain symmetry (Prop.~\ref{prop:lucas_consequences}.1), which extends to a symmetry of traces of Hecke operators around weights of the form $p^m+1$. A likely conceptual explanation will be given in Remark~\ref{rmk:symmetry_hyperder}.

In what follows, we will repeatedly use Lucas's theorem, which we recall here for convenience.

\begin{theorem}[Lucas's theorem]\label{thm:lucas}
    Fix a prime power $q$ and integers $x,y \geq 0$. Suppose $x = \sum x_iq^i$, $k = \sum y_iq^i$ are the $q$-ary expansions of~$x$ and~$y$ respectively. Then as elements of~$\bF_q$, we have
    \[
    \binom{x}{y} = \prod_{i \geq 0} \binom{x_i}{y_i} \in \bF_q.
    \]
\end{theorem}

\begin{proof}
    The statement is well-known when $q$ is a prime. If $q$ is a power of a prime~$p$, the statement follows by writing $x_i$ and $y_i$ in $p$-ary form.
\end{proof}

We record some consequences for later use.

\begin{proposition}\label{prop:lucas_consequences}
    Let $m \geq 1$ and $k = p^m - 1$. The following hold:
    \begin{enumerate}
        \item For all $1 \leq N \leq k$ and $j \geq N$, we have $c_{k+N,j} \equiv_p c_{k-N,j-N}$. That is,
        \[
        (-1)^j \binom{k+N-j}{j} \equiv (-1)^{j-N} \binom{k-j}{j-N} \pmod{p}.
        \]
        \item For all $0 \leq y \leq x \leq k$, we have
        \[
        \binom{x}{y} \equiv \binom{p^m + x}{y} \pmod{p}.
        \]
        \item If $p=2$, we have
        \[
        \binom{2^m - 1 - j}{j} \equiv 0 \pmod{2} \quad \text{for all } 0 < j < 2^m.
        \]
    \end{enumerate}
\end{proposition}

\begin{proof}\hfill
    \begin{enumerate}
    \item Let $0 \leq y \leq x \leq k$. Then Lucas's theorem implies \cite[p.\ 480]{mattarei},
    \begin{equation}\label{eq:mattarei}
       (-1)^x \binom{x}{y}\equiv_p (-1)^y \binom{k - y}{k - x}.
    \end{equation}
    If $2j \leq k + N$, the claimed identity follows by letting $x = k + N - j$ and $y = j$. If $2j > k+N$, the identity also holds, as both sides of the equation are zero.
    \item This follows because the $p^m$-digits of~$x$ and~$y$ are both zero, so by \hyperref[thm:lucas]{Lucas's theorem},
    \[
    \binom{p^m + x}{y} \equiv_p \binom{1}{0} \binom{x}{y}.
    \]
    \item The statement is obvious if $m = 1$, so let $m > 1$. If $j$ is odd, the statement follows from Lucas's theorem by considering the $2^0$-digits of $2^m-1-j$ and~$j$. If $j = 2i$ is even, the $2^0$-digits give a contribution of $\binom{1}{0}$ modulo~$2$, so we can remove them. Since removing the $2^0$-digit of $n \in \mathbb N$ corresponds to the function $n \mapsto \lfloor n/2 \rfloor$, this gives
    \[
    \binom{2^m-1-2i}{2i} \equiv_2 \binom{2^{m-1}-1-i}{i} \equiv_2 0
    \]
    by induction on~$m$.\qedhere
    \end{enumerate}
\end{proof}

\begin{theorem}[Symmetry]\label{thm:symmetry_allq}
    Let $m \geq 1$. Let $\mathfrak p \trianglelefteq \bF_q[T]$ be a maximal ideal 
    and let $n \geq 1$. Then for any $1 \leq N \leq p^m$ and any $l \in \mathbb Z$, we have
        \[
    \Tr(\T_{\mathfrak p}^n \hspace{0.2em} | \hspace{0.1em} \S_{p^m+1+N,l}) = \wp^{Nn} \Tr(\T_{\mathfrak p}^n \hspace{0.2em} | \hspace{0.1em} \S_{p^m+1-N,l-N}) + \epsilon,
    \]
    where 
    \[
    \epsilon = \sum_{a,b} \# \text{Iso}_{\mathfrak p^n}(a,b) (ab^{-1})^{p^m} \sum_{j=0}^{\lfloor (N-1)/2 \rfloor} c_{N-1,j} a^{N-1-2j}b^{j+l-N}\wp^{nj}.
    \]
\end{theorem}

\begin{proof}
Write $k = p^m - 1$. We compute, using \hyperref[prop:trace_formula]{the trace formula} and Prop.~\ref{prop:lucas_consequences}.1:
    \begin{align*}
    \Tr(\T_{\mathfrak p}^n \hspace{0.2em} &| \hspace{0.1em} \S_{k+2+N,l}) = \sum_{a,b} \# \text{Iso}_{\mathfrak p^n}(a,b)  \sum_{j=0}^{\lceil (k+N)/2 \rceil - 1} c_{k+N,j} a^{k+N-2j} b^{j+l-(k+N)-1}\wp^{nj} \\
    &= \sum_{a,b} \# \text{Iso}_{\mathfrak p^n}(a,b)  \sum_{j=N}^{\lceil (k+N)/2 \rceil - 1} c_{k-N,j-N} a^{k-N-2(j-N)} b^{j-N+(l-N)-(k-N)-1}\wp^{nj} + \epsilon_0 \\
    &=
    \wp^{Nn} \sum_{a,b} \# \text{Iso}_{\mathfrak p^n}(a,b)  \sum_{j=0}^{\lceil (k-N)/2 \rceil - 1} c_{k-N,j} a^{k-N-2j} b^{j+(l-N)-(k-N)-1}\wp^{nj} + \epsilon_0 \\
    &= \wp^{Nn} \Tr(\T_{\mathfrak p}^n \hspace{0.2em}| \hspace{0.1em} \S_{k+2-N,l-N}) + \epsilon_0,
    \end{align*}
    where
    \[
    \epsilon_0 = \sum_{a,b} \# \text{Iso}_{\mathfrak p^n}(a,b)  \sum_{j=0}^{N - 1} c_{k+N,j} a^{k+N-2j} b^{j+l-(k+N)-1}\wp^{nj}.
    \]
    We simplify $\epsilon_0$ as follows. Note that for $0 \leq j \leq N-1$, both $j$ and $N-1-j$ are less than~$p^m$. By Prop.~\ref{prop:lucas_consequences}.2, we obtain
    \[
        c_{k+N,j} = (-1)^j \binom{k+N-j}{j} = (-1)^j \binom{p^m + N - 1 - j}{j} \equiv_p (-1)^j \binom{N-1-j}{j} = c_{N-1,j}.
    \]
    The latter binomial coefficient vanishes for all $j > N-1-j$, which gives the desired $\epsilon_0 = \epsilon$.
\end{proof}


\begin{remark}\label{rmk:symmetry_hyperder}
    A possible explanation for the symmetry in the traces are the hyperderivations from \cite{bosser-pellarin}. These are maps $\mathcal{D}_N: \S_{k,l} \to \widetilde{\S}_{k+2N,l+N}$, where $N \geq 0$ is an integer and $\widetilde{\S}_{k+2N,l+N}$ denotes the space of Drinfeld quasi-modular cusp forms of weight $k+2N$ and type $l+N$. It turns out that the image of $\mathcal{D}_N$ may actually lie inside $\S_{k+2N,l+N}$ (and whether this is the case depends only on $k$ and~$n$; cf.~\cite[Prop.~3.1]{bosser-pellarin}). In this situation, another interpretation of the maps $\mathcal{D}_N$ has been given in~\cite{bockle_maeda}.

    In particular, one can show that $\mathcal{D}_N$ defines a map $\S_{p^m+1-N,l-N} \to \S_{p^m+1+N,l}$. Moreover, it is Hecke-equivariant up to a character twist: if $\T_{\mathfrak p}f = \lambda f$ then $\T_{\mathfrak p}(\mathcal{D}_Nf) = \wp^N \lambda \mathcal{D}_Nf$. However, it may be the case that $\mathcal{D}_Nf = 0$.

    Theorem~\ref{thm:symmetry_allq} suggests that $\mathcal{D}_N$ is injective in weights of the form $p^m+1-N$. If this is true, then $\epsilon$ has a concrete interpretation: namely,
    \[
    \epsilon = \Tr\!\left(\T_{\mathfrak p}^n \hspace{0.2em} | \hspace{0.1em} \S_{p^m+1+N,l} / \mathcal{D}_N \S_{p^m+1-N,l-N}\right)\!.
    \]
\end{remark}

\subsection{Primes of degree 1}\label{sec:deg1}

In this section, we apply the trace formula when $\mathfrak p$ is a prime of degree~1. The main result is the following.

\begin{theorem}\label{thm:deg1}
    Fix $k \geq 0$, $l \in \mathbb Z$, $x \in \bF_q$, and let $\mathfrak p = (T-x)$. If $k+2 \equiv 2l \pmod{q-1}$, we have
    \[
    \Tr(\T_{\mathfrak p} \hspace{0.2em} | \hspace{0.1em} \S_{k+2,l}) = \sum_{\substack{0 \leq j < k/2 \\ j \equiv l-1 \pmod{q-1}}} (-1)^j {\binom{k-j}{j}}(T-x)^j,
    \]
    and $\Tr(\T_{\mathfrak p} \hspace{0.2em} | \hspace{0.1em} \S_{k+2,l}) = 0$ otherwise.
\end{theorem}

The proof of Theorem~\ref{thm:deg1} relies on the following lemma.

\begin{lemma}\label{lem:isoclasses}
    Let $\mathfrak p \trianglelefteq A$ be a prime of degree~1. Then for $a \in A$ and $b \in \bF_q^\times$, we have
    \[
    \# \text{Iso}_{\mathfrak p}(a,b) = \begin{cases}
        1 & \text{if }a \in \bF_q; \\
        0 & \text{otherwise.}
    \end{cases}
    \]
\end{lemma}

\begin{proof}
    A Drinfeld module $\phi$ over $\bF_q$ of characteristic $(T-x)$ is determined by
    \[
    \phi_T = x + \alpha \tau + \beta \tau^2 \in \bF_q\{\tau\},
    \]
    with $\beta \neq 0$. One easily checks that its characteristic polynomial is given by
    \[
    c_\phi(X) = X^2 + \beta^{-1}\alpha X - \beta^{-1}(T-x),
    \]
    so the Weil polynomial $X^2 + aX + b(T-x)$ occurs only if $a \in \bF_q$, in which case it occurs precisely once.
\end{proof}

\begin{proof}[Proof of Theorem~\ref{thm:deg1}]
Let $\wp = T-x$ denote the monic generator of $\mathfrak p$. By Prop.~\ref{prop:trace_formula} and Lemma~\ref{lem:isoclasses}, we have
\begin{equation}\label{eq:interhecke}
\Tr(\T_{\mathfrak p} \hspace{0.2em} | \hspace{0.1em} \S_{k+2,l} ) = \sum_{j=0}^{\lceil k/2 \rceil - 1} (-1)^{j}\binom{k-j}{j} \wp^j \sum_{a \in \mathbb F_q} \sum_{b \in \mathbb F_q^\times} a^{k-2j} b^{j+l-k-1}.
\end{equation}
Recall that for $n \geq 1$ and any prime power~$q$, we have
\[
\sum_{x \in \bF_q}x^n = \sum_{x \in \bF_q^\times}x^n = \begin{cases}
    -1 & \text{if } {q-1} \mid n; \\ 0 & \text{otherwise.}
\end{cases}
\]
Applying the above identity to the sums over~$a$ and~$b$, we see that the summand corresponding to a given~$j$ vanishes unless
\begin{equation}
0 \leq j < k/2 \hspace{1em} \text{and} \hspace{1em} k \equiv 2j \pmod {q-1} \hspace{1em} \text{and} \hspace{1em} j \equiv k+1-l \pmod {q-1}.
\end{equation}
Hence the expression \eqref{eq:interhecke} can be rewritten as
\[
\Tr(\T_{\mathfrak p} \hspace{0.2em} | \hspace{0.1em} \S_{k+2,l} ) = \sum_{\substack{0 \leq j < k/2 \\ j \equiv l-1 \pmod {q-1}}} (-1)^{j}\binom{k - j}{j}\wp^j,
\]
as desired.
\end{proof}

\begin{remark}\label{rmk:simple_traces}
    Suppose that for any $b \in \bF_q^\times$, we have
    \[
    \# \text{Iso}_{\mathfrak p^n}(a,b) \equiv_p \begin{cases}
        1 & \text{if }a \in \mathbb F_q; \\ 0 & \text{if } \deg(a) > 0.
    \end{cases}
    \]
    Define
    \[
    f(X) := \sum_{\substack{0 \leq j < k/2 \\ j \equiv l-1 \pmod {q-1}}} (-1)^{j}\binom{k - j}{j}X^j.
    \]
    Then the proof of Theorem~\ref{thm:deg1} shows that $\Tr(\T_{\mathfrak p}^n \hspace{0.2em} | \hspace{0.1em} 
    \S_{k+2,l}) = f(\wp^n)$.
    One example in which the assumptions are satisfied is $q = 3$, $n = 1$, and $\wp = T^3 + 2T + 1$.
\end{remark}

For simplicity, we will from now on focus on the Hecke operator~$\T_T$, taking it to be understood that the results will hold for any prime $T-x$ of degree~1 after substituting $T-x$ for~$T$.

\begin{example} 
    Fix the type $l$ to be~$1$. In Table~\ref{table:type1}, we list some traces of the Hecke operator $\T_{T}$ on~$\S_{k,1}$ for $q=3,5,7,9$, computed using Theorem~\ref{thm:deg1}. Note that the zeroes in Table~\ref{table:type1} appear only for $k \not \equiv 2 \pmod{q-1}$, i.e., when $\S_{k,1} = 0$. 
\begin{table}[hbtp]
\begin{center}
\begin{tabular}{||c c c c c||}
 \hline
 $k$ & $q = 3$ & $q=5$ & $q = 7$ & $q = 9$\\ [0.5ex] 
 \hline\hline
 4 & 1 & 0 & 0 & 0 \\
 \hline
 6 & 1 & 1 & 0 & 0 \\
 \hline
 8 & 1 & 0 & 1 & 0 \\ 
 \hline
 10 & 1 & 1 & 0 & 1 \\
 \hline
 12 & $T^2 + 1$ & 0 & 0 & 0 \\  
 \hline
 14 & $T^4 + 1$ & 1 & 1 & 0\\
 \hline
 16 & $T^6 + 1$ & 0 & 0 & 0 \\
 \hline
 18 & $T^2 + 1$ & 1 & 0 & 1\\ 
 \hline
 20 & $2T^4 + 1$ & 0 & 1 & 0 \\
 \hline
 22 & $2T^4 + 1$ & 1 & 0 & 0\\
 \hline
 24 & $T^6 + T^2 + 1$ & 0 & 0 & 0 \\
 \hline
 26 & $2T^{10} + 1$ & 1 & 1 & 1 \\
 \hline
 28 & $T^{12} + T^{10} + T^4 + 1$ & 0 & 0 & 0 \\
 \hline
 30 & $2T^{12} + T^2 + 1$ & $T^4 + 1$ & 0 & 0 \\
 \hline
 32 & $T^{10} + T^6 + T^4 + 1$ & 0 & 1 & 0 \\
 \hline
 34 & $2T^{10} + T^6 + 1$ & $T^8 + 1$ & 0 & 1 \\
 \hline
 36 & $2T^{12} + T^8 + T^2 + 1$ & 0 & 0 & 0 \\
 \hline
 38 & $T^{12} + T^{10} + 2T^4 + 1$ & $T^{12} + 1$ & 1 & 0 \\
 \hline
 40 & $T^{18} + T^{12} + 2T^4 + 1$ & 0 & 0 & 0 \\
 \hline
 42 & $T^{18}+T^{14}+T^6 +T^2 + 1$ & $T^{16}+ 1$ & 0 & 1 \\
 \hline
 44 & $T^{18} + T^{16} + 1$ & 0 & 1 & 0 \\
 \hline
 46 & $T^{18} + T^{4} + 1$ & $T^{20}+1$ & 0 & 0 \\
 \hline
 48 & $T^{20} + T^{2} + 1$ & 0 & 0 & 0 \\
 \hline
 50 & $T^{22}+2T^{10}+T^6 + T^4 + 1$ & $T^4 + 1$ & 1 & 1 \\
 \hline
 52 & $T^{24}+T^{10}+T^6 + 1$ & 0 & 0 & 0 \\
 \hline
  54 & $T^{12} + T^{8} + T^2 + 1$ & $2T^8 + 1$ & 0 & 0 \\
 \hline
  56 & $2T^{12} + 2T^{10} + 2T^4 + 1$ & 0 & $T^6 + 1$ & 0 \\
 \hline
  58 & $2T^{12} + 2T^{10} + 2T^4 + 1$ & $3T^{12} + 4T^{8} + 1$  & 0 & 1 \\
 \hline
  60 & $2T^{14} + T^{6} + T^2 + 1$ & 0 & 0 & 0 \\
 \hline 
 62 & $2T^{28} + 2T^{16} + 2T^{12} + T^{10} + 1$ & $4T^{16} + 3T^{12} + 1$ & $T^{12} + 1$ & 0 \\
 \hline
\end{tabular}
\end{center}
\caption{Traces of the Hecke operator $\T_T$ acting on $\S_{k,1}$ for varying $k$ and $q \in \{3,5,7,9\}.$}\label{table:type1}
\end{table}
\end{example}

Theorem~\ref{thm:deg1} has several consequences. 

\begin{corollary}\label{cor:rational_fn}
    For every $l \in \mathbb Z$, the power series
    \[
    \sum_{k \geq 0} \Tr(\T_{T} \hspace{0.2em} | \hspace{0.1em} \S_{k+2,l} ) X^k 
    \]
    is a rational function.
\end{corollary}

\begin{proof}
    Consider the generating series
    \[
    g_0(X,Y) := \frac{1}{1-(1+Y)X} = \sum_{n \geq 0}\sum_{j \geq 0} \binom{n}{j} Y^j X^n.
    \]
    Setting $Y = -TX$ gives
    \[
    \frac{1}{1 - X + TX^2} = \sum_{n \geq 0} \sum_{j \geq 0} (-1)^j \binom{n}{j} T^j X^{n+j} = \sum_{k \geq 0} \left( \sum_{j = 0}^{\lfloor k/2 \rfloor} (-1)^j \binom{k-j}{j} T^j \right) X^k,
    \]
    and subtracting $(1+TX^2)^{-1}$ gives
    \[
    g_1(X,T) := \frac{X}{(1-X+TX^2)(1+TX^2)} = \sum_{k \geq 0}\left( \sum_{0 \leq j < k/2} (-1)^j \binom{k-j}{j}T^j \right) X^k.
    \]
    Let $\zeta \in \bF_q^\times$ be a generator. Then 
    \[
    S_N := \sum_{i=1}^{q-1} \zeta^{iN} = \begin{cases} -1 & \text{if } N \equiv 0 \pmod{q-1}; \\ 0 & \text{otherwise}, \end{cases}
    \]
    since $(1-\zeta^N)S_N = 0$. Hence we have
    \[
    g_2(X,T) := -\sum_{i=1}^{q-1} \zeta^{i(l-1)} g_1(X,\zeta^{-i}T) = \sum_{k \geq 0} \Bigg( \sum_{\substack{0 \leq j < k/2 \\ j \equiv l-1 \pmod{q-1}}} (-1)^j \binom{k-j}{j}T^j \Bigg) X^k,
    \]
    and finally, by Theorem~\ref{thm:deg1},
    \[
    \sum_{k \geq 0} \Tr(\T_{T} \hspace{0.2em} | \hspace{0.1em} \S_{k+2,l} ) X^k = -\sum_{i=1}^{q-1} \zeta^{2li} g_2(\zeta^{-i}X,T). \qedhere
    \]
\end{proof}

We note the following improvement of Theorem~\ref{thm:symmetry_allq} for primes of degree~1.

\begin{theorem}[Symmetry for degree~1 primes]\label{thm:symmetry_allq_deg1}
    Let $m \geq 1$. Then for any $1 \leq N \leq p^m$ and any $l \in \mathbb Z$ such that $p^m+1+N \equiv 2l \pmod{q-1}$, we have
        \[
    \Tr(\T_{T} \hspace{0.2em} | \hspace{0.1em} \S_{p^m+1+N,l}) = T^{N} \Tr(\T_{T} \hspace{0.2em} | \hspace{0.1em} \S_{p^m+1-N,l-N}) + \epsilon(T),
    \]
    where 
    \[
    \epsilon(T) = \sum_{\substack{0 \leq 2j \leq N-1 \\ j \equiv l-1 \pmod{q-1}}} (-1)^j {\binom{N-1-j}{j}}T^j.
    \]
    In particular, if $\Tr(\T_{T} \hspace{0.2em} | \hspace{0.1em} \S_{p^m+1-N,l-N}) \neq 0$, we have
    \[
    \deg \Tr(\T_{T} \hspace{0.2em} | \hspace{0.1em} \S_{p^m+1+N,l}) = N + \deg \Tr(\T_{T} \hspace{0.2em} | \hspace{0.1em} \S_{p^m+1-N,l-N}).
    \]
\end{theorem}

\begin{proof}
This follows from Theorem~\ref{thm:symmetry_allq}, the expression of $\epsilon = \epsilon(T)$ being due to Lemma~\ref{lem:isoclasses}. The consequence about the degrees follows because $\deg(\epsilon) < N$.
\end{proof}

\begin{remark}
    In Theorem~\ref{thm:symmetry_allq_deg1}, it is tempting to write $\epsilon(T)$ as the trace of~$\T_T$ on~$\S_{N+1,l}$. However, due to a type mismatch, the space $\S_{N+1,l}$ is zero if $\S_{p^m+1+N,l}$ is non-zero. The only case where this is not an issue is for $q=2$ (see Theorem~\ref{thm:symmetry_q2}).
\end{remark}

\begin{corollary}\label{cor:clusters}
    Let $m \geq 1$. Then for any $k+2 \in \mathbb Z \cap [2p^m+2-q,2p^m+1]$, we have
    \[
    \deg \Tr(\T_{T} \hspace{0.2em} | \hspace{0.1em} \S_{k+2,l}) \leq \frac{k-p^m}{2}.
    \]
\end{corollary}

\begin{proof}
    Setting $p^m + 1 - q \leq N \leq p^m$ in Theorem~\ref{thm:symmetry_allq_deg1} gives
    \[
    \Tr(\T_{T} \hspace{0.2em} | \hspace{0.1em} \S_{k+2,l}) = \epsilon(T),
    \]
    since there are no cusp forms of weight less than~$q+1$. 
    Thus the corollary follows from the explicit description of~$\epsilon(T)$.
\end{proof}

\begin{remark}\label{rmk:improvedclusters}
    Corollary~\ref{cor:clusters} can be improved if there is additional information about the type, because $\S_{k,l} = 0$ if $k < l(q+1)$ for $1 \leq l \leq q-1$.
\end{remark}

\begin{example}
    Theorem~\ref{thm:symmetry_allq_deg1} implies that the distance of $\deg \Tr(\T_{T}\hspace{0.2em} | \hspace{0.1em} \S_{k,l})$ to the Ramanujan bound is symmetric in the weights $p^m+1$ such that $p^m+1 \equiv 2l \pmod{q-1}$. For instance, if $q = 5$ and $l \in \{1,3\}$, the distance is symmetric in the weights $26$, $126$, $626$, etc. Figure~\ref{fig:q5l3sym} shows the quantity 
    \[
    \log_q\left( 1 + \frac{k-(q+1)}{2} - \deg \Tr(\T_{T}\hspace{0.2em} | \hspace{0.1em} \S_{k,l})\right)
    \]
    for varying~$k$ when $q=5$ and $l=3$, so that a point lies on the $k$-axis if and only if the strong Ramanujan bound (see Conj.~\ref{conj:strong_ram_traces}) is attained. 

    Immediately after the visual symmetry in $k = p^m+1$ ends, one observes clusters of data points with large distance to the Ramanujan bound. In Figure~\ref{fig:q5l3sym}, the first four of the data points in these clusters can be explained by Remark~\ref{rmk:improvedclusters}, noting that $\S_{k',3} = 0$ for $0 \leq k' < 18$. However, the clusters tend to consist of six consecutive data points, the last two of which are not explained by Theorem~\ref{thm:symmetry_allq_deg1}.
    
    
    Similar axes of symmetry appear when $\mathfrak p$ is replaced with a prime of higher degree, one example of which is shown in Figure~\ref{fig:q5deg2sym}. Here again $q=5$ and $l=3$, but now $\mathfrak p$ is the degree~2 prime $(T^2+T+2)$ (these traces can be computed using Prop.~\ref{prop:deg2}). The data points represent the quantity
    \[
    \log_q \left( 1 + k-(q+1) - \deg \Tr(\T_{T}\hspace{0.2em} | \hspace{0.1em} \S_{k,l})\right)
    \]
    for varying~$k$. The graph is symmetric in $k=26,376,1876$, although Theorem~\ref{thm:symmetry_allq} alone does not suffice to prove this.
    
    \begin{figure}
        \centering
        \includegraphics[width=\textwidth]{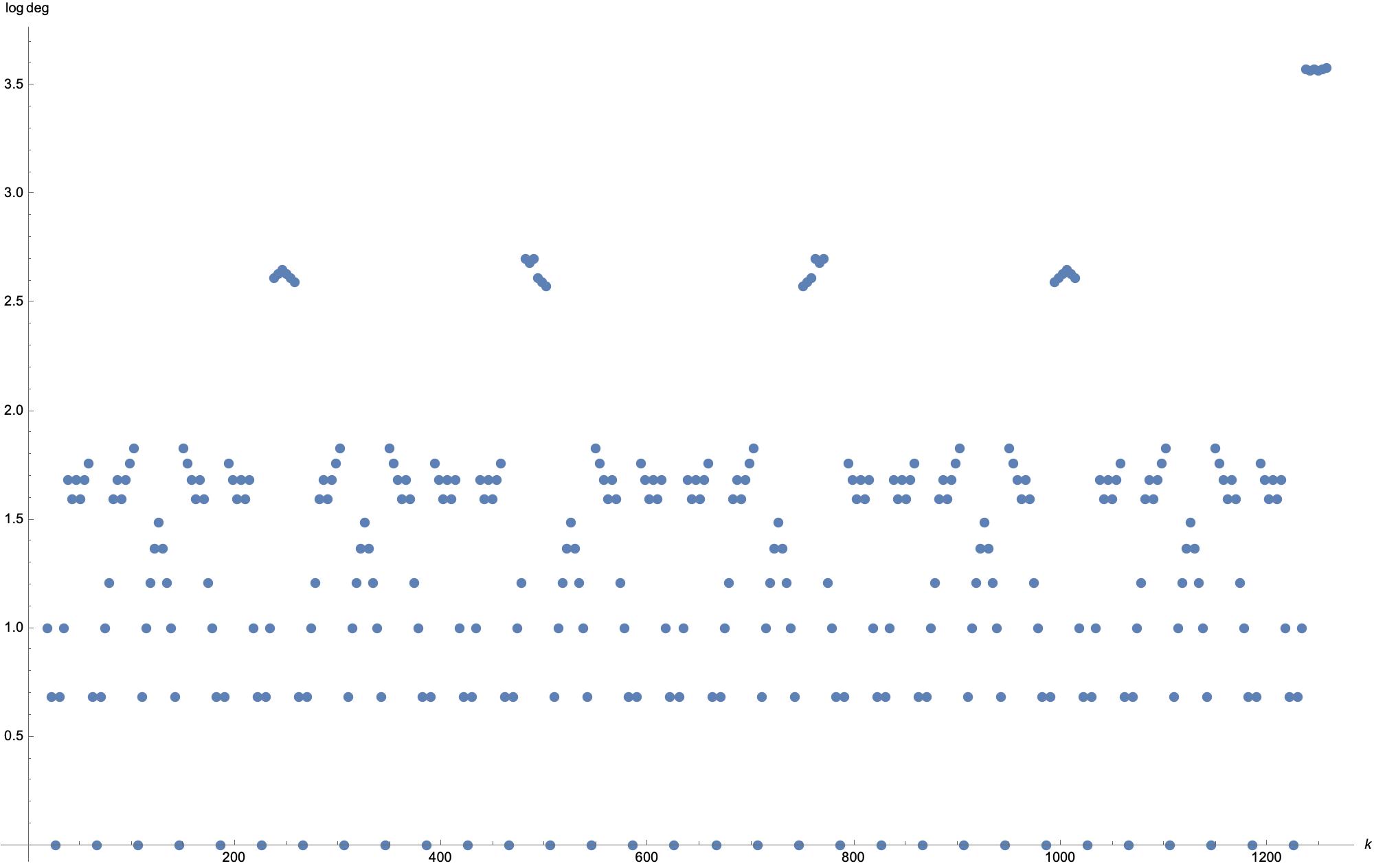}
        \caption{$\log_5( 1 + (k-6)/2 - \deg \Tr(\T_{T}\hspace{0.2em} | \hspace{0.1em} \S_{k,3}))$ for $q=5$ and $18 \leq k \leq 1258$.}
        \label{fig:q5l3sym}
    \end{figure}
    \begin{figure}
        \centering
        \includegraphics[width=\textwidth]{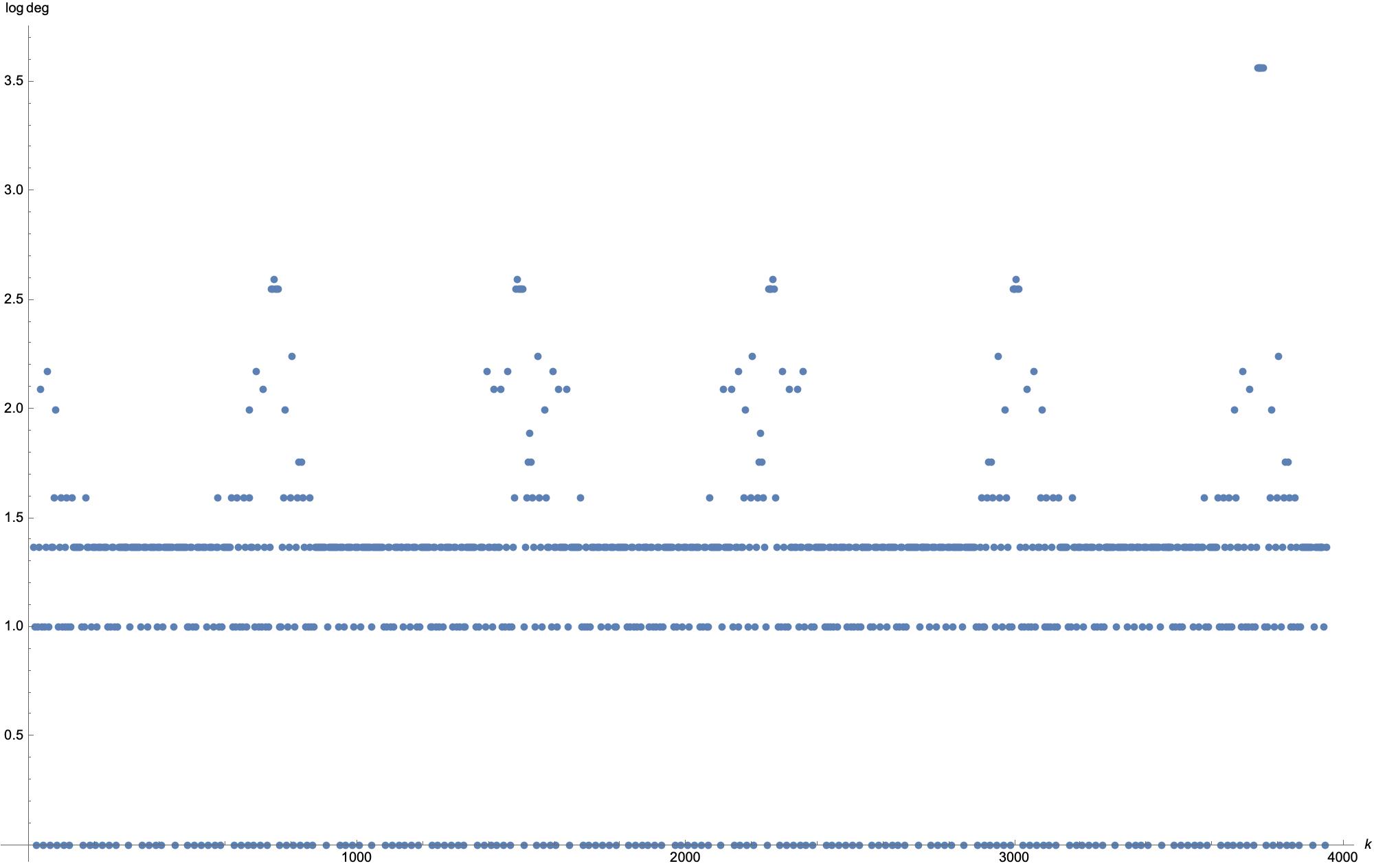}
        \caption{$\log_5 (1 + (k-6) - \deg \Tr(\T_{T^2+T+2} \hspace{0.2em} | \hspace{0.1em} \S_{k,3}))$ for $q=5$ and $18 \leq k \leq 3950$.}
        \label{fig:q5deg2sym}
    \end{figure}
\end{example}

We observe that traces of degree~1 primes are periodic modulo high powers of~$\mathfrak p$.

\begin{proposition}\label{prop:trace_periodic}
    Fix $1 \leq l \leq q-1$. Then for any weight $k > 2l$, the quantity
    \[
    \Tr(\T_{T} \hspace{0.2em} | \hspace{0.1em} \S_{k,l}) \pmod{T^{q+l-2}}
    \]
    depends only on the residue class of~$k$ in $\mathbb Z/q\mathbb Z$.
\end{proposition}

\begin{proof}
    By Theorem~\ref{thm:deg1}, all terms in $\Tr(\T_{\mathfrak p} \hspace{0.2em} | \hspace{0.1em} \S_{k,l})$ except the lowest order term lie in $(T^{q+l-2})$. The coefficient of the lowest order term is
    \[
    C := (-1)^{l-1} {\binom{k-l-1}{l-1}} \in \bF_q.
    \]
    Since $l-1 < q$, \hyperref[thm:lucas]{Lucas's theorem} implies that~$C$ only depends on the $q^0$-digit in the $q$-ary expansion of $k-l+1$, which in turn only depends on the residue class of~$k$ in $\mathbb Z/q\mathbb Z$.
\end{proof}

\begin{example}
Let $q > 2$. As a consequence of Prop.~\ref{prop:trace_periodic} and Ex.~\ref{ex:t0t2}.1, we see that for any $n \in \mathbb Z_{\geq 0}$,
    \[
    \Tr(\T_{T} \hspace{0.2em} | \hspace{0.1em} \S_{2(q+1)+n(q-1),2}) \equiv (n+1)T \pmod{T^{q}}.
    \]
\end{example}

\subsection{Traces in characteristic~2}\label{sec:char2}
In characteristic~2, traces of Hecke operators turn out to be much more well-behaved than in odd characteristic: every prime behaves like a prime of degree~1 (see Thm.~\ref{thm:char2traces}, whose proof relies on Appendix~\ref{sec:appendix}). On the other hand, the traces also contain less information: there are infinitely many weights in which the Hecke eigenvalues of a fixed Hecke operator occur with multiplicity divisible by~2 (see Thm.~\ref{thm:char2_repetition}) and so these eigenvalues cannot be studied via trace methods. By contrast, for odd~$q$ the traces are harder to pin down, but there is no known example of repeated eigenvalues at level~1.

\begin{remark}
    In characteristic~2, the type~$l$ of a non-zero space $\S_{k,l}$ is determined by the weight~$k$ by the congruence $k \equiv 2l \pmod{q-1}$. We will therefore without loss of generality work with the spaces $\S_k$ whenever this is more convenient, where we recall that $\S_k := \bigoplus_{l=1}^{q-1} \S_{k,l}$.
\end{remark}

\begin{theorem}\label{thm:char2traces}
    Suppose $2 \mid q$. Let $\mathfrak p \trianglelefteq A$ be a maximal ideal, let $\wp$ be a monic generator of~$\mathfrak p$ and let $n \geq 1$. Then for any $k \geq 0$ and $l \in \mathbb Z$, we have
    \[
    \Tr(\T_{\mathfrak p}^n \hspace{0.2em} | \hspace{0.1em} \S_{k+2,l})  = \sum_{\substack{0 \leq j < k/2 \\ j \equiv l-1 \pmod{q-1}}} {\binom{k-j}{j}}\wp^{nj} \qquad \text{if } k+2 \equiv 2l \pmod{q-1},
    \]
    and $\Tr(\T_{\mathfrak p}^n \hspace{0.2em} | \hspace{0.1em} \S_{k+2,l}) = 0$ otherwise.
\end{theorem}

\begin{proof}
    By Propositions~\ref{prop:isom} and \ref{prop:evenclassnumber}, we have for $a \in A$ and $b \in \bF_q^\times$,
    \[
    \# \text{Iso}_{\mathfrak p^n}(a,b) \equiv_2 \begin{cases}
        1 & \text{if } \deg(a) \leq 0; \\ 0 & \text{otherwise}.
    \end{cases}
    \]
    Thus the result follows from Rmk.~\ref{rmk:simple_traces}.
\end{proof}

\begin{remark}
    Theorem~\ref{thm:char2traces} implies that any result proved for~$\T_T$ extends to the analogous statement for~$\T_{\mathfrak p}^n$ simply by substituting $\wp^n$ for~$T$.
\end{remark}

\begin{corollary}
    Suppose $q = 2$. Then as elements of $\bF_2[T][\![X]\!]$, we have
    \[
    \sum_{k \geq 0} \Tr(\T^n_{\mathfrak p} \hspace{0.2em} | \hspace{0.1em} \S_{k+2} ) X^k = \frac{X}{(1+X+\wp^n X^2)(1 + \wp^n X^2)}.
    \]
\end{corollary}

\begin{proof}
    See the proof of Corollary~\ref{cor:rational_fn}.
\end{proof}

\begin{theorem}[Symmetry for $q=2$]\label{thm:symmetry_q2}
Suppose $q = 2$. Let $m \geq 1$ and $1 \leq N \leq 2^m$. Then for any maximal ideal $\mathfrak p \trianglelefteq \bF_2[T]$ and any $n \geq 1$, we have
    \[
    \Tr(\T_{\mathfrak p}^n \hspace{0.2em} | \hspace{0.1em} \S_{2^m+1+N}) = \wp^{Nn} \Tr(\T_{\mathfrak p}^n \hspace{0.2em} | \hspace{0.1em} \S_{2^m+1-N}) + \Tr(\T_{\mathfrak p}^n \hspace{0.2em} | \hspace{0.1em} \S_{N+1}) + N \wp^{n(N-1)/2}.
    \]
\end{theorem}

\begin{proof}
    It suffices to prove the case $\wp^n = T$.
    By Theorem~\ref{thm:symmetry_allq_deg1}, we only need to show that
    \[
    \sum_{0 \leq 2j \leq N-1} {\binom{N-1-j}{j}}T^j = \Tr(\T_{T} \hspace{0.2em} | \hspace{0.1em} \S_{N+1}) + N T^{\frac{N-1}{2}},
    \]
    but this is clear from Theorem~\ref{thm:char2traces}.
\end{proof}

\begin{remark}
    If $q=2$, Theorem~\ref{thm:symmetry_q2} allows for the computation of the trace of any Hecke operator in weight~$k$ in $O(k)$ steps, without having to compute any binomial coefficients.
\end{remark}

The simple version of the trace formula combined with the calculus of binomial coefficients modulo~2 allow us to explicitly compute traces of Hecke operators in special weights.

\begin{proposition}\label{prop:2powerweight}
Let $q = 2^r$. Fix $1 \leq s \leq r$ and $m \geq 0$. Write $\delta_{q,2}$ for the Kronecker delta, i.e., $\delta_{q,2} = 1$ if $q=2$ and $\delta_{q,2} = 0$ otherwise.
\begin{enumerate}
\item If $k+2=2^sq^m+1$, we have
    \[
    \Tr(\T_{\mathfrak p}^n \hspace{0.2em} | \hspace{0.1em} \S_{k+2}) = \begin{cases} 1 & \text{if } s=r; \\ 0 & \text{otherwise.}\end{cases}
    \]
\item If $k + 2 = 2^sq^m$, we have
    \[
    \Tr(\T_{\mathfrak p}^n \hspace{0.2em} | \hspace{0.1em} \S_{k+2}) = \wp^{-n} \sum_{j=0}^{m-1} \wp^{2^{s-1}q^jn}.
    \]
\item
    If $k + 2 = 2^sq^m + 2$, we have
    \[
    \Tr(\T_{\mathfrak p}^n \hspace{0.2em} | \hspace{0.1em} \S_{k+2}) = \delta_{q,2} + \sum_{j=0}^{m-1} \wp^{2^{s-1}q^jn} .
    \]
\item
    If $k + 2 = 2^{s+1}q^m + 2^{s}q^m + 1$, we have
    \[
    \Tr(\T_{\mathfrak p}^n \hspace{0.2em} | \hspace{0.1em} \S_{k+2}) = 
    \begin{cases}
        \delta_{q,2} + \wp^{q^{m+1}n} & \text{if } s = r;\\
        0 & \text{otherwise.}
    \end{cases}
    \]
\end{enumerate}
\end{proposition}

\begin{proof}
    The first assertion follows from Theorem~\ref{thm:char2traces} and Prop.~\ref{prop:lucas_consequences}.3. The proofs of the other identities are left to the reader.
\end{proof}

\begin{example}
    To compute the trace of~$\T_T$ on~$\S_{177}$ when $q=2$, one can proceed as follows. Repeated application of Theorem~\ref{thm:symmetry_q2} gives
    \begin{align*}
        \Tr(\T_T \hspace{0.2em} | \hspace{0.1em} \S_{177}) &= \Tr(\T_T \hspace{0.2em} | \hspace{0.1em} \S_{129+48}) = T^{48} \Tr(\T_{T} \hspace{0.2em} | \hspace{0.1em} \S_{129-48}) + \Tr(\T_{T} \hspace{0.2em} | \hspace{0.1em} \S_{49})\\
        &= T^{48}( T^{16} \Tr(\T_{T} \hspace{0.2em} | \hspace{0.1em} \S_{65 - 16}) + \Tr(\T_{T} \hspace{0.2em} | \hspace{0.1em} \S_{17})) + T^{16}\Tr( \T_{T} \hspace{0.2em} | \hspace{0.1em} \S_{33-16}) + \Tr(\T_{T} \hspace{0.2em} | \hspace{0.1em} \S_{17}) \\
        &= (T^{80} + T^{64} + T^{48} + T^{16} + 1) \Tr(\T_{T} \hspace{0.2em} | \hspace{0.1em} \S_{17}) \\
        &= T^{80} + T^{64} + T^{48} + T^{16} + 1,
    \end{align*}
    using that $\Tr(\T_{T} \hspace{0.2em} | \hspace{0.1em} \S_{17}) = 1$ by Prop.~\ref{prop:2powerweight}.1.
\end{example}

\subsection{Traces in odd characteristic: primes of degree~2}

As the characteristic~2 case is now settled, we may safely assume that $q$ is odd. We turn to the study of traces of powers of Hecke operators $\T_{\mathfrak p}^n$ when $nd > 1$, where $d = \deg(\mathfrak p)$. When $nd = 2$, we obtain the following explicit expression.

\begin{theorem}\label{thm:deg2}
    Suppose $2 \nmid q$ and $\wp$ is a monic irreducible polynomial of degree~$d \in \{1,2\}$. Write $\mathfrak p = (\wp)$ and let $nd = 2$. Suppose $k + 2 \equiv 2l \pmod{q-1}$. Then we have
    \begin{align*}
    &\qquad \qquad \quad \Tr(\T_{\mathfrak p}^n \hspace{0.2em} | \hspace{0.1em} \S_{k+2,l}) = 
    \sum_{\substack{0 \leq j < k/2 \\ j \equiv l-1}} (-1)^{j} \binom{k-j}{j} \wp^{nj} \ + \\
    &+ \sum_{m=(q-1)/2}^{q-2} 4^{-m} \binom{m}{(q-1)/2} \sum_{\substack{0 \leq j < k/2 \\ j \equiv l-1+m}} \sum_{\substack{0 < i < k-2j \\ i \equiv -2m}} (-1)^{j+\frac{q-1}{2}} \binom{k-j}{i,\, j,\, k-2j-i} T^i \wp^{nj},
    \end{align*}
    where the congruences imposed on $i$ and~$j$ are modulo~$q-1$.
\end{theorem}

The first step towards proving Theorem~\ref{thm:deg2} is to understand the cardinalities $\# \text{Iso}_{\mathfrak p^n}(a,b)$.
We fix the following notation. For $\alpha,\beta \in \mathbb F_q$, define the Legendre symbol via
\[
\left( \frac{\alpha,\beta}{q} \right) := \begin{cases} 1 & \text{if }X^2 - \alpha X + \beta \text{ has two distinct roots in $\mathbb F_q$}; \\ 0 & \text{if }X^2 - \alpha X + \beta \text{ is a square}; \\ -1 & \text{if }X^2 - \alpha X + \beta \text{ is irreducible over $\mathbb F_q$}. \end{cases}
\]
Note that this is the same as the Legendre symbol for $D = \alpha^2 - 4\beta$ in $\bF_q$, i.e., $\left( \frac{\alpha,\beta}{q} \right) = D^{(q-1)/2}$. If $a \in \bF_q[T]$ is a polynomial of degree $\leq 1$, we denote by $a^+ \in \bF_q$ the coefficient of $T$ in $a$.

\begin{proposition}\label{prop:deg2}
    Suppose $2 \nmid q$ and let $nd = 2$. Then for any $a \in \bF_q[T]$ with $\deg(a) \leq 1$ and $b \in \bF_q^\times$, we have
    \[
    \# \text{Iso}_{\mathfrak p^n}(a,b) \equiv  1 - \left( \frac{a^+,b}{q} \right) \pmod{p}.
    \]
\end{proposition}

\begin{proof}
    We prove this case by case, comparing each time an explicit computation to the answer provided by Prop.~\ref{prop:isom}. Throughout, we abbreviate $H(D) := H(A[\sqrt{D}])$ and $\chi_D := \chi_{K(\sqrt{D})}$.

    In case~1, let $f(X,T) = X^2 - aX + b\wp^n$, with $\wp \nmid a = a^+T + a^-$ and $b \in \bF_q^\times$. Let $D = a^2 - 4b\wp^n$ be the discriminant of~$f$. In this case, Prop.~\ref{prop:isom} implies that $\# \text{Iso}_{\mathfrak p^n}(a,b) = H(D)$ if the splitting field of~$f$ is imaginary, and $\# \text{Iso}_{\mathfrak p^n}(a,b) = 0$ otherwise. We distinguish three subcases.
    \begin{enumerate}
        \item[i)] $D$ is not square-free. In this case, we have $D = \mu Q^2$ for a monic prime $Q$, necessarily of degree~1, and $\mu = (a^+)^2 - 4b \in \bF_q^\times$. In this case, $\mu$ is a non-square. Indeed, if $\mu = \lambda^2$, then $-4b\wp^n = (\lambda Q + a)(\lambda Q - a)$. Since $\wp$ is prime, this implies $a = 0$, which contradicts $\wp \nmid a$. Thus, $K(\sqrt{D}) = K(\sqrt{\mu}) = K \otimes \bF_{q^2}$, so both $\infty$ and $Q$ are inert and Lemma~\ref{lem:hurwitz_number} gives
        \[
        \# \text{Iso}_{\mathfrak p^n}(a,b) = H(D) \equiv_p 1 - \chi_{D}(Q) = 2,
        \]
        The fact that $\mu$ is a non-square also gives
        \[
        1 - \left( \frac{a^+,b}{q} \right) = 2,
        \]
        which is what we wanted.
        \item[ii)] $D$ is square-free and $\deg(D) > 0$. In this case, the ramification at $\infty$ is determined by the splitting behaviour of $\tilde{f}(X,s) = s^2f(X/s,1/s)$ at $s = 0$; that is,
    \[
    \tilde{f}(X,0) = X^2 - a^+X + b.
    \]
        Hence $f$ is a Weil polynomial if and only if $\left( \frac{a^+,b}{q} \right) \neq 1$. Since $\deg(D) \leq 2$, the curve $Y^2 = D$ has genus 0, so its Jacobian has one point. Thus, Lemma~\ref{lem:jac} gives
    \[
    \# \text{Iso}_{\mathfrak p^n}(a,b) = H(D) = 1 - \chi_{D}(\infty) = 1 - \left( \frac{a^+,b}{q} \right).
    \]
        \item[iii)] $D \in \bF_q$. Then $H(D) = 1$. On the other hand, the coefficient of $T^2$ in $D$ equals $(a^+)^2 - 4b = 0$. Hence $X^2 - a^+X + b$ is a square, so again $\# \text{Iso}_{\mathfrak p^n}(a,b) = 1 - \left( \frac{a^+,b}{q} \right)$.
    \end{enumerate}
    In case 2, let $f(X,T)$ be as before, except now $\wp \mid a$. We again consider subcases.
    \begin{enumerate}
        \item[i)] $d = 2$. Since $\deg(a) \leq 1$, this implies $a = 0$, so let $f(X) = X^2 - b \wp$ with $b \in \bF_q^\times$. Clearly $\left( \frac{0,b}{q} \right) = b^{(q-1)/2} = \chi_D(\infty)$; in particular, if $f$ is a Weil polynomial, then $\infty$ is inert.
    Since $Y^2 = \mu \wp$ is a curve of genus~0, we obtain
    \[
    \# \text{Iso}_{\mathfrak p}(0,b) = H(b \wp) = 1 - \chi_D(\infty) = 1 - \left( \frac{0,b}{q} \right).
    \]
    \item[ii)] $d = 1$. Then $f(X,T) = X^2 - a^+\wp X + b\wp^2$ with $a^+ \in \bF_q$ and $b \in \bF_q^\times$, and the splitting behaviour of $f$ is equivalent to the splitting behaviour of $X^2 - a^+X + b \in \bF_q[X]$. In particular, if $X^2 - a^+X + b$ splits, then $f(X,T)$ is not a Weil polynomial as it does not occur as a case in Prop.~\ref{prop:yu_pols}. Comparing each of the possibilities for $\left( \frac{a^+,b}{q} \right)$ with the corresponding case in Prop.~\ref{prop:isom} gives $\# \text{Iso}_{\mathfrak p^n}(a,b) = 1 - \left( \frac{a^+,b}{q} \right)$ every time.\qedhere
    \end{enumerate} 
\end{proof}

\begin{lemma}\label{lem:deg2sum}
    For $1 \leq m \leq q-1$, we have
\[
    \sum_{\alpha, \beta \in \bF_q^\times} \left(\frac{\alpha, \beta}{q} \right) (\alpha^{-2}\beta)^m \equiv (-1)^{\frac{q-1}{2}} 4^{-m} \binom{m}{(q-1)/2} \pmod{p}.
\]
\end{lemma}

\begin{proof}
    Note that for any $\alpha, \beta \in \bF_q^\times$, we have
    \[
    \left( \frac{\alpha,\beta}{q} \right) = (\alpha^2 - 4\beta)^{\frac{q-1}{2}} = (1 - 4\alpha^{-2}\beta)^{\frac{q-1}{2}},
    \]
    and hence
    \[
    \sum_{\alpha, \beta \in \bF_q^\times} \left( \frac{\alpha,\beta}{q} \right) (\alpha^{-2}\beta)^m = - \sum_{\gamma \in \bF_q^\times} (1 - 4 \gamma)^{\frac{q-1}{2}}\gamma^m = -\sum_{i=0}^{(q-1)/2} \binom{(q-1)/2}{i}(-4)^i \sum_{\gamma \in \bF_q^\times} \gamma^{m+i}.
    \]
    The final sum is zero unless $m + i \equiv  0 \pmod{q-1}$, in which case it contributes a factor of~$-1$. This gives
    \[
    \sum_{\alpha, \beta \in \bF_q^\times} \left(\frac{\alpha, \beta}{q} \right) (\alpha^{-2}\beta)^m = (-4)^{-m} \binom{(q-1)/2}{q-1-m},
    \]
    and the latter equals the claimed expression by Eq.~\eqref{eq:mattarei}.
\end{proof}

\begin{proof}[Proof of Theorem~\ref{thm:deg2}]
    Combining \hyperref[prop:trace_formula]{the trace formula} with Prop.~\ref{prop:deg2} gives
    \[
    \Tr(\T^n_{\mathfrak p} \hspace{0.2em} | \hspace{0.1em} \S_{k+2,l}) = \sum_{j=0}^{\lceil k/2 \rceil - 1} c_{k,j} \sum_{a^+ \in \bF_q} \sum_{a^- \in \bF_q} \sum_{b \in \bF_q^\times} \left( 1 - \left(\frac{a^+,b}{q}\right)\right) (a^+T + a^-)^{k-2j} b^{j-l+1} \wp^{nj}.
    \]
    Expanding $(a^+T + a^-)^{k-2j}$ via the binomial theorem and simplifying gives
    \[
    \Tr(\T^n_{\mathfrak p} \hspace{0.2em} | \hspace{0.1em} \S_{k+2,l}) = \sum_{j=0}^{\lceil k/2 \rceil - 1} \sum_{\substack{i=0 \\ i \equiv k-2j}}^{k-2j-1} (-1)^{j+1} \binom{k-j}{i,j,k-2j-i} s_{i,j} T^i \wp^{nj},
    \]
    where
    \begin{align*}
    s_{i,j} &= \sum_{a^+ \in \bF_q} \sum_{b \in \bF_q^\times} \left( 1 - \left(\frac{a^+,b}{q}\right)\right) (a^+)^{i} b^{j-l+1} \\
    &= \sum_{a^+,b \in \bF_q^\times} \left( 1 - \left(\frac{a^+,b}{q}\right)\right) (a^+)^{i} b^{j-l+1} + \sum_{b \in \bF_q^\times} \left( 1 - \left(\frac{0,b}{q}\right)\right) 0^{i} b^{j-l+1}.
    \end{align*}
    Write $m := j-l+1$. Since $i \equiv k - 2j \equiv -2m \pmod{q-1}$, the first sum can be computed via Lemma~\ref{lem:deg2sum}. In particular, if $i = 0$, we find that
    \[
    s_{0,j} = \begin{cases}
        -1 & \text{if } j \equiv l-1 \pmod{q-1}; \\ 
        0 & \text{if } j \equiv l-1 + (q-1)/2 \pmod{q-1}.
    \end{cases}
    \]
    If $i > 0$, the second sum vanishes. Let $[m]$ denote the representative of~$m$ modulo~$q-1$ satisfying $1 \leq [m] \leq q-1$. Then Lemma~\ref{lem:deg2sum} yields
    \[
    s_{i,j} = \epsilon_{j} + (-1)^{1 + \frac{q-1}{2}} 4^{-m} \binom{[m]}{(q-1)/2},
    \]
    where $\epsilon_{j} = 1$ if $j \equiv l-1 \pmod{q-1}$ and $\epsilon_{j} = 0$ otherwise. In particular, $s_{i,j} = 0$ if $i > 0$ and $j \equiv l-1 \pmod{q-1}$. Note also that the binomial coefficient vanishes whenever $2[m] < q-1$. Combining the above yields the claimed formula.
\end{proof}

It is worth noting that $\# \text{Iso}_{\mathfrak p^n}(a,b)$ is independent of $\mathfrak p$ and~$n$ as long as $nd = 2$.

\begin{corollary}\label{cor:deg2pol}
    Let $k \geq 0$ and $l \in \mathbb Z$. Then there exists a polynomial $f_{k+2,l}(X) \in A[X]$ of degree $\deg (f_{k+2,l}) < k/2$ such that
    \[
    f_{k+2,l}(\wp^n) = \Tr(\T_{\wp}^n \hspace{0.2em} | \hspace{0.1em} \S_{k+2,l})
    \]
    for every pair $(\wp,n)$ such that $\wp \in \bF_q[T]$ is a monic irreducible polynomial of degree $d$ and $nd = 2$.
\end{corollary}

\begin{proof}
    If $2 \mid q$, this is obvious from Theorem~\ref{thm:char2traces}. If $2 \nmid q$, then the polynomial in question is given by
    \[
    f_{k+2,l}(X) = \sum_{0 \leq j < k/2} c_{k,j} \sum_{\deg(a) \leq 1} \sum_{b \in \bF_q^\times} \left(1 -\left( \frac{a^+,b}{q} \right) \right) a^{k-2j}b^{j+l-k-1}X^j,
    \]
    which equals the expression from Theorem~\ref{thm:deg2} with $\wp^n$ replaced by~$X$.
\end{proof}

\begin{remark}
    Since there are $q(q+1)/2$ pairs $(\wp,n)$ such that $nd=2$, Corollary~\ref{cor:deg2pol} trivially holds when $k \geq q(q+1)$. When $\dim \S_{k+2,l} = 1$, however, the polynomial $f_{k+2,l}$ is unique and its existence is not a priori obvious.
\end{remark}

See Section~\ref{sec:comp_deg2} for an analysis of the polynomials $f_{k+2,l}$ in some special cases.

\subsection{Odd characteristic: primes of higher degree}

Suppose we want to compute traces when $nd > 2$. It gets increasingly harder to obtain explicit formulas, for example because the Hurwitz class number $H(D)$ is now related to the number of points on a Jacobian of positive dimension (see Lem.~\ref{lem:jac}).
Instead, we will combine the results from this paper into an algorithm to compute the trace of~$\T_{\mathfrak p}^n$ on~$\S_{k,l}$ in \textsc{Magma} \cite{magma_system}.

\begin{algorithm}\label{alg}\hfill \\
\textsc{Input:} \begin{itemize} \item An odd prime power $q$;
\item Integers $k,l,n \in \mathbb Z$ with $k \geq 2$ and $n \geq 1$;
\item A monic irreducible polynomial $\wp \in \bF_q[T]$.
\end{itemize}
\textsc{Output:} $\Tr(\T_{\mathfrak p}^n \hspace{0.2em} | \hspace{0.1em} \S_{k,l})$ for $A = \bF_q[T]$ and $\mathfrak p = (\wp)$.
\begin{enumerate}
\item Using Propositions~\ref{prop:charpol_properties} and~\ref{prop:yu_pols}, create a list $L_0$ containing all the characteristic polynomials of Drinfeld modules of rank~2 over~$\bF_{\mathfrak p^n}$. For instance, one could take
\[
L_0 = \left \{ X^2 - aX + b\wp^n \ \bigg{|} \ \deg(a) \leq \frac{n \deg(\wp)}{2}, \ b \in \bF_q^\times \right \}.
\]
\item Since \textsc{Magma} can determine whether or not a field extension is imaginary \cite[Sec.~42.13]{magma_handbook}, one can algorithmically find the sublist $L \subseteq L_0$ of Weil polynomials.
\item Since \textsc{Magma} can compute the class number of the finite maximal order in an imaginary quadratic extension of function fields \cite[Sec.~42.2]{magma_handbook}, as well as compute how a prime ramifies in an extension, one can use Prop.~\ref{prop:isom} and Lem.~\ref{lem:hurwitz_number} to compute $\# \text{Iso}_{\mathfrak p^n}(a,b)$ for all polynomials $f = X^2 - aX + b\wp^n \in L$.
\item Set $k' = k-2$ and compute
\[
H = \sum_{f \in L} \# \text{Iso}_{\mathfrak p^n}(a,b) \sum_{0 \leq j < k'/2} (-1)^j \binom{k'-j}{j} a^{k'-2j} b^{j+l-k'+1} \wp^{nj}.
\]
By Prop.~\ref{prop:trace_formula}, $H$ is the desired Hecke trace.
\end{enumerate}
\end{algorithm}

An implementation of Algorithm~\ref{alg} in \textsc{Magma} can be found at~\url{https://github.com/Sjoerd-deVries/DMF_Trace_Formula.git}. The results in the appendix (especially Rmk.~\ref{rmk:point_counts}) give an alternative way to algorithmically determine the numbers $\# \text{Iso}_{\mathfrak p^n}(a,b)$ in terms of point counts on hyperelliptic curves.

\begin{example}
    Figure~\ref{fig:deg5traces} shows the degrees of the traces of the Hecke operator associated to $\wp = T^5 + 2T + 1 \in \bF_3[T]$ for $l = 0$ and $2 \leq k \leq 200$ obtained using Algorithm~\ref{alg}. The dotted line denotes the strong Ramanujan bound from Conj.~\ref{conj:strong_ram_traces}, in this case given by $\deg = \frac{5(k-4)}{2}$. The computation of these traces in \textsc{Magma} took 4.8~seconds on a standard laptop.
        \begin{figure}
        \centering
        \includegraphics[width=\textwidth]{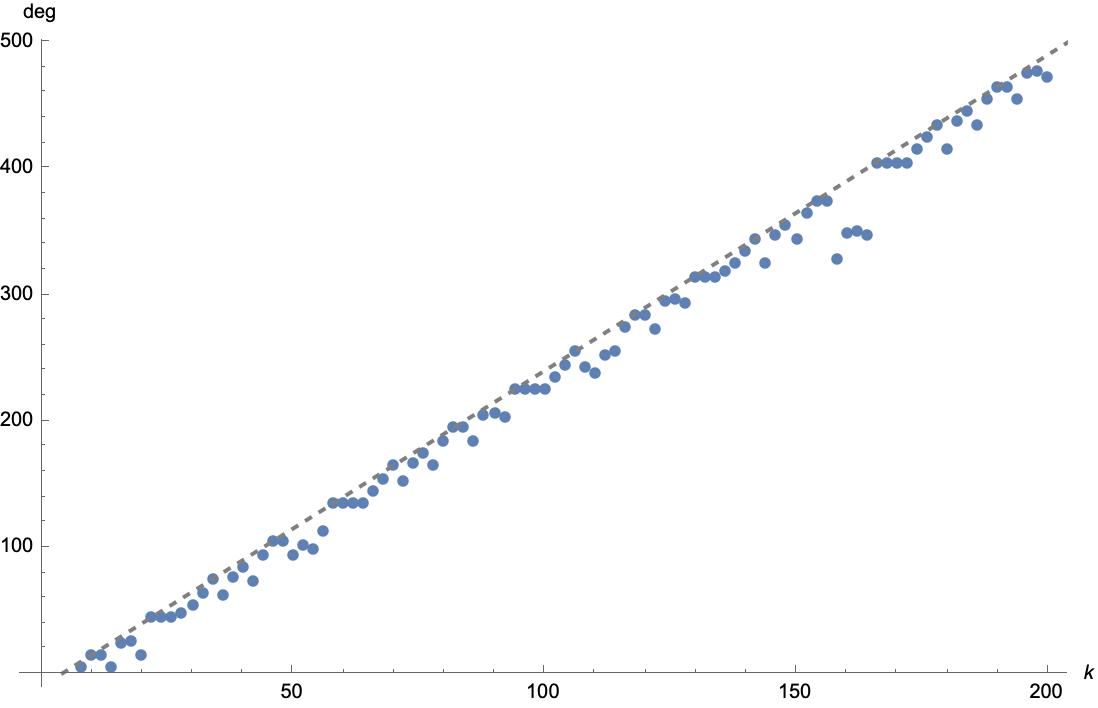}
        \caption{$\deg \Tr(\T_{T^5+2T+1} \hspace{0.1em} | \hspace{0.1em} \S_{k,0})$ for $q=3$ and $2 \leq k \leq 200$ with the strong Ramanujan bound.}
        \label{fig:deg5traces}
        \end{figure}
\end{example}

\subsection{Traces modulo \texorpdfstring{$\mathfrak p$}{P}}
Let $\mathfrak p \trianglelefteq A$ be a maximal ideal with monic generator~$\wp$ of degree~$d$. The trace formula for~$\T_{\mathfrak p}$ becomes much simpler modulo~$\mathfrak p$.

\begin{proposition}\label{prop:modp}
    For any $k \geq 0$, $n \geq 1$, and $l \in \mathbb Z$, we have
    \[
    \Tr(\T^n_{\mathfrak p} \hspace{0.2em} | \hspace{0.1em} \S_{k+2,l}) \equiv \sum_{a,b} \# \text{Iso}_{\mathfrak p^n}(a,b) a^{k} b^{l-k-1} \pmod{\mathfrak p^n}.
    \]
\end{proposition}

\begin{proof}
    Modulo $\mathfrak p^n$, all terms with $j > 0$ in Equation~\eqref{eq:trace_formula_2} are zero.
\end{proof}

\begin{lemma}\label{lem:gcdmodp}
    If $k \geq 2$, we have
    \[
    \Tr(\T^n_{\mathfrak p} \hspace{0.2em} | \hspace{0.1em} \S_{k+2,l}) \equiv \sum_{(a,\wp) = 1} \sum_{b \in \bF_q^\times} \# \text{Iso}_{\mathfrak p^n}(a,b) a^{k} b^{l-k-1} \pmod{\mathfrak p^n},
    \]
    i.e., only the ordinary Drinfeld modules contribute.
\end{lemma}

\begin{proof}
    In each of the cases 2, 3 and 4 of Prop.~\ref{prop:yu_pols}, we have $a=0$ if $n$ is odd or $\wp^{n/2} \mid a$ if $n$ is even; in either case, $\wp^n \mid a^k$. Hence only the Weil polynomials from case 1 contribute, which are precisely the ones for which $(a,\wp) = 1$.
\end{proof}

\begin{proposition}\label{prop:congruence}
For any $k \geq 2$, $n \geq 1$, and $l \in \mathbb Z$, we have
\[
    \Tr(\T^n_{\mathfrak p} \hspace{0.2em} | \hspace{0.1em} \S_{k+2,l}) \equiv \Tr(\T^n_{\mathfrak p} \hspace{0.2em} | \hspace{0.1em} \S_{k+(q^d-1) + 2,l}) \pmod{\mathfrak p}.
\]
Moreover, let $e \in \mathbb Z_{\geq 0}$ and let $m := \min(q^e,n)$.
Then we have
\[
    \Tr(\T^n_{\mathfrak p} \hspace{0.2em} | \hspace{0.1em} \S_{k+2,l}) \equiv \Tr(\T^n_{\mathfrak p} \hspace{0.2em} | \hspace{0.1em} \S_{k+q^e(q^d-1) + 2,l}) \pmod{\mathfrak{p}^{m}}.
\]
\end{proposition}

\begin{proof}
The first congruence follows from the second by setting $e = 0$. By Prop.~\ref{prop:modp}, it suffices to show that $a^k \equiv a^{k+q^e(q^d-1)} \pmod{\mathfrak p^m}$ for any $a \in A$, as it is clear that $b^{q^e(q^d-1)} = 1$. This follows because $A/\mathfrak p^m \cong \bF_{q^d}[\wp]/(\wp^m)$ and in this ring, we have
\[
f = \sum_{i=0}^{m-1} f_i\wp^i \implies f^{q^e} = \sum_{i=0}^{m-1} f_i^{q^e}\wp^{iq^e} = f_0^{q^e} \in \bF_{q^d},
\]
since $q^e \geq m$. Hence, if $f_0 \neq 0$, we have $f^{q^e(q^d-1)} = 1$. But for $a \in A$ with image $f \in A/{\mathfrak p^m}$, we have $f_0 \neq 0 \iff (a,\wp) = 1$, so we are done by Lemma~\ref{lem:gcdmodp}. 
\end{proof}

\begin{proposition}\label{prop:frob_modp}
    For any $k \geq 1$, $n \geq 1$, and $l \in \mathbb Z$, we have
    \[
    \Tr(\T^n_{\mathfrak p} \hspace{0.2em} | \hspace{0.1em} \S_{kq+2,l}) \equiv \Tr(\T^n_{\mathfrak p} \hspace{0.2em} | \hspace{0.1em} \S_{k+2,l})^q \pmod{\mathfrak p^n}.
    \]
\end{proposition}

\begin{proof}
By Prop.~\ref{prop:modp}, we obtain
    \begin{align*}
    \Tr(\T^n_{\mathfrak p} \hspace{0.2em} | \hspace{0.1em} \S_{kq+2,l}) &\equiv \sum_{a,b} \# \text{Iso}_{\mathfrak p^n}(a,b) a^{kq} b^{l-kq-1}\\
    &\equiv \sum_{a,b} \# \text{Iso}_{\mathfrak p^n}(a,b)^q a^{kq} b^{(l-k-1)q} \equiv \Tr(\T^n_{\mathfrak p} \hspace{0.2em} | \hspace{0.1em} \S_{k+2,l})^q \pmod{\mathfrak p^n},
    \end{align*}
as claimed.
\end{proof}

Recall that there is a Frobenius map on modular forms. It is given by
\[
F_q\: \S_{k,l} \to \S_{kq,l}, \qquad F_q(f) = f^q.
\]
Frobenius sends Hecke eigenforms to Hecke eigenforms: more precisely, if $\T_{\mathfrak p}f = \lambda f$, then $\T_{\mathfrak p}f^q = \wp^{q-1}\lambda^qf^q$. This leads to the following congruences.

\begin{proposition}\label{prop:frob2_modp}
    If $k \geq 4$ is such that $q^d-1 \mid k(q-1)$, then for any $n \geq 1$ and $l \in \mathbb Z$, we have
    \[
    \Tr(\T^n_{\mathfrak p} \hspace{0.2em} | \hspace{0.1em} \S_{kq,l}) \equiv \Tr(\T^n_{\mathfrak p} \hspace{0.2em} | \hspace{0.1em} \S_{kq,l} / F_q(\S_{k,l})) \equiv \Tr(\T^n_{\mathfrak p} \hspace{0.2em} | \hspace{0.1em} \S_{k,l}) \pmod{\mathfrak p}.
    \]
\end{proposition}

\begin{proof}
    Since $F_q(\S_{k,l}) \subseteq \S_{kq,l}$ is Hecke-stable, we have
    \[
    \Tr(\T^n_{\mathfrak p} \hspace{0.2em} | \hspace{0.1em} \S_{kq,l}) = \Tr(\T^n_{\mathfrak p} \hspace{0.2em} | \hspace{0.1em} F_q(\S_{k,l})) + \Tr(\T^n_{\mathfrak p} \hspace{0.2em} | \hspace{0.1em} \S_{kq,l} / F_q(\S_{k,l})).
    \]
    Since $\T_{\mathfrak p}f^q = \wp^{q-1}\lambda^qf^q$ if $\T_{\mathfrak p}f = \lambda f$, we have
    \[
    \Tr(\T^n_{\mathfrak p} \hspace{0.2em} | \hspace{0.1em} F_q(\S_{k,l})) = \wp^{q-1} \Tr(\T^n_{\mathfrak p} \hspace{0.2em} | \hspace{0.1em} \S_{k,l})^q \equiv 0 \pmod{\mathfrak p}.
    \]
    This gives the first congruence. On the other hand, by Prop.~\ref{prop:congruence} and the assumption on~$k$, we have
    \[
    \Tr(\T^n_{\mathfrak p} \hspace{0.2em} | \hspace{0.1em} \S_{kq,l}) = \Tr(\T^n_{\mathfrak p} \hspace{0.2em} | \hspace{0.1em} \S_{k+k(q-1),l}) \equiv \Tr(\T^n_{\mathfrak p} \hspace{0.2em} | \hspace{0.1em} \S_{k,l}) \pmod{\mathfrak p}.\qedhere
    \]
\end{proof}
\section{The strong Ramanujan bound}\label{sec:ram}

The Ramanujan bound from~Def.~\ref{def:ram_bd} states that for any maximal ideal $\mathfrak p$ and integers $n \geq 1$, $k \geq 0$, and $l \in \mathbb Z$,
\begin{equation}\label{eq:ram_bound}
\deg \Tr \left( \T_{\mathfrak p}^n \hspace{0.2em} | \hspace{0.1em} \S_{k+2,l} \right) \leq \frac{ndk}{2},
\end{equation}
where $d = \deg(\mathfrak p)$. In level 1, this bound is not sharp (see Prop.~\ref{prop:ram_bd_strict}).
Based on much computational evidence, we propose an improved bound (the strong Ramanujan bound) in the form of Conjecture~\ref{conj:strong_ram_traces}, which we prove in some special cases (see Thm.~\ref{thm:strong_ram_known}).

\begin{proposition}\label{prop:ram_bd_strict}
    For any $k,l,n$ and $\mathfrak p$, the Ramanujan bound \eqref{eq:ram_bound} is a strict inequality.
\end{proposition}

\begin{proof}
    Denote by $\Tr(\T_{\mathfrak p}^n \hspace{0.2em} | \hspace{0.1em} \S_{k+2,l})^+$ the term of degree $ndk/2$ in $\Tr(\T_{\mathfrak p}^n \hspace{0.2em} | \hspace{0.1em} \S_{k+2,l})$; we will show that it is zero. Given a Weil polynomial $X^2 - aX + b\wp^n \in \mathbb F_q[T][X]$, write 
    \[
    a = a^+T^{nd/2} + \text{lower order terms}.
    \]
    Then we have by \hyperref[prop:trace_formula]{the trace formula} and the fact that $\wp$ is monic,
    \[
    \Tr(\T_{\mathfrak p}^n \hspace{0.2em} | \hspace{0.1em} \S_{k+2,l})^+ = \sum_{\deg(a) = nd/2} \sum_{b \in \bF_q^\times} \# \text{Iso}_{\mathfrak p^n}(a,b) \sum_{j=0}^{\lceil k/2 \rceil - 1} c_{k,j} ( a^+T^{nd/2} )^{k-2j} b^{j+l-k-1} T^{ndj}.
    \]
    The summand corresponding to a given $j$ is thus
    \[
    c_{k,j} T^{ndk/2} \sum_{\deg(a) = nd/2} \sum_{b \in \bF_q^\times} \# \text{Iso}_{\mathfrak p^n}(a,b) (a^+)^{k-2j} b^{j+l-k-1}.
    \]
    Note that $a^+ \in \bF_q^\times$ for all $a$ appearing in the sum above, so it suffices to show that
    \begin{equation}\label{eq:sum0}
    \sum_{\deg(a) = nd/2} \sum_{b \in \bF_q^\times} \# \text{Iso}_{\mathfrak p^n}(a,b) (a^+)^{m} b^{t} = 0
    \end{equation}
    for all integers $1 \leq m \leq q-1$ and all $t \in \mathbb Z$. We do this by induction. Note that $\S_{m+2,t}$ contains no double cusp forms for $m \leq q-1$ and any $t \in \mathbb Z$, so these spaces have trivial Hecke action by Prop.~\ref{prop:a-exp_examples}. In particular, we have $\Tr(\T_{\mathfrak p}^n)^+ = 0$ in the entire range of $m$ and~$t$ we consider.
    
    The base cases are $m=1$ and $m=2$. Then
    \begin{align*}
    0 &= \Tr(\T_{\mathfrak p}^n \hspace{0.2em} | \hspace{0.1em} \S_{3,t})^+ = T^{nd} \sum_{\deg(a) = nd/2} \sum_{b \in \bF_q^\times} \# \text{Iso}_{\mathfrak p^n}(a,b) a^+ b^{t-2};\\
    0 &= \Tr(\T_{\mathfrak p}^n \hspace{0.2em} | \hspace{0.1em} \S_{4,t})^+ = T^{nd} \sum_{\deg(a) = nd/2} \sum_{b \in \bF_q^\times} \# \text{Iso}_{\mathfrak p^n}(a,b) (a^+)^{2} b^{t-3},
    \end{align*}
    which gives \eqref{eq:sum0} for $m \in \{1,2\}$ and any $t$.
    
    Next, let $m' \leq q-1$ and suppose~\eqref{eq:sum0} holds for all $m = 1,2,\ldots, m'-2$ and any $t$. Then
    \begin{align*}
    0 &= \Tr(\T_{\mathfrak p}^n \hspace{0.2em} | \hspace{0.1em} \S_{m'+2,t})^+ \\
    &= T^{ndm'/2} \sum_{\deg(a) = nd/2} \sum_{b \in \bF_q^\times} \# \text{Iso}_{\mathfrak p^n}(a,b) \left( (a^+)^{m'} b^{t-m'-1} + \sum_{j=1}^{\lceil m'/2 \rceil -1} c_{m',j} (a^+)^{m'-2j}b^{j+t-m'-1} \right) \\
    &= T^{ndm'/2} \sum_{\deg(a) = nd/2} \sum_{b \in \bF_q^\times} \# \text{Iso}_{\mathfrak p^n}(a,b) (a^+)^{m'} b^{t-m'-1},
    \end{align*}
    which gives \eqref{eq:sum0} for $m = m'$ and any $t$.
\end{proof}

\begin{remark}
    Assuming the Tate conjecture, Coleman and Edixhoven have shown that the Ramanujan bound on elliptic modular forms is likewise never attained \cite[Thm.~4.1]{coleman-edixhoven}.
\end{remark}

\begin{conjecture}[Strong Ramanujan bound]\label{conj:strong_ram_traces}
    Let $\mathfrak p \trianglelefteq A$ be a maximal ideal of degree~$d$ and let $n \geq 1$. Then for any $k,l \in \mathbb Z$,
    \begin{equation}\label{eq:strong_ram}
    \deg \Tr(\T_{\mathfrak p}^n \hspace{0.2em} | \hspace{0.1em} \S_{k,l}) \leq \frac{nd(k-(q+1))}{2}.
    \end{equation}
\end{conjecture}

Throughout the remainder of this paper, we will say that \emph{the strong Ramanujan bound holds for $\mathfrak p^n$} if the inequality~\eqref{eq:strong_ram} holds for all $k,l \in \mathbb Z$.

\begin{remark}\label{rem:strongram}\hfill
    \begin{enumerate}
        \item If $nd = 1$, the strong Ramanujan bound is sharp infinitely often for any type. To see this, fix a type~$1 \leq l \leq q-1$ and put $k^{(l)}_{n} := q-1+2q(l-1+n(q-1))$ for $n \geq 0$. Then $k_n^{(l)} + 2 \equiv 2l \pmod{q-1}$, and the trace of~$\T_{T}$ on the space $\S_{k^{(l)}_{n}+2,l}$ can be computed using Theorem~\ref{thm:deg1}. In particular, by \hyperref[thm:lucas]{Lucas's theorem}, the trace has leading term $T^{q(l-1+n(q-1))}$ with coefficient
    \begin{align*}
    c_{k^{(l)}_{n},q(l-1+n(q-1))} &= (-1)^{q(l-1+n(q-1))} \binom{q-1 + q(l - 1 + n(q-1))}{q(l-1 + n(q-1))} \\ &\equiv_p (-1)^{l-1} \binom{l-1 + n(q-1)}{l-1 + n(q-1)} \binom{q-1}{0} = (-1)^{l-1} \neq 0.
    \end{align*}
        \item With notation as above, we have $\S_{k_0^{(l)}+2,l} = \langle g^{l-1}h^l \rangle$. The eigenforms $g^{l-1}h^l$ are the first forms to attain the strong Ramanujan bound for~$\T_{T}$ in each type. In other words, the strong Ramanujan bound is not attained for $\T_T$ acting on $\S_{k'+2,l}$ for any $k' < k^{(l)}_{0}$. 
    \item In the other direction, Corollary~\ref{cor:clusters} shows that the distance of a trace to the strong Ramanujan bound is unbounded as the weight goes to infinity.
    \end{enumerate}
\end{remark}

Recall that $\# \text{Iso}_{\mathfrak p^n}(a,b) = 0$ whenever $\deg(a) > \lfloor \frac{nd}{2} \rfloor =: N$. If $X^2 - aX + b\wp^n$ is the Weil polynomial of a Drinfeld module over $\bF_{\mathfrak p^n}$, write
\[
a = a_0 + a_1T + \ldots + a_NT^N, \qquad a_i \in \bF_q.
\]

\begin{proposition}\label{prop:eq_cond}
    Fix $\mathfrak p$ and $n$, and let $N = \lfloor nd/2 \rfloor$. The following are equivalent:
    \begin{enumerate}
        \item The strong Ramanujan bound holds for $\mathfrak p^n$.
        \item For all $k \geq 0$ and $t \in \mathbb Z$, we have
        \begin{equation}\label{eq:equiv_strong}
        \deg \sum_{a,b}\# \text{Iso}_{\mathfrak p^n}(a,b) a^kb^t \leq \frac{nd(k-(q-1))}{2}.
        \end{equation}
        \item For all $k \geq 0$ and $t \in \mathbb Z$, we have
        \[
        \sum_{a,b} \# \text{Iso}_{\mathfrak p^n}(a,b)b^t \sum_{\substack{v_0 + \ldots + v_N = k \\ v_1 + 2v_2 + \ldots + Nv_N = m}} \binom{k}{v_0,v_1,\ldots,v_N} a_0^{v_0}a_1^{v_1} \cdots a_N^{v_N} = 0
        \]
        for all $m$ such that $2m > nd(k-(q-1))$.
    \end{enumerate}
\end{proposition}

\begin{proof}
    We use the formula
    \begin{equation}\label{eq:trform_proof}
    \Tr(\T^n_{\mathfrak p} \hspace{0.2em} | \hspace{0.1em} \S_{k+2,l}) = \sum_{a,b} \# \text{Iso}_{\mathfrak p^n}(a,b)  \sum_{j=0}^{\lceil k/2 \rceil - 1} c_{k,j} a^{k-2j} b^{j+l-k-1}\wp^{nj}
    \end{equation}
    from Prop.~\ref{prop:trace_formula}. This immediately yields $2 \implies 1$, since the summand corresponding to a fixed~$j$ then has degree bounded by
    \[
    \frac{nd(k-2j-(q-1))}{2} + ndj = \frac{nd(k+2-(q+1))}{2},
    \]
    as desired.
    
    For $1 \implies 2$, note that~\eqref{eq:equiv_strong} holds for $k \in \{0,1\}$ and any~$t$, since there are no double cusp forms of weight~$2$ or~$3$ for any~$q$. Suppose by induction that \eqref{eq:equiv_strong} holds for any $k' < k$ and any $t$. Then~\eqref{eq:trform_proof} implies that, for any $l \in \mathbb Z$,
    \[
    \sum_{a,b} \# \text{Iso}_{\mathfrak p^n}(a,b) a^k b^{l-k-1} =  \Tr(\T^n_{\mathfrak p} \hspace{0.2em} | \hspace{0.1em} \S_{k+2,l}) + O \left( T^{\frac{1}{2}nd(k-(q-1))} \right),
    \]
    which by the strong Ramanujan bound implies \eqref{eq:equiv_strong} for $k$ and any $t$.

    The equivalence $2 \iff 3$ follows from the computation
\begin{align*}
\sum_{a,b}\# \text{Iso}_{\mathfrak p^n}(a,b) a^kb^t &= \sum_{a,b} \# \text{Iso}_{\mathfrak p^n}(a,b) \left( \sum_{i=0}^N a_iT^i \right)^k b^t \\
&= \sum_{m = 0}^{kN} T^m \sum_{a,b} \# \text{Iso}_{\mathfrak p^n}(a,b) b^t \sum \binom{k}{v_0,v_1,\ldots,v_N} a_0^{v_0} a_1^{v_1} \cdots a_N^{v_N},
\end{align*}
where the last sum is over $(N+1)$-tuples $(v_0,\ldots,v_N) \in \mathbb Z_{\geq 0}^{N+1}$ such that $\sum v_i = k$ and $\sum iv_i = m$.
\end{proof}

\begin{corollary}\label{cor:nd_odd}
    Suppose $nd$ is odd. Then the strong Ramanujan bound holds for $\mathfrak p^n$ if it holds for all pairs $(k,t)$ with $k < nd(q-1)$ and $1 \leq t \leq q-1$.
\end{corollary}

\begin{proof}
    If $nd$ is odd, then $N = (nd-1)/2$. Hence $\deg a^kb^t \leq k(nd-1)/2$ for any pair $(a,b)$ such that $\# \text{Iso}_{\mathfrak p^n}(a,b) \neq 0$. Now if $k \geq nd(q-1)$, we have
    \[
    \deg \sum_{a,b}\# \text{Iso}_{\mathfrak p^n}(a,b) a^kb^t \leq \frac{k(nd-1)}{2} \leq \frac{nd(k-(q-1))}{2},
    \]
    so by Prop.~\ref{prop:eq_cond}, it only remains to check $k < nd(q-1)$.
\end{proof}


More generally, we have the following sufficient condition for the Ramanujan bound to hold, which for each fixed $\mathfrak p^n$ is a finite computation. It is worth mentioning that we have not found any examples where this sufficient condition is not satisfied.

\begin{proposition}\label{prop:suff_cond}
    Fix $\mathfrak p \trianglelefteq A$ and $n \geq 1$. Suppose that
    \begin{equation}\label{eq:sumcond}
    \sum_{a,b} \# \text{Iso}_{\mathfrak p^n}(a,b) b^t a_0^{v_0} a_1^{v_1} \cdots a_N^{v_N} = 0
    \end{equation}
    under the following conditions:
    \begin{enumerate}
        \item $0 \leq v_i \leq q-1$ for all $i = 0,\ldots, N$;
        \item $2 \sum iv_i > nd(k-(q-1))$, where $k := \sum v_i$;
        \item $1 \leq t \leq q-1$ and $2t \equiv -k \pmod{q-1}$.
    \end{enumerate}
Then the strong Ramanujan bound holds for $\mathfrak p^n$.
\end{proposition}

\begin{proof}
    By the third statement in Prop.~\ref{prop:eq_cond}, it suffices to show that Equation~\eqref{eq:sumcond} holds for all $k := \sum v_i \geq 0$, $2\sum iv_i > nd(k-(q-1))$ and $t \in \mathbb Z$. Since $b \in \bF_q^\times$, we may assume $1 \leq t \leq q-1$. Since $\# \text{Iso}_{\mathfrak p^n}(a,b) = \# \text{Iso}_{\mathfrak p^n}(ca,c^2b)$ for all $c \in \bF_q^\times$, the sum is automatically zero if $2t \not \equiv -k \pmod{q-1}$. Note also that all~$a_i$ lie in~$\bF_q$ and hence $a_i^{s} = a_i^{s-(q-1)}$ for all $s \geq q$. Hence if some $v_j \geq q$, we can apply the above identity to reduce to the case $k' = k - (q-1)$, where we note that
    \[
    2 \sum_{i=1}^N iv_i > nd(k-(q-1)) \implies 2 \sum_{i=1}^N iv_i - 2j(q-1) > nd(k'-(q-1)),
    \]
    since $2j \leq nd$. One may keep doing this until $0 \leq v_i \leq q-1$ for all $i$, in which case we are in the setting of the proposition.
\end{proof}

\begin{remark}\label{rmk:rambd_case0}
    The condition $2 \sum iv_i > nd(k-(q-1))$ implies that $v_0 < q-1$. Considering the term of lowest degree in $\Tr(\T_{\mathfrak p}^n \hspace{0.2em} | \hspace{0.1em} \S_{k,l})$ for all $k \leq q+1$, we find that for any $k, t \in \mathbb Z$,
    \[
    \sum_{a,b} \# \text{Iso}_{\mathfrak p^n}(a,b) a_0^k b^t = \begin{cases}
        1 & \text{if } k \neq 0,\ k \equiv t \equiv 0 \pmod{q-1}; \\ 0 & \text{otherwise.}
    \end{cases}
    \]
    This settles the cases of Prop.~\ref{prop:suff_cond} in which $v_i = 0$ for all $i > 0$.
\end{remark}

\begin{theorem}\label{thm:strong_ram_known}
    The strong Ramanujan bound holds for all $\mathfrak p^n$ such that $nd \leq 3$. If $2 \mid q$, the strong Ramanujan bound holds for any~$\mathfrak p^n$.
\end{theorem}

\begin{proof}
    If $nd = 1$ we have $N = 0$, so the claim follows from Rmk.~\ref{rmk:rambd_case0}. If $nd = 2$, then by Prop.~\ref{prop:deg2} we can write each sum in Prop.~\ref{prop:suff_cond} as
    \[
    \sum_{a,b} \# \text{Iso}_{\mathfrak p^n}(a,b)b^ta_0^{v_0}a_1^{v_1} = \sum_{a_0 \in \bF_q} a_0^{v_0} \sum_{a_1 \in \bF_q} \sum_{b \in \bF_q^\times} \left( 1 - \left(\frac{a_1,b}{q}\right)\right)b^ta_1^{v_1} = 0,
    \]
    since the condition $2 \sum iv_i > nd(k-(q-1))$ implies that $v_0 < q-1$.

    If $nd = 3$ we apply Corollary~\ref{cor:nd_odd}. Note that the spaces $\S_{k,l}$ for $k < 3(q-1)$ contain no double cusp forms if $q \leq 5$, so in this case we are done. If $q > 5$, there is only one double cusp form of weight $k < 3(q-1)$, namely $h^2 \in \S_{2(q+1),2}$. This cusp form has an $A$-expansion with $A$-exponent~2, which implies that
    \[
    \deg \Tr(\T_{\mathfrak p}^n \hspace{0.2em} | \hspace{0.1em} \S_{2(q+1),2}) = nd < \frac{nd(q+1)}{2}.
    \]
    Finally, suppose $2 \mid q$. By Theorem~\ref{thm:char2traces}, the trace of any Hecke operator behaves like the trace of~$\T_{T}$, so the result follows from the case $nd=1$.
\end{proof}

\section{Eigenvalues}\label{sec:slopes}

In positive characteristic~$p$, it is not obvious that trace methods are sufficiently powerful to make deductions about eigenvalues. The problem is the following: if the linear operator~$\T$ has an eigenvalue~$\alpha$ with algebraic multiplicity~$p$, then there is no way to recover~$\alpha$ from the sequence $(\Tr(\T^n))_{n \geq 1}$: the contribution of these $p$ eigenvalues to $\Tr(\T^n)$ is $p \alpha^n = 0$. 

This problem is studied in more detail in~\cite{devries_newton}. In many situations, one can recover information about the eigenvalues of~$\T$ if one assumes that no eigenvalue is repeated $p$ times, by which we mean that the algebraic multiplicity of any eigenvalue is not a positive multiple of~$p$.

Throughout this section, we fix the following notation.
    Let $V$ be a $K$-vector space and let $\T \: V \to V$ be a linear map. We denote by $\text{Sp}(\T) \subset \bar{K}$ the set of eigenvalues of $\T$. For any $\alpha \in \text{Sp}(\T)$, we write $m_\alpha \in \mathbb N$ for the algebraic multiplicity of $\alpha$, i.e.,
    \[
    m_\alpha = \max \{ n \in \mathbb N \ | \ (X-\alpha)^n \text{ divides } \det(\mathds{1}X - \T)\}.
    \]

\subsection{Detecting repeated eigenvalues}

We noted that any eigenvalue which is repeated $p$ times is impossible to recover from the sequence of traces. What the traces do detect is whether or not there is a repeated eigenvalue, assuming that the dimension of the space is known.

\begin{proposition}\label{prop:vandermonde_rep}
    Let $d = \dim \S_{k,l}$. The action of $\T_{\mathfrak p}$ on $\S_{k,l}$ has a repeated eigenvalue if and only if the matrix
    \[
    M = \begin{pmatrix}
        \dim \S_{k,l} & \Tr(\T_{\mathfrak p}) & \Tr(\T_{\mathfrak p}^2) & \cdots & \Tr(\T_{\mathfrak p}^{d-1}) \\
        \Tr(\T_{\mathfrak p}) & \Tr(\T_{\mathfrak p}^2) & \Tr(\T_{\mathfrak p}^3) & \cdots & \Tr(\T_{\mathfrak p}^{d}) \\
        \vdots & \vdots & & \ddots & \vdots \\
        \Tr(\T_{\mathfrak p}^{d-1}) & \Tr(\T_{\mathfrak p}^{d}) & \Tr(\T_{\mathfrak p}^{d+1}) & \cdots & \Tr(\T_{\mathfrak p}^{2d-2})
        \end{pmatrix}
    \]
    has zero determinant.
\end{proposition}

\begin{proof}
      Denote the eigenvalues of $\T_{\mathfrak p}$ by $\alpha_1,\ldots,\alpha_d \in \bar{K}$. For any $n \geq 0$, we have $\Tr(\T_{\mathfrak p}^n \hspace{0.2em} | \hspace{0.1em} \S_{k,l}) = \alpha_1^n + \ldots + \alpha_d^n$. Hence, if we let
     \[
        A = \begin{pmatrix}
        1 & \alpha_1 & \alpha_1^2 & \cdots & \alpha_1^{d-1} \\
        1 & \alpha_2 & \alpha_2^2 & \cdots & \alpha_2^{d-1} \\
        \vdots & \vdots & & \ddots & \vdots \\
        1 & \alpha_d & \alpha_d^2 & \cdots & \alpha_d^{d-1}
    \end{pmatrix},
     \]
     then $A^t A = M$. But $A$ is a Vandermonde matrix, so
     \[
    \det M = \det A^tA = \prod_{i<j} (\alpha_i - \alpha_j)^2.
     \]
     Thus $\det M = 0$ if and only if $\alpha_i = \alpha_j$ for some $i \neq j$.
\end{proof}

\begin{remark}\label{rmk:repeated_eigs}
    If $p=2$, the action of~$\T_{\mathfrak p}$ on~$\S_{k,l}$ almost always has repeated eigenvalues by Theorem~\ref{thm:char2_repetition}. If $p > 2$, we suspect that the action of~$\T_{\mathfrak p}$ on~$\S_{k,l}$ never has~$p$ repeated eigenvalues at level~1. At higher levels, though, there can be repeated eigenvalues for any~$q$ \cite{li-meemark,hattori_ordinary}. This is analogous to the classical setting, where Buzzard's conjecture \cite[Conj.~2.1]{maeda} predicts that each Hecke operator acts irreducibly on $\S_k(\operatorname{SL}_2(\mathbb Z))$ for any~$k$. This would in particular imply that eigenvalues are not repeated at level~1.
\end{remark}

\subsection{Injectivity}
In their study of oldforms and newforms, Bandini and Valentino conjectured that the Hecke operator $\T_T$ acting on $\S_{k,l}$ is always injective \cite[Conj.~1.1]{band-val_slopes}. This is related to the diagonalisability of the Atkin operator~$\mathbf{U}_T$ at level~$\Gamma_0(T)$. Since then, the conjecture has been proven in the case that the space of cusp forms has dimension~1 \cite{band-val_hecke,dalal-kumar}.

It turns out that a result by Joshi and Petrov \cite{joshi-petrov} on the structure of Hecke operators modulo a prime of degree~1 immediately implies the following stronger version of the conjecture.

\begin{theorem}\label{thm:injectivity}
    Let $\mathfrak p \trianglelefteq A$ be any maximal ideal and let $k, l \in \mathbb Z$. Then the Hecke operator $\T_{\mathfrak p}$ is injective on~$\S_{k,l}$.
\end{theorem}

\begin{proof}
    Let $\wp$ denote the monic generator of~$\mathfrak p$ and pick $x \in \bF_q$ such that $\wp(x) \neq 0$.
    By \cite[Cor.~3.3]{joshi-petrov}, the reduction $\Tmod_{\mathfrak p}$ of $\T_{\mathfrak p}$ modulo the prime $(T-x)$ is well-defined and has eigenvalue $\wp(x)^{l-1}$ with multiplicity~$\dim \S_{k,l}$. But if $\T_{\mathfrak p}$ were not injective, then 0 would be among the eigenvalues of~$\Tmod_{\mathfrak p}$.
\end{proof}

 \begin{remark}
     For elliptic modular forms, we believe the injectivity of Hecke operators at level~1 to be an interesting question. Despite a wealth of computational data, there are no known examples of Hecke operators~$\T_p$ acting on $\S_k(\operatorname{SL}_2(\mathbb Z))$ with a zero eigenvalue. Lehmer's conjecture on the Ramanujan $\tau$-function is equivalent to the statement that for $k=12$, there is no such Hecke operator. Buzzard's conjecture (mentioned in Rmk.~\ref{rmk:repeated_eigs}) would imply that Hecke operators are injective for all~$k$.
 \end{remark}

\subsection{Eigenvalues in characteristic 2}

In this section we assume that $q$ is a power of~2. We have seen in Theorem~\ref{thm:char2traces} that the traces of Hecke operators are exceptionally simple. We use this to determine many eigenvalues of the Hecke operators explicitly (Thm.~\ref{thm:char2eigs}). On the other hand, the Hecke action almost always has repeated eigenvalues (Thm.~\ref{thm:char2_repetition}). The simple expression of the traces is therefore partly due to the fact that not all eigenvalues contribute to the traces.

\begin{definition}
    Suppose $2 \mid q$. For any $k \geq 0$ and $l \in \mathbb Z$, we define the sets
    \[
    P(k+2,l,q) := \left\{ 0 \leq j < k/2 \ \bigg{|} \ j \equiv_{q-1} l-1 \text{ and } \binom{k-j}{j} \equiv_2 1 \right\},
    \]
    and write $N(k+2,l,q) := \# P(k+2,l,q)$.
\end{definition}

\begin{remark}\label{rmk:stern}
    If $q = 2$, the sequence $N(k+2,1,q)$ is closely related to the \emph{Stern--Brocot sequence} $(a_k)_{k \geq 0}$ defined as follows: $a_0 = a_1 = 1$ and for $k \geq 2$,
    \[
    a_k = \begin{cases} a_{k/2} + a_{k/2 - 1} & \text{if }k \text{ is even}; \\ a_{(k-1)/2} & \text{if } k \text{ is odd}.
    \end{cases}
    \]
    Then we have
    \[
    N(k+2,1,2) = a_k - \frac{1}{2}\left(1 + (-1)^k\right).
    \]
\end{remark}

\begin{lemma}\label{lem:dim_congruence}
    Suppose $2 \mid q$. If $k+2 \equiv 2l \pmod{q-1}$, then
    \[
    N(k+2,l,q) \equiv \dim \S_{k+2,l} \pmod{2}.
    \]
\end{lemma}

\begin{proof}
Let $x \in \bF_q$ and let $\mathfrak p = (T-x)$ be a prime of degree~1. In \cite{joshi-petrov}, the authors define, for any $\mathfrak q \neq \mathfrak p$, Hecke operators $\Tmod_{\mathfrak q}$ on Drinfeld modular forms mod~$\mathfrak p$ in such a way that
\[
\Tr(\T_{\mathfrak q} \hspace{0.2em} | \hspace{0.1em} \S_{k,l}) \equiv \Tr(\Tmod_{\mathfrak q} \hspace{0.2em} | \hspace{0.2em} \Smod_{k,l}) \pmod{\mathfrak p}.
\]
Moreover, if $Q$ is the monic generator of~$\mathfrak q$, they show that the only eigenvalue of $\Tmod_{\mathfrak q}$ on $\Smod_{k,l}$ is $Q(x)^{l-1} \in \bF^\times_q$. Since $\dim \S_{k,l} = \dim_{\bF_{\mathfrak p}} \Smod_{k,l}$, this implies that
\begin{equation}\label{eq:dimmod}
\dim \S_{k,l} \ (\text{mod }p) = \Tr(\T_{T} \hspace{0.2em} | \hspace{0.1em} \S_{k,l})  \ (\text{mod }T-1)
\end{equation}
as elements of $\bF_p \subseteq \bF_{\mathfrak p}$.

Now suppose $k+2 \equiv 2l \pmod{q-1}$. Then by Theorem~\ref{thm:deg1},
\begin{equation}\label{eq:dimmod2}
\Tr(\T_{T} \hspace{0.2em} | \hspace{0.1em} \S_{k+2,l}) \equiv \sum_{\substack{0 \leq j < k/2 \\ j \equiv l-1 \pmod{q-1}}} (-1)^j\binom{k-j}{j} \pmod{T-1}.
\end{equation}
Setting $p = 2$ and combining \eqref{eq:dimmod} and \eqref{eq:dimmod2} gives the result.
\end{proof}

\begin{remark}
    If $q=2$, Lemma~\ref{lem:dim_congruence} can also be proved combinatorially using the recursions satisfied by $\dim \S_{k+2}$ and by the Stern--Brocot sequence. If $q > 2$, the analogous recursion satisfied by the numbers $N(k+2,l,q)$ gets more involved as it involves a change in the type~$l$, which makes a direct calculation difficult.
\end{remark}

The following theorem completely describes the eigenvalues of Hecke operators in characteristic~2 which occur with algebraic multiplicity not divisible by~2.

\begin{theorem}\label{thm:char2eigs}
    Suppose $2 \mid q$. Let $k \geq 0$ and $l \in \mathbb Z$ satisfy $k+2 \equiv 2l \pmod{q-1}$. Let $\mathfrak p \trianglelefteq A$ be any maximal ideal. Then the set of eigenvalues with odd multiplicity of~$\T_{\mathfrak p}$ acting on~$\S_{k+2,l}$ is given by
    \[
    \{ \alpha \in \text{Sp}(\T_{\mathfrak p}) \ | \ m_\alpha \equiv_2 1 \} =  \{ \wp^j \ | \ j \in P(k+2,l,q) \}.
    \]
\end{theorem}

\begin{proof}
    Theorem~\ref{thm:char2traces} implies that for all $n \geq 1$,
    \[
    \Tr(\T_{\mathfrak p}^{n} \hspace{0.2em} | \hspace{0.1em} \S_{k+2,l}) = \sum_{\substack{0 \leq j < k/2 \\ j \equiv l-1 \pmod{q-1}}} {\binom{k-j}{j}}\wp^{nj} = \sum_{j \in P(k+2,l,q)} \wp^{nj}.
    \]
    Write $P(k+2,l,q) = \{j_1,\ldots,j_N\}$ with $N = N(k+2,l,q)$. It follows that for all $n \geq 1$,
    \[
    \Tr(\T_{\mathfrak p}^{n} \hspace{0.2em} | \hspace{0.1em} \S_{k+2,l}) = p_n(\wp^{j_1},\ldots,\wp^{j_N}),
    \]
    where $p_n$ denotes the $n$-th power sum symmetric polynomial in $N$ variables. 

    On the other hand, let $E = \{\alpha_1,\ldots,\alpha_e\}$ denote the set of eigenvalues with odd algebraic multiplicity, with $e = \#E$. Then for all $n \geq 1$, we also have
    \[
    \Tr(\T_{\mathfrak p}^{n} \hspace{0.2em} | \hspace{0.1em} \S_{k+2,l}) = p_n(\alpha_1,\ldots,\alpha_e).
    \]
    For $c \geq e$, write $A_c$ for the $e \times c$ Vandermonde matrix on $\{\alpha_1,\ldots,\alpha_e\}$, i.e.,
    \[
    A_c = \begin{pmatrix}
        1 & \alpha_{1} & \alpha_1^2 & \cdots & \alpha_1^{c-1} \\
        1 & \alpha_2 & \alpha_2^{2} & \cdots & \alpha_2^{c-1} \\
        \vdots & \vdots & \vdots & \ddots & \vdots \\
        1 & \alpha_e & \alpha_e^{2} & \cdots & \alpha_e^{c-1}
        \end{pmatrix}.
    \]
    Then $A_c^t A_c$ is the matrix $M_c = (\Tr(\T_{\mathfrak p}^{i+j-2}))_{i,j=1}^{c}$; the equality for $i=j=1$ follows because $\dim \S_{k+2,l} \equiv_2 e$. Hence the rank of $M_c$ is at most $e$. Since the $\alpha_i$ are distinct, the upper leftmost $e \times e$ minor $A_e^tA_e$ has non-zero determinant, so in fact $\rk(M_c) = e$.

    For $c \geq N$, write $B_c$ for the $N \times c$ Vandermonde matrix on $\{\wp^{j_1},\ldots,\wp^{j_N}\}$. By Lemma~\ref{lem:dim_congruence}, we know that $\dim \S_{k+2,l} \equiv_2 N$. Therefore it follows in the same way that $B_c^t B_c = M_c$ and that $\rk(M_c) = N$. Comparing the ranks of $M_c$ for $c \geq \max \{N,e\}$ gives $N = \# E$. But knowing this, we can in fact deduce that $E = \{ \wp^j \ | \ j \in P(k+2,l,q)\}$; in words, the elements $\wp^{j_1},\ldots,\wp^{j_N}$ are precisely the eigenvalues of~$\T_{\mathfrak p}$ with odd multiplicity. This follows because the sequence of all power sums of distinct elements $x_1,\ldots,x_m \in \bar{K}$ uniquely determines these elements \cite[Prop.~3.16]{devries_newton}.
\end{proof}

\begin{remark}
    If $\dim \S_{k+2,l} \equiv_2 N + 1$, then the matrix $B_c$ in the proof of Theorem~\ref{thm:char2eigs} would have to be replaced by the $(N+1) \times c$ Vandermonde matrix on $\{ 0, \wp^{j_1},\ldots,\wp^{j_N} \}$. Hence, Lemma~\ref{lem:dim_congruence} is equivalent to the statement that 0 occurs with even multiplicity as an eigenvalue of~$\T_{\mathfrak p}$, which we know to be true by Theorem~\ref{thm:injectivity}.
\end{remark}

\begin{remark}
    It follows immediately from Theorem~\ref{thm:char2eigs} that $N(k+2,l,q) \leq \dim \S_{k+2,l}$ if $k+2 \equiv 2l \pmod{q-1}$. The dimension of the space of cusp forms is given by Lemma~\ref{lem:cuspdim}, and implies the weaker inequality
    \[
    N(k+2,l,q) \leq \left\lceil \frac{k}{q^2-1} \right\rceil.
    \]
    This inequality is not obvious for small values of~$k$, and there are in fact plenty of counterexamples if one considers $k$ and $l$ such that $k+2 \not \equiv 2l \pmod{q-1}$. For instance, if $(k+2,l,q) = (12,1,4)$, we have $P(k+2,l,q) = \{ 0, 3 \}$ whereas $\lceil 10/15 \rceil = 1$.
\end{remark}

\begin{example}\label{ex:rep_eigs}
    Let $m \geq 1$ be an integer. By Prop.~\ref{prop:2powerweight}.1 and Thm.~\ref{thm:char2eigs}, the action of $\T_{\mathfrak p}$ on~$\S_{2^m+1,l}$ has at most one non-repeated eigenvalue, namely~1. All other eigenvalues are repeated.
\end{example}

Theorem~\ref{thm:char2eigs} implies that $\T_{\mathfrak p}$ has a repeated eigenvalue as soon as $N(k+2,l,q) < \dim \S_{k+2,l}$. We now show that this almost always happens.

\begin{theorem}\label{thm:char2_repetition}
    Suppose $2 \mid q$. Then there are only finitely many weights in which the action of~$\T_{\mathfrak p}$ has no repeated eigenvalues.
\end{theorem}

\begin{proof}
    By Lemma~\ref{lem:cuspdim}, we have $\dim \S_{k+2,l} = \Theta(k)$, i.e., the dimension grows linearly with~$k$. We will show that $N(k+2,l,q) = O(k^{\log_2(\phi)})$, where $\phi$ is the golden ratio. Since $2 > \phi$, Theorem~\ref{thm:char2eigs} then implies that some eigenvalue must be repeated if $k$ is large enough.

    Let $(a_k)_{k \geq 0}$ be the Stern--Brocot sequence, given by
    \[
    a_k = \# \left\{ 0 \leq j \leq k/2 \ \bigg{|} \ \binom{k-j}{k} \equiv_2 1 \right\};
    \]
    then certainly $N(k+2,l,q) \leq a_{k}$ for all $k$. In \cite{lehmer_stern}, it is proven that for $N \in \mathbb N_0$,
    \begin{equation}\label{eq:lehmer}
    \max \left\{ a_k \ | \ k \leq 2^N - 1 \right\} = F_{N+1},
    \end{equation}
    where $F_i$ is the $i$-th Fibonacci number (with $F_0 = 0, F_1 = 1$). Since the Fibonacci sequence is $O(\phi^k)$ with $\phi$ the golden ratio, we have
    \[
    N(k+2,l,q) \leq a_k =  O\left(\phi^{\log_2(k)}\right) = O\left(k^{\log_2(\phi)}\right),
    \]
    which is what we wanted. 
\end{proof}

\begin{remark}
    If the action of $\T_{\mathfrak p}$ on $\S_{k,l}$ has repeated eigenvalues for some~$\mathfrak p$, then it has repeated eigenvalues for all~$\mathfrak p$. However, it should be noted that the repeated eigenvalues may well depend on~$\mathfrak p$ in some non-trivial way (unlike the eigenvalues with odd multiplicity).
\end{remark}

For any given value of~$q$, the constants can be made explicit. This is demonstrated by the following theorem, which sheds light on the computations done for primes of degree $\leq 5$ in \cite[Ex.~4.6]{joshi-petrov}.

\begin{theorem}\label{thm:q2mult}
    For $q=2$, the only values of $k$ such that the Hecke action on $\S_{k}$ has no repeated eigenvalues are $3 \leq k \leq 8$ and $k \in \{10,11,12,14,20,22\}$.
\end{theorem}

\begin{proof}
    We need to show that for all $k \geq 1$, the inequality
    \[
    N(k,1,2) \leq \dim \S_{k} = \left\lceil \frac{k-2}{3} \right\rceil
    \]
    is strict for all weights $k$ except those listed in the theorem.

    By Rmk.~\ref{rmk:stern}, we have 
    \[
    N(k,1,2) = a_{k-2} - \frac{1}{2} \left(1 + (-1)^k\right) \leq a_{k-2},
    \]
    with $(a_k)_{k \geq 0}$ the Stern--Brocot sequence. By Lehmer's bound~\eqref{eq:lehmer}, if $k_0 = 2^{N-1}$ satisfies
    \[
    F_{N+1} <  \frac{2^{N-1} - 2}{3} \leq \dim \S_{k_0},
    \]
    then $a_k < \dim \S_k$ for all $k \geq k_0$.
    This inequality is achieved for $N = 8$, as $34 < 42 = 126/3$. We conclude that there is a strict inequality
    \[
    N(k,1,2) < \dim \S_{k} \text{ for all } k \geq 128,
    \]
    and checking the cases $1 \leq k \leq 127$ by hand gives the theorem.
\end{proof}

\subsection{Slopes}

A popular way to study the Hecke action is via slopes, due in part to the relationship (in the classical setting) between slopes and $p$-adic families of modular forms. We use a slightly more general definition than the one usually found in the literature by allowing different valuations~$v$.

\begin{definition}\label{def:slopes}
    Let $V$ be a $K$-vector space and let~$v$ be a valuation on $K$. Let $\T \: V \to V$ be a linear map. An element $\alpha \in \mathbb Q \cup \{\infty\}$ is a \emph{$v$-adic slope of $\T$} if there exists an eigenvalue $\lambda \in \text{Sp}(\T)$ and a valuation $w$ on~$\bar{K}$ such that $w|_K = v$ and $w(\lambda) = \alpha$.
    A slope $\alpha$ is called \emph{finite} if $\alpha \neq \infty$. The \emph{multiplicity} of a slope $\alpha$ is defined to be the integer
    \[
    d_v(\alpha) = \sum_{\substack{\lambda \in \text{Sp}(\T) \\ w(\lambda) = \alpha}} m_\lambda,
    \]
    where $w$ is a fixed valuation on $\bar{K}$ extending~$v$.
    A \emph{$v$-adic slope of weight~$k$ and type~$l$} is a $v$-adic slope of $\T_{\mathfrak p}\: \S_{k,l} \to \S_{k,l}$. In this case, we also denote the multiplicity by $d_{v}(k,l,\mathfrak p,\alpha) := d_{v}(\alpha)$.
\end{definition}

Equivalently, $\alpha$ is a slope of multiplicity $d$ if and only if the $v$-adic Newton polygon of the characteristic polynomial $\det(\mathds{1}X - \T) \in K[X]$ of $\T$ has a line segment of slope~$\alpha$ and projected length~$d$.

\begin{remark}
    Classically, the eigenvalues of the Hecke operator $\T_p$ acting on elliptic modular forms of weight~$k$ are algebraic integers of complex absolute value $p^{(k-1)/2}$. In particular, if $\ell \neq p$ is a prime number, then the $\ell$-adic slopes of~$\T_p$ are all zero. For this reason, a slope of~$\T_p$ (or the Atkin operator $\mathbf{U}_p$, if $p$ divides the level) is defined to be a $p$-adic slope. 
    
    For Hecke operators on Drinfeld modular forms, it appears that non-zero $v$-adic slopes of~$\T_{\mathfrak p}$ may occur for any place~$v$. In particular, we believe the case $v = v_\infty$ to be of great interest. One reason for this is that bounds on $\infty$-adic slopes imply bounds on $v$-adic slopes for all other places~$v$, since the Hecke operators are defined over~$A$ (see Cor.~\ref{cor:slopes_bounded}).
\end{remark}

\begin{proposition}({\cite[Prop.~4.2]{devries_newton}}) \label{prop:trace_vals}
    Suppose $\T \: V \to V$ is a $K$-linear map such that no eigenvalue of~$\T$ is repeated $p$ times. Then for any valuation $v$ on~$K$ and any $c \in \mathbb R \cup \{ \infty \}$, the following are equivalent:
    \begin{enumerate}
        \item $v(\lambda) \geq c$ for all $\lambda \in \text{Sp}(\T)$;
        \item $v(\Tr(\T^n \hspace{0.2em} | \hspace{0.2em} V) ) \geq cn$ for all $n \geq 1$.
    \end{enumerate}
\end{proposition}

In particular, under the condition that the Hecke eigenvalues have multiplicity less than~$p$, Prop.~\ref{prop:ram_bd_strict} implies bounds on the $\infty$-adic slopes. Recall that the $\infty$-adic valuation on $K$ is given by $v_\infty(f/g) = \deg(g) - \deg(f)$ for $f,g \in \bF_q[T]$, $g \neq 0$.

\begin{theorem}[Ramanujan bound for slopes]\label{thm:ram_bound_slopes} 
    Suppose that the action of~$\T_{\mathfrak p}$ on~$\S_{k,l}$ does not have $p$ repeated eigenvalues. Then any $\infty$-adic slope~$\alpha$ of~$\T_{\mathfrak p}$ acting on~$\S_{k,l}$ satisfies 
    \[
    0 \leq -\alpha < \deg(\mathfrak p) \cdot \frac{k-2}{2}.
    \]
\end{theorem}

\begin{proof}
    $\infty$-adic slopes are non-positive because the Hecke operators are defined over~$A$. The other inequality follows from Prop.~\ref{prop:trace_vals} by setting $v = v_\infty$ and applying Prop.~\ref{prop:ram_bd_strict}.
\end{proof}


\begin{corollary}\label{cor:slopes_bounded}
    Let $v \neq v_\infty$ be a valuation on~$K$. Consider the action of $\T_{\mathfrak p}$ on $\S_{k,l}$.
    Let $\lambda \in \Sp(\T_{\mathfrak p})$ be an eigenvalue which is not repeated~$p$ times, i.e., $p \nmid m_\lambda$. Then we have
    \[
    \frac{v(N_{K(\lambda)/K}(\lambda))}{[K(\lambda):K]} < \frac{\deg(\mathfrak p)}{\deg(v)} \frac{k-2}{2}.
    \]
\end{corollary}

\begin{proof}
    Let $f := N_{K(\lambda)/K}(\lambda) \in A$ be the constant term of the minimal polynomial $m(X)$ of~$\lambda$. Let $w_\infty$ be a valuation on a normal closure of~$K(\lambda)$ such that $w_\infty|_K = v_\infty$. By Theorem~\ref{thm:ram_bound_slopes}, we have
    \[
    -v_\infty(f) = -w_\infty(f) = \sum_{\lambda' \sim \lambda} -w_\infty(\lambda') < [K(\lambda):K] \deg(\mathfrak p) \frac{k-2}{2},
    \]
    where the sum runs over all roots $\lambda'$ of~$m(X)$, counted with multiplicity.
    By the product formula for valuations, we have $\deg(v)v(f) \leq -v_\infty(f)$, which concludes the proof.
\end{proof}

\subsection{Oldforms and newforms}

So far, we have focused on Drinfeld cusp forms of level 1, but there is no particular reason to restrict to this case; many results in this paper can be extended to higher level as well. There are some differences, however. For instance, to apply B\"ockle's Eichler--Shimura theory at level~$\Gamma$, one should work with Hecke operators $\T_{\mathfrak p}$ for $\mathfrak p$ not dividing the level. In exchange, one has an interesting new operator available at level~$\mathfrak p$: the \emph{Atkin operator} $\mathbf{U}_{\mathfrak p}$, which is closely related to the Hecke operator~$\T_{\mathfrak p}$ at level~1 \cite{band-val_atkin}.

Another interesting phenomenon at higher level is the decomposition of cusp forms into oldforms and newforms. Classically, there is a natural way to decompose the space of elliptic cusp forms $\S_k(\Gamma_1(Np))$ into $p$-oldforms (cusp forms which come from $\S_k(\Gamma_1(N))$) and $p$-newforms (those which do not). The Petersson inner product plays a key role in proving this decomposition, but in the Drinfeld setting, we have no such inner product at our disposal.

In~\cite{band-val_oldnew}, Bandini and Valentino propose a decomposition of $\S_{k,l}(\Gamma_0(\mathfrak p))$ into oldforms and newforms, where the oldforms are spanned by the image of certain degeneracy maps $\delta_1, \delta_{\mathfrak p} \: \S_{k,l} \to \S_{k,l}(\Gamma_0(\mathfrak p))$. They conjecture that this always yields a direct sum decomposition and prove it in some special cases. They moreover prove \cite[Cor.\ 2.10]{band-val_oldnew} that the direct sum decomposition holds if and only if $\pm \wp^{(k-2)/2}$ does not occur as an eigenvalue for the action of~$\T_{\mathfrak p}$ on~$\S_{k,l}$. Combined with our results, this can be strengthened.

\begin{theorem}\label{thm:oldnew}
    Let $k \geq 0$ and $l \in \mathbb Z$, and suppose that
    \[
    \S_{k,l}(\Gamma_0(\mathfrak p)) \neq \S^\text{new}_{k,l}(\Gamma_0(\mathfrak p)) \oplus \S^\text{old}_{k,l}(\Gamma_0(\mathfrak p)).
    \]
    Then the action of $\T_{\mathfrak p}$ on~$\S_{k,l}$ has eigenvalue $\wp^{(k-2)/2}$ or $-\wp^{(k-2)/2}$ with multiplicity a positive multiple of~$p$.
\end{theorem}

\begin{proof}
    By assumption, one of the eigenvalues $\pm \wp^{(k-2)/2}$ must occur at level~1. But if this eigenvalue is not repeated~$p$ times, this would (by Prop.~\ref{prop:trace_vals}) contradict Prop.~\ref{prop:ram_bd_strict}.
\end{proof}

\begin{corollary}\label{cor:oldnew}
If $\dim \S_{k,l} < p$, then $\S_{k,l}(\Gamma_0(\mathfrak p)) = \S^\text{new}_{k,l}(\Gamma_0(\mathfrak p)) \oplus \S^\text{old}_{k,l}(\Gamma_0(\mathfrak p))$. If $l \equiv 1 \pmod{q-1}$, then the same holds if $\dim \S_{k,l} \leq p$.
\end{corollary}

\begin{proof}
    The first part is immediate from Theorem~\ref{thm:oldnew}, and the second part follows because the action of $\T_{\mathfrak p}$ on~$\S_{k,1}$ has Hecke eigenvalue~1 with multiplicity~1 \cite[Lem.~2.4]{hattori_gouv}.
\end{proof}

\subsection{Conjectures}\label{sec:conjectures}
It is possible to compute the characteristic polynomial of an operator from its traces under the condition that no eigenvalue is repeated~$p$ times. For~$\T_{\mathfrak p}$, this condition has in practice always been fulfilled, so we have been able to compute slopes from the traces obtained via the trace formula. The resulting data contained some notable patterns. We state the most convincing of these here in the form of conjectures.

Our first conjecture is the strong Ramanujan bound, which we have already stated for traces. The variant for $\infty$-adic slopes is a priori stronger. We in fact believe that the bound holds for $v$-adic slopes, where $v$ is arbitrary.

\begin{conjecture}[Strong Ramanujan bound]\label{conj:strong_ram}
    Let $\mathfrak p \trianglelefteq \bF_q[T]$ be a non-zero prime ideal. Let $\alpha \in \mathbb Q$ be a $v$-adic slope of weight $k$ for $\T_{\mathfrak p}$. Then
    \[
    |\alpha| \leq \frac{\deg(\mathfrak p)}{\deg(v)} \cdot \frac{k - (q+1)}{2}.
    \]
\end{conjecture}

\begin{remark}
Let $v = v_\infty$ and assume there are no repeated eigenvalues. By Prop.~\ref{prop:trace_vals}, it is then equivalent to ask that for all $n \geq 1$,
\[
\deg \Tr(\T^n_{\mathfrak p} \hspace{0.2em} | \hspace{0.1em} \S_{k,l}) \leq \frac{n d (k-(q+1))}{2},
\]
where $d = \deg(\mathfrak p)$; hence Theorem~\ref{thm:strong_ram_known} provides some evidence towards the conjecture.
\end{remark}

Going further, our data suggests that the $\infty$-adic slopes attain the strong Ramanujan bound periodically, and with predictable multiplicity.

\begin{conjecture}\label{conj:slope_attainment}
    Fix $1 \leq l \leq q-1$ and let $\mathfrak p \trianglelefteq \bF_q[T]$ be a prime of degree~$d$. For $n \geq 1$, we set $k_n = (n-1)q^2 + (2l-n)q + 1$. Then the $\infty$-adic slopes of weight~$k$ and type~$l$ for~$\T_{\mathfrak p}$ attain the strong Ramanujan bound if and only if $k = k_n$ for some $n$ and $(n,l) \neq (2,1)$. Moreover, the multiplicity with which the bound is attained is given by
    \[
    d_l(n) := d_{\infty}\left(k_n,l,\mathfrak p, -\frac{d(k_n - (q+1))}{2} \right) = n - 2 \left\lceil \frac{n-l}{q+1} \right\rceil.
    \]
\end{conjecture}

\begin{remark}\hfill
\begin{enumerate}
    \item In Conj.~\ref{conj:slope_attainment}, the exception for $(n,l) = (2,1)$ is necessary: one checks that $\S_{q^2+1,1}$ is a one-dimensional space spanned by $g^qh$, which is not a double cusp form and thus has eigenvalue~$1$, which does not attain the strong Ramanujan bound. This is consistent with the fact that $d_l(n) = 0$ in this (and only this) case.
    \item Note that for $1 \leq l \leq q-1$, the space $\S_{k_1,l}$ is spanned by $g^{l-1}h^l$. According to the conjecture, this is the first eigenform to attain the strong Ramanujan bound in type~$l$, which is true for primes of degree~1 by Rmk.~\ref{rem:strongram}.
    \item A more intuitive description of~$d_l(n)$ is the following: we have $d_l(1) = 1$ for any~$l$, and for $n \geq 1$ we have
    \[
    d_l(n+1) = 
    \begin{cases} 
    d_l(n) + 1 & \text{ if } \dim \S_{k_{n+1},l} = \dim \S_{k_{n},l} + 1; \\
    d_l(n) - 1 & \text{ if } \dim \S_{k_{n+1},l} = \dim \S_{k_{n},l}.
    \end{cases}
    \]
\end{enumerate}
\end{remark}

Finally, we consider the following question: which $\alpha \in \mathbb Q$ can appear as a $v$-adic slope of~$\T_{\mathfrak p}$? For primes of degree~1 and $v \in \{v_{\mathfrak p},v_\infty\}$, the answer seems to depend on the type in a precise way.

\begin{conjecture}[The slopes which appear]
     Suppose $2 \nmid q$. Let $\mathfrak p \trianglelefteq \bF_q[T]$ be a prime of degree~1. Let $v \in \{v_{\mathfrak p},-v_{\infty}\}$, and let $\alpha \in \mathbb Q$ be a $v$-adic slope of type~$l$. Then $\alpha \in \mathbb N$ is an integer, and moreover
     \[
     \alpha \equiv l-1 \pmod{q-1} \qquad \text{or} \qquad \alpha \equiv l - 1 + \frac{q-1}{2} \pmod{q-1}.
     \]
\end{conjecture}
\section{Computations for low weights}\label{sec:comp}

In this section, we collect some consequences of \hyperref[prop:trace_formula]{the trace formula} when applied to spaces of cusp forms $\S_{k,l}$ of small dimension.

\subsection{Power eigensystems and \texorpdfstring{$A$}{A}-expansions}

Theorem~\ref{thm:a-exp_eigs} shows that the Hecke eigensystem of an eigenform with $A$-expansion is of a very particular form. It is natural to ask whether the converse is true: if $f$ is an eigenform such that $\T_{\mathfrak p}f = \wp^{n-1}f$ for all~$\mathfrak p$, does $f$ admit an $A$-expansion? 

In \cite[Ex.~2.7]{petrov_a-exp}, Petrov considers the eigenform $g^2h^2 \in \S_{12,0}$ for $q=3$. By considering its $t$-expansion, one sees that it cannot have an $A$-expansion, but Petrov's computations indicate that its Hecke eigensystem looks like that of a form with $A$-expansion. This suggests that the answer to the above question is negative. We confirm Petrov's suspicion here.

\begin{theorem}\label{thm:trivial_eigs}
    Let $\mathfrak p \trianglelefteq \bF_3[T]$ be a non-zero prime ideal with monic generator~$\wp$. Then the $\T_{\mathfrak p}$-eigenvalue of $g^2h^2$ equals~$\wp^3$.
\end{theorem}

\begin{proof}
    Since $\S_{12,0} = \langle g^2h^2 \rangle$ is one-dimensional, the $\T_{\mathfrak p}$-eigenvalue of $g^2h^2$ equals the trace of~$\T_{\mathfrak p}$ on~$\S_{12,0}$. Using \hyperref[prop:trace_formula]{the trace formula} twice, we have
    \[
    \Tr(\T_{\mathfrak p} \hspace{0.2em} | \hspace{0.1em} \S_{12,0})
    = \sum_{a,b} \# \text{Iso}_{\mathfrak p}(a,b) a^{10}b + \wp^2 \Tr(\T_{\mathfrak p} \hspace{0.2em} | \hspace{0.1em} \S_{8,0})
    = \wp^3 + \sum_{a,b} \# \text{Iso}_{\mathfrak p}(a,b) a^{10}b,
    \]
    using that $\S_{8,0} = \langle h^2 \rangle$ and $h^2$ has an $A$-expansion by Prop.~\ref{prop:a-exp_examples}.
    On the other hand, consider
    \begin{equation}\label{eq:S_14}
    \Tr(\T_{\mathfrak p} \hspace{0.2em} | \hspace{0.1em} \S_{14,1}) 
    = \sum_{a,b} \# \text{Iso}_{\mathfrak p}(a,b) \left(a^{12} + a^{10}b\wp + a^6\wp^4 \right) .
    \end{equation}
By Prop.~\ref{prop:a-exp_examples} again, $\S_{14,1} = \langle g^5h, gh^3\rangle$ is spanned by eigenforms with $A$-expansions and we have $\Tr(\T_{\mathfrak p} \hspace{0.2em} | \hspace{0.1em} \S_{14,1}) = 1 + \wp^4$ for all~$\mathfrak p$. Since $\S_{6,1} = \langle gh \rangle$ and $\S_{8,1} = \langle g^2h \rangle$, we can also deduce that
\[
\sum_{a,b} \#\text{Iso}_{\mathfrak p}(a,b)a^{12} = \bigg{(} \sum_{a,b} \# \text{Iso}_{\mathfrak p}(a,b)a^4 \bigg{)}^3 = 1, \qquad \sum_{a,b}\#\text{Iso}_{\mathfrak p}(a,b) a^6 = 1.
\]
Plugging this into \eqref{eq:S_14} gives the desired
\[
\sum_{a,b} \# \text{Iso}_{\mathfrak p}(a,b) a^{10}b = 0.\qedhere
\]
\end{proof}

\begin{remark}
    If $q=3$, the eigensystem of $gh^2 \in \S_{10,0}$ remains elusive. As already observed by Petrov in his thesis, computations suggest that 
    \[
    \Tr(\T_{\mathfrak p} \hspace{0.2em} | \hspace{0.1em} \S_{10,0}) = \wp(-T-T^3) \quad \text{for all }\mathfrak{p} = (\wp).
    \]
    Computing the trace by mimicking the proof of Theorem~\ref{thm:trivial_eigs} does not work.
    Indeed, \hyperref[prop:trace_formula]{the trace formula} gives
    \[
    \Tr(\T_{\mathfrak p} \hspace{0.2em} | \hspace{0.1em} \S_{10,0}) = \sum_{a,b} \# \text{Iso}_{\mathfrak p}(a,b)(a^8b - a^6\wp - a^2\wp^3) = -\wp - \wp^3 + \sum_{a,b} \# \text{Iso}_{\mathfrak p}(a,b)a^8b,
    \]
    but the term $\# \text{Iso}_{\mathfrak p}(a,b)a^8b$ does not appear in the trace formula for the spaces spanned by eigenforms with $A$-expansion because the corresponding binomial coefficient vanishes in each case.
\end{remark}

\subsection{Hecke eigenvalues for primes of degree~1}

In this section, we explicitly compute some traces of the Hecke operator~$\T_T$. If $\wp$ is a monic generator of $\mathfrak p$ of degree~1 or if $2 \mid q$, the results are also valid for $\T_{\mathfrak p}$ after substituting $\wp$ for~$T$.

\hyperref[prop:trace_formula]{The trace formula} leads to explicit formulae for the $\T_T$-eigenvalues on one-dimensional spaces of cusp forms. We omit the discussion of the eigenforms $g^qh^l$, since these have $A$-expansions with $A$-exponent~$l$ by Proposition~\ref{prop:a-exp_examples}.

\begin{theorem}\label{thm:deg1dim1}
Fix $1 \leq l \leq q-1$ and $0 \leq n \leq q-1$. Let $j_0 := \max\{0,1 + \lfloor (l+n-q)/2 \rfloor\}$ and $j_1 := \min \{l-1,n\}$. Then the eigenform $g^nh^l \in \S_{(n+l)q+l-n,l}$ has $\T_T$-eigenvalue
\[
\sum_{j=j_0}^{j_1} (-1)^j \binom{n+l-j-1}{j,\,n-j,\,l-j-1} T^{l-1+j(q-1)}.
\]
\end{theorem}

\begin{proof}
    Write $k+2 = (n+l)q + l - n$. Since the spaces $\S_{k+2,l}$ are 1-dimensional for the given values of $n$ and~$l$, all modular forms in the statement of the proposition are eigenforms and the trace of~$\T_{T}$ is equal to the corresponding eigenvalue. Hence we can apply Theorem~\ref{thm:deg1} to compute the eigenvalues. We obtain
    \[
    \Tr( \T_T \hspace{0.2em} | \hspace{0.1em} \S_{k+2,l} ) = \sum_{2j < n+l} c_{k,l-1+j(q-1)} T^{l-1+j(q-1)},
    \]
    where
    \[
    c_{k,l-1+j(q-1)} = (-1)^{l-1} \binom{l-1+(n+l-j)(q-1)}{l-1+j(q-1)}.
    \]
    We study the numbers $c_{k,l-1+j(q-1)}$ via \hyperref[thm:lucas]{Lucas's theorem}. The $q$-ary expansions depend on $j$:
    \begin{align*}
    l-1+(n+l-j)(q-1) &= \begin{cases} (n+l-j-1)q + (q+j-n-1) & \text{if } j \leq n; \\ (n+l-j)q + (j-n-1) & \text{if } j > n;
    \end{cases} \\
    l-1+j(q-1) &= \begin{cases} jq + l-1-j & \text{if } j<l;\\ 
    (j-1)q + q+l-1-j & \text{if } j \geq l.
    \end{cases}
    \end{align*}
Note that if $j \geq l$, then $j \leq n$ and the binomial coefficient $\binom{q+j-n-1}{q+l-1-j}$ is zero modulo~$p$. Similarly one sees that $c_{k,l-1+j(q-1)} \equiv_p 0$ when $j > n$ or when $j \leq n$ and $2j < l+n-q$. The only remaining case is $l+n-q \leq 2j \leq 2\min\{l-1,n\}$. Then we have
\begin{align*}
(-1)^{l-1} \binom{n+l-j-1}{j} \binom{q+j-n-1}{l-1-j} &=
(-1)^{l-1} \binom{n+l-j-1}{n+l-2j-1} \binom{q+j-n-1}{q+2j-n-l} \\
&\equiv_p (-1)^j \binom{n+l-j-1}{n+l-2j-1} \binom{n+l-2j-1}{n-j} \\
&= (-1)^{j} \binom{n+l-j-1}{j,\,n-j,\,l-j-1},
\end{align*}
where the congruence mod~$p$ comes from Eq.~\eqref{eq:mattarei}.
\end{proof}

\begin{example}\label{ex:t0t2}\hfill
    \begin{enumerate}
        \item Suppose $l=2$. Then we find that for $0 \leq n \leq q-1$,
        \[
        \Tr(\T_T \hspace{0.2em} | \hspace{0.1em} g^nh^2) = (n+1)T - nT^{q}.
        \]
        \item If $l = q-1$, we find that
        \[
        \Tr(\T_T \hspace{0.2em} | \hspace{0.1em} g^n\Delta) = (n+1)T^{(n+1)(q-1)-1} - nT^{n(q-1)-1}.
        \]
        \item More generally, for $2 \leq l \leq q-1$ and $0 \leq n \leq q-1$, we have
        \[
        \Tr(\T_T \hspace{0.2em} | \hspace{0.2em} g^{n}h^l ) = T^{q(l+n-q)+l-n-1} \Tr(\T_T \hspace{0.2em} | \hspace{0.2em} g^{q-1-n}h^{q+1-l}).
        \]
        This is a specific instance of symmetry, which can be proved by applying Theorem~\ref{thm:symmetry_allq_deg1} and showing that $\epsilon(T) = 0$ in these cases.
    \end{enumerate}
\end{example}

Consider now the case of type $l = 1$. It follows from Prop.~\ref{prop:a-exp_examples} that the spaces $\S_{k,1}$ are spanned by eigenforms with $A$-expansion whenever $\dim \S_{k,1} \leq 2$. In particular, the Hecke action on these spaces is completely understood. The next proposition takes us one step beyond the forms with $A$-expansion.

\begin{proposition}
    Let $k = (2q+3)(q-1)$. Then we have
    \[
    \Tr(\T_T \hspace{0.2em} | \hspace{0.1em} \S_{k+2,1}) = 1 + 2T^{2(q-1)}.
    \]
\end{proposition}

\begin{proof}
    We compute the trace using \hyperref[prop:trace_formula]{the trace formula} and \hyperref[thm:lucas]{Lucas's theorem}. The coefficient of $T^{i(q-1)}$ in $\Tr(\T_T \hspace{0.2em} | \hspace{0.1em} \S_{k+2,1})$ is
    \[
    c_i := c_{(2q+3)(q-1),i(q-1)} = \binom{(2q+3-i)(q-1)}{i(q-1)}, \qquad 0 \leq i \leq q+1.
    \]
    Thus $c_0 = 1$. Note that we have $(2q+3-i)(q-1) = q^2 + (q+1-i)q + (i-3)$ and $i(q-1) = (i-1)q + (q-i)$. Hence if $3 \leq i \leq q$, we have
    \[
    c_i = \binom{1}{0} \binom{q-(i-1)}{i-1} \binom{i-3}{q-i} = 0,
    \]
    since either $q-(i-1) < i-1$ or $i-3 < q-i$. It remains to check $i = 1,2,$ and $q+1$:
    \begin{align*}
        c_1 &= \binom{1}{0} \binom{q-1}{0} \binom{q-2}{q-1} = 0;\\
        c_2 &= \binom{1}{0} \binom{q-2}{1} \binom{q-1}{q-2} = 2;\\
        c_{q+1} &= \binom{1}{0} \binom{0}{q-1} \binom{q-2}{q-1} = 0.\qedhere
    \end{align*}
\end{proof}

It is also possible to calculate $\T_T$-eigenvalues on spaces of higher dimension if there exist enough eigenforms with $A$-expansion.

\begin{proposition}\label{prop:dim2_eig}
\hfill
\begin{enumerate}
    \item Suppose $q \neq 2$. 
    Let $k = (2q+2)(q-1)+2$ and $l = 2$, so that $\dim \S_{k+2,l} = 2$. Then the $\T_T$-action on~$\S_{k+2,l}$ has eigenvalues $T$ and~$T^q$.
    \item Let $k + 2 = 3q^2-q$ and $l=1$, so that $\dim \S_{k+2,l} = 3$. Then the $\T_T$-action on~$\S_{k+2,l}$ has eigenvalues $1$, $T^{q-1}$, and~$T^{q(q-1)}$.
\end{enumerate}
\end{proposition}

\begin{proof}
    We prove the first statement. By \cite[Thm.~1.3]{petrov_a-exp}, the space $\S_{k+2,l}$ has an eigenform with $A$-expansion and $A$-exponent~2. By Theorem~\ref{thm:a-exp_eigs}, this eigenform has $\T_T$-eigenvalue~$T$. If $f \in A$ denotes the other eigenvalue, we have
    \[
    \Tr(\T_T \hspace{0.2em} | \hspace{0.1em} \S_{k+2,l}) = T + f = \sum_{\substack{0 \leq j < k/2 \\ j \equiv 1 \pmod{q-1}}} (-1)^j \binom{k-j}{j}T^j.
    \]
    The claim follows because
    \[
    -\binom{1 + (2q+2-m)(q-1)}{1+m(q-1)} \equiv_p \begin{cases} 1 & \text{if } m \in \{ 0,1 \}; \\ 0 & \text{if } m \in \{ 2, \ldots, q \}, \end{cases}
    \]
    as can be seen by writing $1 + (2q+2-m)(q-1) = q^2 + (q-1-m)q + (q-1+m)$ and applying \hyperref[thm:lucas]{Lucas's theorem}.

    The second statement is proved similarly, noting that in this case $\S_{k+2,l}$ has two eigenforms with $A$-expansions with $A$-exponents~1, resp.~$q$.
\end{proof}

\begin{remark}
    The pairs $(k,n) = ((2q+2)(q-1)+4,q+1)$ and $(k,n) = (3q^2-q,q^2-q+1)$ do not satisfy the hypotheses of \cite[Thm.~1.3]{petrov_a-exp} for any~$q$. In other words, it is unlikely that the eigenvalues $T^q$, resp.~$T^{q(q-1)}$ computed in Prop.~\ref{prop:dim2_eig} come from eigenforms with $A$-expansions.
\end{remark}

So far, all Hecke eigenvalues we computed have been elements of~$A$. The next proposition gives quadratic $\T_T$-eigenvalues for any odd~$q$.

\begin{proposition}
    Suppose $2 \nmid q$ and let $k+2 = 2q^2-2$. Then the characteristic polynomial of $\T_T$ acting on the two-dimensional space $\S_{k+2,0}$ is irreducible and given by
    \[
    \det(\mathds{1}X-\T_T \hspace{0.2em} | \hspace{0.1em} \S_{k+2,0}) = X^2 - (3T^{2q-3} - T^{q-2})X + 2T^{q^2+3q-6} - 2T^{q^2+2q-5} + T^{3q-5}.
    \]
\end{proposition}

\begin{proof}
    By Theorem~\ref{thm:deg1} and \hyperref[thm:lucas]{Lucas's theorem}, we obtain
    \[
    \Tr(\T_T \hspace{0.2em} | \hspace{0.1em} \S_{k+2,0}) = - \sum_{i = 0}^{q-1} \binom{2q^2-q-2-i(q-1)}{q-2+i(q-1)} T^{q-2+i(q-1)} = 3T^{2q-3} - T^{q-2}.
    \]
    By Theorem~\ref{thm:deg2} and the above computation, we have
    \begin{align*}
    \Tr &(\T_T^2 \hspace{0.2em} | \hspace{0.1em} \S_{k+2,0}) = -T^{2q-4} + 3T^{4q-6} + \sum_{m=(q-1)/2}^{q-2} (-1)^{m-1+\frac{q-1}{2}} 4^{-m} \binom{m}{(q-1)/2} S_m, \\
    S_m &= \sum_{N = 0}^{2q-1} \sum_{0 \leq 2a \leq N} \binom{2q^2-3-m-a(q-1)}{m-1+a(q-1)} \binom{2(q^2-1-m-a(q-1))}{2(q-1-m)+(N-2a)(q-1)} T^{2q-4 + N(q-1)}.
    \end{align*}
    By \hyperref[thm:lucas]{Lucas's theorem}, we see that most of the binomial coefficients vanish,
    and after computing the non-zero terms we are left with
    \[
    \Tr(\T_T^2 \hspace{0.2em} | \hspace{0.1em} \S_{k+2,0}) = -4T^{q^2+3q-6} + 4T^{q^2+2q-5} + 9T^{4q-6} - 8T^{3q-5} + T^{2q-4}.
    \]
    The constant term of the characteristic polynomial is $\left(\Tr(\T_T)^2 - \Tr(\T_T)^2 \right)/2$ and the linear term is $-\Tr(\T_T)X$, which gives the claimed expression. To see that the characteristic polynomial is irreducible, we make the change of variables $Y = X - T^{2q-3}$ to obtain
    \[
    c(Y) = Y^2 - T^{q-2}(T^{q-1}-1)X + 2T^{3q-5}(T^{q-1}-1)^{q+1}.
    \]
    Suppose for contradiction that $c(Y) = (Y-f)(Y-g)$ for some $f,g \in  \bF_q[T]$. The Newton polygons of~$c$ then show that there exists a sequence of elements $(\epsilon_x)_{x \in \bF_q^\times} \in \{0,1\}^{q-1}$ such that
    \[
    f = \alpha T^{q-2} \prod_{x \in \bF_q^\times} (T-x)^{\epsilon_x + (1-\epsilon_x)q}, \qquad g = \beta T^{2q-3} \prod_{x \in \bF_q^\times} (T-x)^{(1-\epsilon_x) +\epsilon_x q},
    \]
    where $\alpha, \beta \in \bF_q^\times$. This gives
    \[
    f + g = T^{q-2}(T^{q-1}-1) \left( \alpha \prod_{\epsilon_x = 0} (T-x)^{q-1} + \beta T^{q-1} \prod_{\epsilon_x = 1} (T-x)^{q-1} \right).
    \]
    On the other hand, we have $f + g = T^{q-2}(T^{q-1}-1)$, and hence
    \[
    Q := \alpha \prod_{\epsilon_x = 0} (T-x)^{q-1} + \beta T^{q-1} \prod_{\epsilon_x = 1} (T-x)^{q-1} = 1.
    \]
    Let $N = \# \{ x \in \bF_q^\times \ | \ \epsilon_x = 0\}$. Since $\deg(Q) = 0$, we must have $N(q-1) = q-1 + (q-1-N)(q-1)$, which happens if and only if $2N = q$. Since $q$ is odd, this is a contradiction.
\end{proof}

\subsection{Hecke eigenvalues for primes of degree~2}\label{sec:comp_deg2}

Suppose that the space $\S_{k,l}$ is one-dimensional. Then Corollary~\ref{cor:deg2pol} suggests that the trace of~$\T_T^2$ behaves like the trace of~$\T_{\mathfrak p}$ where $\mathfrak p$ is a prime of degree~2. This is not exactly true, but it is up to an error term. More precisely, we have the following statement.

\begin{proposition}\label{prop:deg2pol}
    Suppose $2 \nmid q$. Let $0 \leq n \leq q-1$ and $1 \leq l \leq q-1$. Write $k = n(q-1)+l(q+1)$. Consider the polynomials (with $j_0, j_1$ as defined in Theorem~\ref{thm:deg1dim1})
    \[
    G_q(X) := \prod_{x \in \bF_q} (X - (T-x)^2), \quad h_{k,l}(X) := \left( \sum_{j=j_0}^{j_1} (-1)^j \binom{n+l-j-1}{j,\,n-j,\,l-j-1} X^{\frac{1}{2}(l-1+j(q-1))} \right)^2,
    \]
    and let $f_{k,l} \in A[X]$ be as in Corollary~\ref{cor:deg2pol}.
    Then there exists some $e_{k,l} \in A[X]$ such that
    \[
    f_{k,l}(X) = h_{k,l}(X) + G_q(X) e_{k,l}(X).
    \]
\end{proposition}

\begin{proof}
    In each case, the space $\S_{k,l}$ is one-dimensional, so for any prime~$\mathfrak p$ we have
    \[
    \Tr(\T_{\mathfrak p}^2 \hspace{0.2em} | \hspace{0.1em} \S_{k,l}) = \Tr(\T_{\mathfrak p} \hspace{0.2em} | \hspace{0.1em} \S_{k,l})^2.
    \]
    This implies that for any $Q = T-x$, $x \in \bF_q$, we have
    \[
    f_{k,l}(Q^2) = \Tr(\T_{Q} \hspace{0.2em} | \hspace{0.1em} \S_{k,l})^2.
    \]
    By Theorem~\ref{thm:deg1dim1}, $h_{k,l}(X)$ satisfies $f_{k,l}(Q^2) = h_{k,l}(Q^2)$ for all monic polynomials $Q$ of degree~1. Hence $X - Q^2$ divides $f_{k,l}(X) - h_{k,l}(X)$ for all such~$Q$, so $G_q \mid f_{k,l} - h_{k,l}$.
\end{proof}

\begin{example}
    We keep the notation of Prop.~\ref{prop:deg2pol}. When $l =1$, we have $h_{k,l} = 1$ and $e_{k,1} = 0$ for all~$k$ by the theory of $A$-expansions.
    
    For types 0 and~2, the $e_{k,l}$ seem to follow a simple pattern, as can be seen in Table~\ref{table:errors}. Note that the zeroes for $n = 0$ can be explained because $h^l$ has an $A$-expansion for $1 \leq l \leq q$. The same goes for the zeroes for $(n,l) \in \{ (3,2),(6,2)\}$ when $q=9$. For $q=3$, the zero for $n = 2$ is explained by Theorem~\ref{thm:trivial_eigs}.

    In general, it is not the case that $e_{k,l}(X)$ is a monomial: for $(q,l,n) = (5,3,2)$, we have
    \[
    e_{k,l} = 2X^3 + 4T^{10} + 2T^6 + 4T^2 = 2X^3 + T^6 - \Tr(\T_T \hspace{0.2em} | \hspace{0.1em} \S_{k,l}).
    \]

    In principle one should be able to obtain a general expression for $e_{k,l}$ since $f_{k,l}$ and $h_{k,l}$ are known, but the necessary calculations quickly get out of hand.
\begin{table}[hbt]
    \begin{center}
\begin{tabular}{||c | c ||} 
 \hline
 \rule{0pt}{12pt}
 $n$ & $e_{k,0}$ \\ [0.5ex] 
 \hline\hline
 0 & 0  \\ 
 \hline
 1 & 1 \\
 \hline
 2 & 0 \\
 \hline
\end{tabular}
\qquad 
\begin{tabular}{||c | c | c ||} 
 \hline
  \rule{0pt}{12pt}
 $n$ & $e_{k,2}$ & $e_{k,4}$ \\ [0.5ex] 
 \hline\hline
 0 & 0 & 0 \\
 \hline 
 1 & $3$ & $3X^2$  \\
 \hline
 2 & $4$ & $4X^6$ \\
 \hline
 3 & $3$ & $3X^{10}$ \\
 \hline
 4 & 0 & 0 \\
 \hline
 \end{tabular}
 \qquad 
\begin{tabular}{||c | c | c ||} 
 \hline
  \rule{0pt}{12pt}
 $n$ & $e_{k,2}$ & $e_{k,6}$ \\ [0.5ex] 
 \hline\hline
 0 & 0 & 0 \\
 \hline 
 1 & $5$ & $5X^4$  \\
 \hline
 2 & $1$ & $X^{10}$ \\
 \hline
 3 & $2$ & $2X^{16}$ \\
 \hline
 4 & $1$ & $X^{22}$ \\
 \hline
 5 & $5$ & $5X^{28}$ \\
 \hline
 6 & 0 & 0 \\
 \hline
 \end{tabular}
  \qquad 
\begin{tabular}{||c | c | c ||} 
 \hline
  \rule{0pt}{12pt}
 $n$ & $e_{k,2}$ & $e_{k,8}$ \\ [0.5ex] 
 \hline\hline
 0 & 0 & 0 \\
 \hline 
 1 & $1$ & $X^6$  \\
 \hline
 2 & 0 & 0 \\
 \hline
 3 & 0 & 0 \\
 \hline
 4 & $1$ & $X^{30}$ \\
 \hline
 5 & 0 & 0 \\
 \hline
 6 & 0 & 0 \\
 \hline
 7 & $1$ & $X^{54}$ \\
 \hline
 8 & 0 & 0 \\
 \hline
 \end{tabular}
\end{center}
\caption{The error terms $e_{k,l}$ for $q \in \{3,5,7,9\}$ and $l \in \{2,q-1\}$.}\label{table:errors}
\end{table}



\end{example}

\newpage
\appendix
\section{The number of isomorphism classes of Drinfeld modules in a fixed isogeny class}

\label{sec:appendix}

\begin{center}
\large{Jonas Bergstr\"om and Sjoerd de Vries}
\end{center}

\subsection{Hurwitz class numbers}
We begin by defining Hurwitz class numbers (alternatively called Gau{\ss} class numbers) following \cite[Sec.~6]{gekeler_frob_dist_dm}.

Let $L$ be an imaginary extension of degree two of the quotient field $K$ of $A=\mathbb F_q[T]$. Let $\mathcal O_L$ denote the integral closure of $A$ in $L$. An \emph{$A$-order} in~$L$ is a subring~$B$ of~$\mathcal O_L$ which contains~$A$ and is free of rank~2 as an $A$-module. Any $A$-order is of the form $A+f \mathcal O_L$ for some monic $f \in A$. 

A \emph{fractional ideal} of~$B$ is a non-zero finitely generated $B$-submodule of~$L$. Two fractional ideals $I,J$ are said to be \emph{equivalent} if they are related by $I=gJ$ for some $g \in L^\times$. Let $H(B)$ denote the number of equivalence classes of fractional ideals of~$B$; call $H(B)$ the \emph{Hurwitz class number} of~$B$.

Define the Kronecker symbol~$\chi_L$ via
\begin{align*}
    \chi_L(\mathfrak p) = \begin{cases}
        1 & \text{if } \mathfrak p \text{ splits in } L; \\
        0 & \text{if }\mathfrak p \text{ ramifies in } L;\\
        -1 & \text{if }\mathfrak p \text{ is inert in } L,
    \end{cases}
\end{align*}
for any prime $\mathfrak p$ of~$K$. If $\mathfrak p = (\wp) \subset A$, we also write $\chi_L(\wp)$ for the Kronecker symbol of~$\mathfrak p$.

For any monic $f\in A$, let $\mathbb P(f)$ denote the set of monic prime divisors of~$f$ and put $B_f=A+f\mathcal O_L$. 
For any $f \in A$, we then have
\begin{equation} \label{eq:hurwitz}
H(B_f)=\sum_{g |f} \frac{H(\mathcal O_L)\cdot q^{\deg(g)}}{[\mathcal O_L^\times:B^\times_g]}\prod_{\wp \in \mathbb P(g)} (1 - \chi_L( \wp )q^{-\deg(\wp)}), 
\end{equation}
where the sum is over monic $g$ in~$A$, see \cite[Sec.~6]{gekeler_frob_dist_dm}. 

Equation~\eqref{eq:hurwitz} simplifies modulo~$p$ to the following.

\begin{lemma}\label{lem:hurwitz_number}
    For any monic $f \in A$,
    \[
    H(B_f) \equiv_p H(\mathcal O_L)\prod_{\wp \in \mathbb P(f)} (1 - \chi_L (\wp)).
    \]
\end{lemma}
\begin{proof} Since $L$ is a quadratic extension, any $A$-order has unit group $\bF_q^\times$ or $\bF_{q^2}^\times$. This tells us that $[\mathcal O_L^\times : B_g^\times] \in \{1,q+1\}$ and so $[\mathcal O_L^\times : B_g^\times] \equiv_p 1$. 
If a monic $g \in A$ is not square-free, then 
\[
\frac{H(\mathcal O_L)\cdot q^{\deg(g)}}{[\mathcal O_L^\times :B^\times_g]}\prod_{\wp \in \mathbb P(g)} (1 - \chi_L(\wp)q^{-\deg(\wp)}) \equiv_p 0. 
\]  
Since there is a $1$-to-$1$ correspondence between subsets of~$\mathbb P(f)$ and monic square-free divisors of~$f$, we find from Equation~\eqref{eq:hurwitz} that
\[
H(B_f) \equiv_p H(\mathcal O_L) \sum_{S \subseteq \mathbb P(f)} \prod_{\wp \in S} -\chi_L(\wp) = H(\mathcal O_L) \prod_{\wp \in \mathbb P(f)} (1 - \chi_L(\wp)),
\]
as claimed. 
\end{proof}

\subsection{Isomorphism classes}
The following result gives a formula for the number of isomorphism classes of Drinfeld modules over a finite field in a fixed isogeny class. In the cases of a commutative endomorphism algebra, this number can be expressed in terms of Hurwitz class numbers. 

\begin{proposition}\label{prop:isom}
    Given a characteristic polynomial as in Prop.~\ref{prop:yu_pols}, we have
    \[
    \# \text{Iso}_{\mathfrak p^n}(a,b) = 
    \begin{cases}
        H(A[\pi]) & \text{in case 1;} \\
        H(A[\sqrt{-b \wp}]) & \text{in case 2;}\\
        2 & \text{in case 3;}\\
        (q^{\deg (\mathfrak p)}-1)/(q-1) & \text{in case 4.}
    \end{cases}
    \]
\end{proposition}

\begin{remark}
    For $q$ odd, this result follows from \cite[Prop.~7]{yu_isogs}. We note that each isomorphism class of Drinfeld modules in loc.\ cit.\ is weighted by the factor $q-1$ divided by the number of automorphisms of any representative of this isomorphism class. For $q$ even and $n=1$, this result follows from \cite[Prop.~6.8]{gekeler_frob_dist_dm}.
\end{remark}

\begin{proof} 
The endomorphism algebra of any Drinfeld module in an isogeny class in case 1, 2, 3 will be commutative, and more precisely an imaginary extension $L$ of~$K$ of degree two. 
By \cite[Prop.~5.1]{KKP} there is a Drinfeld module with endomorphism ring $B$ precisely if $B$ is a ring such that $A[\pi] \subset B \subset \mathcal O_L$ and $B$ is locally maximal at $\pi$. In cases 2 and 3 we see directly that only $B=\mathcal O_L$ with $A[\pi] \subset B \subset \mathcal O_L$ is locally maximal at $\pi$, and in case 1, by \cite[Cor.~2.9]{KKP}, every $A[\pi] \subset B \subset \mathcal O_L$ is locally maximal at $\pi$. 

It follows from an upcoming erratum to Theorem 5.4 of~\cite{KKP} that $\# \text{Iso}_{\mathfrak p^n}(a,b) = H(A[\pi])$ in case~1, that $\# \text{Iso}_{\mathfrak p^n}(a,b) = H(\mathcal O_L)$ in case~2 and that $\# \text{Iso}_{\mathfrak p^n}(a,b)=2$  in case~3 (which in case~3 is due to the fact that there are two  Frobenius-stable lattices inside the Dieudonn\'e module at $\mathfrak p$).

For case~4, let $\Sigma(2,\mathfrak p)$ denote the set of $\overline{\bF}_q$-isomorphism classes of supersingular Drinfeld modules of rank~2 with characteristic~$\mathfrak p$. Then we have (see \cite[Ex.~4.4]{gekeler_finite-dms} or \cite[Cor.~4.4.12]{papikian}),
\[
\#\Sigma(2,\mathfrak p) = \begin{cases}
    (q^{\deg(\mathfrak p)} - 1)/(q^2-1) & \text{if } {\deg(\mathfrak p)} \equiv_2 0; \\
    (q^{\deg(\mathfrak p)} - 1)/(q^2 - 1) + \frac{q}{q+1} &\text{if } {\deg(\mathfrak p)} \equiv_2 1.
\end{cases}
\]
On the other hand, since $n$ is even, every supersingular Drinfeld module is defined over $\bF_{\mathfrak p^n}$ \cite[Prop.~4.2]{gekeler_finite-dms}. Therefore we also have
\[
\# \Sigma(2,\mathfrak p) = \sum_{[\phi]/\bF_{\mathfrak p^n}} \frac{1}{\# \Aut(\phi)},
\]
where the sum is taken over the isomorphism classes of supersingular Drinfeld modules over $\bF_{\mathfrak p^n}$. If $\deg(\mathfrak p) \equiv_2 1$, the supersingular Drinfeld modules in case~3 of Prop.~\ref{prop:yu_pols} contribute 
\[
2\cdot \frac{q^2-q}{2}\cdot \frac{1}{q^2-1}=\frac{q}{q+1}
\]
to $\# \Sigma(2,\mathfrak p)$. The only other contribution comes from the supersingular Drinfeld modules $\phi$ of case~4. These come in $q-1$ isogeny classes, each of which contains the same number $N$ of isomorphism classes by Remark~\ref{rem:iso_c}. Moreover, any Drinfeld module in case~4 has automorphism group $\bF_{q^2}^ \times$. Putting this together gives
\[
\frac{q-1}{q^2-1} N = \frac{q^{\deg(\mathfrak p)}-1}{q^2-1}.\qedhere
\]
\end{proof}

\begin{remark}
   If $q$ is even then $L/K$ is inseparable in case~2. Moreover, $A[\sqrt{-b \wp}]=\mathcal O_L$ and $H(A[\sqrt{-b \wp}])=1$ by Lemma~\ref{lem:jac}.
\end{remark}

\subsection{Hyperelliptic curves} \label{sec:hypers}

In this subsection, we describe a way to make the computation of the Hurwitz class number of~$\mathcal O_L$ explicit in terms of hyperelliptic curves (and their Jacobians). 

Fix a monic irreducible polynomial $c(X)$ of degree two with coefficients in~$A$.
Denote its splitting field over~$K$ by~$L$. 
Let $\pi$ be a fixed root of $c(X)$ in~$L$.

Let us first assume $q$ to be odd. 
Factor the discriminant of~$c(X)$ into a product $Df^2$ with $D$ square-free and $f$ monic. The extension $L/K$ is then the splitting field of $\tilde c(X,T)=X^2-D$. Put $g = \lceil \deg(D)/2-1 \rceil$. 

Secondly, let us assume $q$ to be even. Say that $c(X)=c_1(X)=X^2+r_1X+s_1$. If $r_1=0$, then $L/K$ is inseparable, and all places of $K$ ramify in~$L$. 
Assume therefore that $r_1 \neq 0$. Let $g_1$ be the integer such that $2g_1+1 \leq \max(2\deg(r_1),\deg(s_1)) \leq
2g_1+2$. Put 
$f_1=\gcd(r_1,(s_1')^2+s_1(r_1')^2)$, 
$\hat f_1=\prod_{\wp \in \mathbb P(f_1)}\wp$ and $m=\deg (\hat f_1)$.
If $\deg(f_1) >0$, then let $l$ be the remainder of $s_1^{q^{m}/2}$ divided by $r_1$. 
Finally put $r_2=r_1/ \hat f_1$ and $s_2=(s_1+r_1l+l^2)/\hat f_1^2$. Then $r_2$ and $s_2$ will be in $A$, see \cite[Lemma 9.3]{Bhypers}. 
Let $g_2$ be the integer such that $2g_2+1 \leq \max(2\deg(r_2),\deg(s_2)) \leq
2g_2+2$ and note that $g_2<g_1$. 
Repeat this process with $c_2(X)=X^2+r_2X+s_2$, continuing until $f_k=\gcd(r_k,(s_k')^2+s_k(r_k')^2)$ has degree~$0$, for some $k \geq 1$.
Then put $f=\prod_{i=1}^k \hat f_i$, $\tilde c(X)=\tilde c(X,T)=X^2+r_kX+s_k$ and $g=g_k$. Note that $r_1 = fr_k$.

Let again $q$ be arbitrary, and $L/K$ separable. If $g \geq 0$, then $\tilde c(X,T)$ is an affine equation for the (geometrically irreducible, projective and non-singular) curve $C_L$ over $\mathbb F_q$ of genus $g$ that comes with a degree two cover of the projective line. So if $g \geq 2$ then $C_L$ will be a hyperelliptic curve. 

Let $\tilde \pi$ be a root of~$\tilde c(X)$. Then $A[\tilde \pi]=\mathcal O_L$, and by the Jacobian criterion for smoothness of $c(X)$ (see for instance \cite[Sec.~8]{Bhypers} for the characteristic two case) we see that $A[\pi]=B_f$. In particular, $A[\pi] = \mathcal O_L \iff f=1 \iff r_1 = r_k$.
For $\wp \in \mathbb P(f)$, we find that 
\begin{align*}
    \chi_L(\wp) = \begin{cases}
        1 & \text{if } \tilde c(X,T)  \text{ has two distinct linear factors in } (A/(\wp))[X]; \\
        0 & \text{if } \tilde c(X,T)  \text{ is a square in } (A/(\wp))[X];\\
        -1 &  \text{if } \tilde c(X,T)  \text{ is irreducible in } (A/(\wp))[X].
    \end{cases}
\end{align*}

If we put $\tilde c_\infty(Y,s) := s^{2g+2} \tilde c(s^{-(g+1)} Y,1/s)$, then $\tilde c_\infty(Y,s)$ lies in $\bF_q[s,Y]$. Thus $\tilde c_\infty(Y,0) \in \bF_q[Y]$, and the Kronecker symbol of the prime $\infty = (1/T)$ of~$K$ is given by
    \[
    \chi_L(\infty) =
    \begin{cases} 
    1 & \text{if }\tilde c_\infty(Y,0) \text{ has two distinct factors in } \bF_q[Y];\\ 
    0 & \text{if }\tilde c_\infty(Y,0) \text{ is a square in } \bF_q[Y];\\ 
    -1 & \text{if } \tilde c_\infty(Y,0) \text{ is irreducible in } \bF_q[Y].
    \end{cases}
    \]
Note that the extension $L/K$ is imaginary if and only if $\chi_L(\infty) \neq 1$.

\begin{lemma}\label{lem:jac} Let $L/K$ be an imaginary quadratic extension. 
\begin{enumerate}
    \item If $L/K$ is inseparable, then $H(\mathcal O_L) = 1$.
    \item If $L/K$ is separable and $g=-1$, then $H(\mathcal O_L) = 1$. 
    \item If $L/K$ is separable and $g \geq 0$,
then
    \[
    H(\mathcal O_{L}) = (1 - \chi_L(\infty))\#J_L(\bF_q),
    \]
    where $J_L$ denotes the Jacobian of $C_L$. 
\end{enumerate}
\end{lemma}

\begin{proof} 
    If $L/K$ is inseparable, then $q$ is even and~$L = K(\sqrt{T})$, so the map $L \to K$ induced by $\sqrt{T} \mapsto T$ gives an isomorphism $L \cong K$ of abstract fields. Since $\mathcal O_L \cong A$ under this isomorphism, we have $H(\mathcal O_L)=1$.
    
    If $L/K$ is separable and $g=-1$, then $L \cong \mathbb F_{q^2}(T)$ and $\mathcal O_L \cong \mathbb F_{q^2}[T]$, so $H(\mathcal O_L)=1$.
    
    Finally, assume that $L/K$ is separable and that $g \geq 0$.
    The divisor class group of~$L$ will then be isomorphic to $J_L(\mathbb F_q)$.    
    The Hurwitz class number $H(\mathcal O_L)$ is equal to the $S$-class number of~$L$, with $S$ the set containing only the place above~$\infty$.
    If $\infty$ is ramified in~$L$ then $H(\mathcal O_L)$ is equal to the divisor class number of~$L$, and if $\infty$ is inert in~$L$ then $H(\mathcal O_L)$ is twice the divisor class number, see \cite[Prop.~14.1]{rosen} or \cite[Exerc.~5.10]{sticht}. 
\end{proof}

\begin{remark}\label{rmk:point_counts}
The number of $\bF_q$-points on the Jacobian of a curve can be computed in terms of the number of points on the curve as follows. If $\psi_L(z)$ is the characteristic polynomial of the geometric Frobenius acting on the first $\ell$-adic \'etale cohomology of $C_L \otimes_{\mathbb F_q} \overline{\mathbb F}_q$, then $\# J_L(\mathbb F_q)=\psi_L(1)$. Moreover, if $e_i$ denotes the $i$-th elementary symmetric polynomial and $p_i$ the $i$-th power sum polynomial, then there are rational numbers $r_{i,\lambda}$ for each partition $\lambda=(\lambda_1,\ldots,\lambda_i)$ of the integer $i \geq 1$, such that
\[
e_i=\sum_{\lambda \vdash i} r_{i,\lambda} \prod_{j=1}^i p_j^{\lambda_j}
\]
and 
\[
\psi_L(z)=z^{2g}+q^g+(-1)^g \left( \sum_{\lambda \vdash g} r_{g,\lambda} \prod_{j=1}^g a_j^{\lambda_j} \right) z^{g} +\sum_{i=1}^{g-1} (-1)^i \left( \sum_{\lambda \vdash i} r_{i,\lambda} \prod_{j=1}^i a_j^{\lambda_j} \right) (z^{2g-i}+q^iz^i)
\]
with $a_j=q^j+1-\# C_L(\mathbb F_{q^j})$ for $j=1,\ldots,g$.
\end{remark}

Finally, if~$q$ is even, let us classify when $H(A[\pi])$ is even. This result is used in Section~\ref{sec:char2}.

\begin{proposition} \label{prop:evenclassnumber}
    Let $q$ be even and put $c(X)=X^2+rX+s$. Assume that $L/K$ is imaginary. Then $H(A[\pi])$ is even if and only if $\deg(r)>0$.
\end{proposition}

\begin{proof} 
By Lemma~\ref{lem:hurwitz_number}, $H(A[\pi])$ is even if $H(\mathcal O_L)$ is even. Let us first consider the parity of~$H(\mathcal O_L)$.

If $L/K$ is inseparable, which happens only if $r=0$, then $H(\mathcal O_L)=1$ by Lemma~\ref{lem:jac}.1. 
If $g=-1$, then $\deg(r) = 0$ and $H(\mathcal O_L)=1$ by Lemma~\ref{lem:jac}.2. 

Assume now that $g \geq 0$ and $L/K$ is separable.
By Lemma~\ref{lem:jac}.3, we know that $H(\mathcal O_L)$ is even if $\infty$ is inert. We also find that if $\infty$ is inert then $\deg(r)\geq \deg(\tilde r)>0$.
Assume therefore that $\infty$ is ramified. This means that $\deg(\tilde{r}) \leq g$ and the coefficient of the degree $2g+2$ term in $\tilde{s}$ is a square, say $b^2$. By sending $X$ to $X + bT^{g+1}$ we may therefore assume that $\deg(\tilde{s}) = 2g+1$.

Since any $\bF_q$-point of the Jacobian $J_L$ of $C_L$ is a torsion point and $J_L(\bF_q)$ is finite, we have that 
\[
J_L(\bF_q) = \prod_{\ell} J_L[\ell^N](\bF_q)
\]
for some $N \geq 0$, where the product is over all prime numbers $\ell$. 
Moreover, $J_L[\ell^N](\bF_q)$ is a finite $\ell$-group, so $\# J_L(\bF_q)$ is even if and only if $\# J_L[2^N](\bF_q)$ is even. 
This is equivalent to $J_L[2^N](\bF_q)$ being non-trivial which is in turn equivalent to $J_L[2](\bF_q)$ being non-trivial.

Let $P_{\infty}$ denote the unique point of~$C_L(\mathbb F_q)$ over ~$\infty$. 
Let $\tilde c(X,T)$ be the affine equation for~$C_L$ as above. 
For any point $P_{(\alpha,\beta)}$ of $C_L(\overline{\mathbb F}_q)\setminus \{P_\infty\}$ with $T$-coordinate $\alpha$ and $X$-coordinate $\beta$, we have that $[P]+[P_{(\alpha,\beta+\tilde r(\alpha))}]-2[P_{\infty}]=0 \in J_L[2](\overline{\mathbb F}_q)$. 
Hence, if $\tilde r(\alpha)=0$ then $\beta$ is the unique square root of $\tilde s(\alpha)$ and the divisor $[P_{(\alpha,\sqrt{\tilde s(\alpha)})}]-[P_{\infty}]$ is in $J_L[2](\overline{\mathbb F}_q)$. 
Thus, for any $h \in \mathbb P(\tilde r)$ we get a divisor 
\[
D_{h}=\sum_{\substack{\alpha \in \overline{\bF}_q \\ h(\alpha)=0}} [P_{(\alpha,\sqrt{\tilde s(\alpha)})}]-\deg(h)[P_{\infty}] \in J_L[2](\mathbb F_q) \subset J_L[2](\overline{\bF}_q).
\]
For any non-empty $S \subset \mathbb P(\tilde r)$, $D_S=\sum_{h\in S} D_h$ is a reduced divisor, and hence will be non-zero (see \cite[Thm.~4.145]{cohen-frey}). 
Moreover, for any non-empty distinct $S,T \subset \mathbb P(\tilde r)$, $D_{S}+D_{T}=D_{S \cup T \setminus S\cap T} \neq 0$. 
It then follows from the Deuring--Shafarevich formula (see for instance \cite[Thm.~1.2]{Shiomi}) that the $2^{\# \mathbb P( \tilde r)}$ elements $D_S$ for $S \subset \mathbb P(\tilde r)$ are precisely the elements of $J_L[2](\mathbb F_q)$. 
We conclude that $H(\mathcal O_L)$ will be odd precisely if $\deg(\tilde r)=0$ and $\deg(\tilde s)>0$. 

Finally, let $A[\pi]=B_f$ for some $f \in A$. For any $\wp \in \mathbb P(f)$, $1 - \chi_L(\wp)$ is odd precisely if $\wp$ ramifies in $L$. If $L$ is inseparable then this happens for all $\wp \in \mathbb P(f)$, so in this case it follows from Lemma~\ref{lem:hurwitz_number} and the above that $H(A[\pi])$ is odd. If $\deg(\tilde r)=0$ then no $\wp \in \mathbb P(f)$ ramifies in~$L$. We conclude that if $\deg( \tilde r)=0$ then, by Lemma~\ref{lem:hurwitz_number} and the above, $H(A[\pi])$ is odd precisely if $f=1$, that is, if $\deg(r)=0$. Finally, if $\deg(\tilde{r}) > 0$, then $\deg(r) > 0$ and by the above $H(\mathcal O_L)$ is even, so we are done. 
\end{proof}

\begingroup
\sloppy
 \printbibliography
\endgroup

\end{document}